\documentclass[10pt]{amsart}

\usepackage{amsmath,amssymb,amsthm,mathrsfs}
\usepackage{tikz-cd}
\usepackage[all]{xy}
\usepackage{hyperref}
\usepackage{amsthm}
\usepackage{caption}
\usetikzlibrary{calc}
\usepackage{mathtools}  
\usepackage{scalerel}

\newtheorem{theorem}{Theorem}[section]
\newtheorem{proposition}[theorem]{Proposition}
\newtheorem{lemma}[theorem]{Lemma}
\newtheorem{corollary}[theorem]{Corollary}
\theoremstyle{definition}
\newtheorem{definition}[theorem]{Definition}
\newtheorem{remark}[theorem]{Remark}
\newtheorem{example}[theorem]{Example}
\newtheorem{claim}[theorem]{Claim}

\makeatletter
\let\c@equation\c@theorem   
\makeatother

\usepackage[T1]{fontenc}
\usepackage[utf8]{inputenc}
\usepackage[english]{babel}

\usepackage{amsmath,amssymb,amsfonts,amscd}
\usepackage{mathrsfs}
\usepackage{enumitem}
\usepackage{hyperref}
\usepackage{tikz-cd}
\usepackage{subcaption}

\usepackage{caption}
\usepackage{xcolor} 
\definecolor{DarkRed}{RGB}{173,0,0}
\definecolor{LightRed}{RGB}{201,0,0}

\hypersetup{
  colorlinks=true,
  linkcolor=DarkRed,
  citecolor=LightRed,
  urlcolor=LightRed
}

\newcommand{\C}{\mathbb{C}}
\newcommand{\Pbb}{\mathbb{P}}
\newcommand{\A}{\mathbb{A}}
\newcommand{\Q}{\mathbb{Q}}
\newcommand{\R}{\mathbb{R}}
\newcommand{\Z}{\mathbb{Z}}
\newcommand{\Mbar}{\overline{M}}
\newcommand{\OO}{\mathcal{O}}
\newcommand{\Pic}{\operatorname{Pic}}
\newcommand{\Eff}{\operatorname{Eff}}
\newcommand{\Nef}{\operatorname{Nef}}
\newcommand{\Mov}{\operatorname{Mov}}
\newcommand{\PsAut}{\operatorname{PsAut}}
\newcommand{\Aut}{\operatorname{Aut}}
\newcommand{\LG}{\operatorname{LG}}
\newcommand{\OG}{\operatorname{OG}}

\newcommand{\Bl}{\operatorname{Bl}}
\newcommand{\CQ}{\mathsf{CQ}} 
\newcommand{\CS}{\mathsf{CS}}

\DeclareMathOperator{\Cox}{Cox}
\DeclareMathOperator{\Sym}{Sym}
\DeclareMathOperator{\rk}{rk}
\DeclareMathOperator{\Cl}{Cl}
\DeclareMathOperator{\Div}{Div}
\DeclareMathOperator{\Exc}{Exc}
\DeclareMathOperator{\Supp}{Supp}
\DeclareMathOperator{\mult}{mult}
\DeclareMathOperator{\Ker}{Ker}
\DeclareMathOperator{\Osc}{Osc}

\DeclareMathOperator{\NE}{NE}
\DeclareMathOperator{\Sp}{Sp}
\DeclareMathOperator{\PSp}{PSp}
\DeclareMathOperator{\PO}{PO}

\DeclareMathOperator{\PSO}{PSO}

\DeclareMathOperator{\contr}{contr}
\newcommand{\ev}{\operatorname{ev}}
\newcommand{\TL}{\mathcal{TL}}
\newcommand{\TO}{\mathcal{TO}}

\newcommand{\QQ}{\mathbb{Q}} 
\newcommand{\ZZ}{\mathbb{Z}}

\newcommand{\CC}{\mathsf{CC}}   

\DeclareMathOperator{\lcm}{lcm}
\DeclareMathOperator{\unb}{unb}

\newcommand{\GL}{\operatorname{GL}}
\newcommand{\Gm}{\mathbb{G}_m}

\newcommand{\PP}{\mathbb{P}}

\newcommand{\KL}{\mathrm{KL}} 
\newcommand{\KO}{\mathrm{KO}} 
\newcommand{\RL}{\mathrm{R}}  
\newcommand{\RO}{\mathrm{R}} 
\newcommand{\PL}{\mathrm{P}}

\newcommand{\cone}{\operatorname{Cone}}

\theoremstyle{plain} 

\newtheorem{theoremABC}{Theorem}



\newtheorem{corollaryABC}[theoremABC]{Corollary}

\newcommand\tikzmark[1]{%
  \tikz[remember picture,overlay]\coordinate (#1);}

\newcommand{\underbracedmatrixll}[2]{%
  \left(\;\hspace{-.27in}
  \smash[b]{\underbrace{
    \begin{matrix}#1\end{matrix}
  }_{#2}}
  \;\right.
  \vphantom{\underbrace{\begin{matrix}#1\end{matrix}}_{#2}}
}

\newcommand{\underbracedmatrixrr}[2]{%
  \left. \;
  \smash[b]{\underbrace{
    \begin{matrix}#1\end{matrix}
  }_{#2}}
  \;\hspace{-.32in}\right)
  \vphantom{\underbrace{\begin{matrix}#1\end{matrix}}_{#2}}
}

\usepackage{tikz}

\newcommand{\vdashline}{\tikz[baseline=-0.5ex]\draw[dashed] (0,-0.8ex) -- (0,1.2ex);}

\begin{document}
\title
[Compactifications of spaces of symmetric matrices and pointed Kontsevich spaces]
{Compactifications of spaces of symmetric matrices and pointed Kontsevich spaces of isotropic Grassmannians}

\author[Hanlong Fang]{Hanlong Fang}
\address{{\sc Hanlong Fang}\\ School of Mathematical Sciences, Peking University, Beijing 100871, China}
\email{hlfang@pku.edu.cn}

\author[Alex Massarenti]{Alex Massarenti}
\address{\sc Alex Massarenti\\ Department of Mathematics and Computer Science, University of Ferrara, Via Machiavelli 30, 44121 Ferrara, Italy}
\email{msslxa@unife.it}

\author[Xian Wu]{Xian Wu}
\address{{\sc Xian Wu}\\ School of Mathematical Sciences, Peking University, Beijing 100871, China}
\email{xianwu.ag@gmail.com}

\date{\today}

\begin{abstract}
We study two closely related families of varieties arising from genus $0$ stable maps to the Lagrangian Grassmannian $\LG(n,2n)$. First, we construct the Kausz--type compactification $\TL_n$ of the space of symmetric matrices and give an explicit description of its birational geometry. Second, we realize $\TL_n$ as a general evaluation fiber in a Kontsevich space, and then exploit this modular interpretation to derive consequences for the birational geometry of the space of pointed conics $\overline{M}_{0,1}(\LG(n,2n),2)$. Analogous compactifications related to orthogonal Grassmannians are also presented.
\end{abstract}

\maketitle
\setcounter{tocdepth}{1}
\tableofcontents

\maketitle
\setcounter{tocdepth}{1}
\tableofcontents

\section*{Introduction}

Kontsevich moduli spaces of stable maps to homogeneous varieties are fundamental  objects in modern algebraic geometry: they encode Gromov--Witten theory while simultaneously providing a rich testing ground for the minimal model program and the birational classification of higher--dimensional varieties.
Recently, there has been significant interest in the cohomological properties of Kontsevich spaces of rational curves with homogeneous targets. Pandharipande \cite{P97} computed the canonical classes of the moduli space
$\Mbar_{0,m}(\mathbb P^n,d)$ of genus $0$ degree $d$ stable maps with $m$ marked points and determined their Picard groups \cite{P99}. Getzler and Pandharipande \cite{GP} established an algorithm to compute their Betti numbers. Furthermore, Oprea \cite{OP06} showed that the rational cohomology groups of $\Mbar_{0,m}(X,\beta)$ are generated by tautological classes for any flag variety of Type A. Musta\c{t}\u{a}--Musta\c{t}\u{a} 
\cite{MM07,MM08,MM10} described the rational Chow rings of $\Mbar_{0,m}(\mathbb P^n,d)$.

Coskun and Starr \cite{CS06} initiated a systematic study of divisor classes on the unpointed Kontsevich moduli spaces of rational curves in Grassmannians. Coskun, Harris and Starr \cite{CHS08,CHS09} described the divisor cones of several Kontsevich spaces of low degree via explicit contractions and flips. Chen \cite{Ch08} carried out a Mori--theoretic analysis of
$\Mbar_{0,0}(\mathbb P^3,3)$.  Chen and Coskun \cite{CC10} fully described the birational geometry of spaces of conics and cubics in Grassmannians. Their description has been taken up and further developed by Chung and Moon \cite{CM17}, and related to the theory of complete quadrics by Lozano Huerta \cite{LoH}. In the isotropic setting relevant here, the unpointed conics in $\LG(n,2n)$ and its divisor theory were analyzed in \cite{CM}.

Despite the central role that marked points play in Gromov--Witten theory, relatively few works offer an explicit and uniform birational description of pointed Kontsevich spaces that would allow for a systematic derivation of their algebro-geometric and complex geometric invariants--such as Fano thresholds, Fano indices, and automorphism groups. 
One notable distinction from the unpointed setting is that unpointed Kontsevich spaces often admit explicit birational transformations to various alternative compactifications, including those arising from GIT quotients, Hilbert schemes, Chow varieties, or specific spherical varieties such as the De Concini--Procesi wonderful compactifications \cite{KM10,CC10,CC11,LoH,CM17,CM,CCM}. These compactifications serve as powerful tools for studying divisor theory. In contrast, the situation for pointed ones is less transparent; for instance, $\overline{M}_{0,1}(\LG(2,4),2)$ is not a spherical variety under natural group actions.

The theory of wonderful compactifications constitutes a cornerstone in the study of homogeneous varieties with various applications in algebraic geometry, representation theory, algebraic statistics \cite{Fa,Sp1,Lu2,DKKS,MMM}. Nevertheless, the framework fails for homogeneous spaces whose automorphism groups possess a nontrivial center.  The development of a more general framework is still an enticing open question, while a major step forward has been taken by Kausz \cite{Ka}, who constructed a novel compactification for general linear groups. Recently, \cite{FW} realized the Kausz compactifications as blow--ups of Grassmannians $G(n,2n)$ and interpreted them as universal families of the $\mathbb C^*$ Hilbert quotients of  Grassmannians. Thaddeus \cite{Th1} proved that the $\mathbb C^*$ Hilbert quotients of Grassmannians, Lagrangian Grassmannians, and orthogonal Grassmannians are isomorphic, respectively, to the spaces of complete collineations, complete quadrics, and complete skew--forms. This suggests that Kausz--type compactifications may exist in the isotropic setting, work that has been fully carried out for Grassmannians in \cite{FW}.  

In this paper, we investigate the algebro-geometric and complex geometric properties of pointed Kontsevich spaces targeting Lagrangian Grassmannians $\LG(n,2n)$, by constructing an analogous Kausz--type compactification in the symplectic setting, which we denote by $\TL_n$ and realize as blow--ups of $\LG(n,2n)$. 
Then, we derive a modular interpretation of $\TL_n$ as a general evaluation fiber inside a suitable pointed Kontsevich space $\overline{M}_{0,k}(\LG(n,2n),d)$. A key structural feature is that $\TL_n$ is a toroidal spherical variety for the action of $\Sp_{2n}$, placing it within a well--developed framework in which many birational invariants admit combinatorial descriptions \cite{BrionIntro,BrionSpherical,Br93}. In our setting, this structure can be made completely explicit. Finally, we obtain concrete applications to the divisor theory of pointed conics in $\LG(n,2n)$.
\smallskip

Let us be more precise. Throughout, we work over the complex number field. The following theorem provides a uniform, computable description of the birational geometry of $\TL_n$, which is a synthesis of Theorems~\ref{thm:TLn-spherical}, \ref{effcone}, \ref{thm:Fano-weakFano}, \ref{thm:rigidity-absolute}, \ref{thm:PsAut-equals-Aut}, Propositions~\ref{modulil}, \ref{prop:nef-cone}, \ref{prop:Cox-TLn}, and Corollary~\ref{cor:toroidal-TLn}:

\begin{theoremABC}\label{thm:A}
Fix $n\ge 2$. Let $\TL_n$ be the iterated blow--up of $\LG(n,2n)$ constructed by starting with a transverse pair of Lagrangians $p^\pm\in \LG(n,2n)$ and  blowing up, in a prescribed order, the strict transforms of the osculating loci at $p^+$ and $p^-$. Let $\RL_n\colon\TL_n\longrightarrow \LG(n,2n)$ be the blow--up morphism, and $D_i^\pm$ ($0\leq i\leq n-2$) the exceptional divisors of the successive blow--ups over $p^\pm$.  Write $H:=(\RL_n)^*\OO_{\LG(n,2n)}(1)$.  Then the following holds.

\begin{enumerate}

\item $\TL_n$ is a smooth projective toroidal spherical variety for the natural action of the stabilizer of $(p^+,p^-)$ inside the symplectic group $\Sp_{2n}$. In particular, $\TL_n$ is a Mori dream space.

\item There is an equivariant flat morphism to the space of complete quadrics $\CQ_n$:
\begin{equation}\label{pln}
\PL_n \colon \TL_n \longrightarrow \CQ_n.
\end{equation}
Moreover, $\PL_n$ realizes the universal family over the Hilbert quotient of $\LG(n,2n)$ under the standard $\mathbb C^*$--action in the sense of \cite{BS,Kap}.

\item $\Pic(\TL_n)$ is free with standard $\mathbb Z$--basis
\[
\Pic(\TL_n)\ \cong\ \ZZ\!\cdot H\ \oplus\ \bigoplus_{i=0}^{n-2}\ZZ\!\cdot D_i^+\ \oplus\ \bigoplus_{i=0}^{n-2}\ZZ\!\cdot D_i^- .
\]

\item The effective cone is the rational polyhedral cone
\[
\Eff(\TL_n)=\bigl\langle D_0^+,\dots,D_{n-2}^+,D_{n-1}^+,\,D_0^-,\dots,D_{n-2}^-,D_{n-1}^-\bigr\rangle,
\]
where $D_{n-1}^{\pm}$ are two further boundary divisors besides the exceptional divisors $D_i^\pm$ ($0\le i\le n-2$), whose numerical classes are
$$
D_{n-1}^+\ \equiv\ H-\sum_{i=0}^{n-2}(n-i)\,D_i^+,
\qquad
D_{n-1}^-\ \equiv\ H-\sum_{i=0}^{n-2}(n-i)\,D_i^-.
$$

\item For integers $p,q$ such that $0\le p,q\le n-1$ and $p+q\le n$, set
\begin{equation}\label{npq}
N_{p,q}:= H-\sum_{i=0}^{p-1}(p-i)\,D_i^+-\sum_{i=0}^{q-1}(q-i)\,D_i^-\in \Pic(\TL_n),
\end{equation}
where empty sums are understood as $0$. Then $\Nef(\TL_n)$ is a rational polyhedral cone, and its extremal rays are $\mathbb{R}_{\geq 0}N_{p,q}$:
\[
\Nef(\TL_n)=\bigl\langle N_{p,q}\, \big| \,0\le p,q\le n-1,\,\, p+q\le n\bigr\rangle.
\]
The Cox ring $\Cox(\TL_n)$ is generated by the canonical sections of the boundary divisors $D_i^\pm$ ($0\le i\le n-1$) together with the $\Sp_{2n}$--modules of sections corresponding to the colors.

\item The variety $\TL_n$ is weak Fano for every $n\ge 2$, and Fano if and only if $n=2$. Moreover,  for every $n\ge 2$, the higher cohomology groups of the tangent bundle vanish:
\[
H^k\bigl(\TL_n,T_{\TL_n}\bigr)=0,\,\,\,\forall\,k>0.
\]
In particular, $\TL_n$ is locally rigid, namely, it admits no nontrivial local deformations.

\item The group of pseudoautomorphisms is given by
\[
\PsAut(\TL_n)=\Aut(\TL_n) \cong\ (\GL_n/\{\pm I\})\rtimes S_2,
\]
where $S_2$ is the symmetric group on $2$ elements.
\end{enumerate}
\end{theoremABC}

In addition, in Propositions~\ref{prop:Mori-generators-correct} and~\ref{prop:movcone-TLn-correct}, we determine explicit generators for the Mori cone $\NE(\TL_n)$ and for the cone of moving curves $\Mov_1(\TL_n)$. Furthermore, in Proposition~\ref{prop:mov-by-degrees} we provide an algorithm that computes the extremal generators of the movable cone of divisors $\Mov(\TL_n)$. We implemented these constructions in a collection of \textsc{Maple} scripts: given $n$ as input, the code returns the generators of $\Eff(\TL_n)$, $\Nef(\TL_n)$, $\Mov(\TL_n)$, $\NE(\TL_n)$, and $\Mov_1(\TL_n)$. See Remark~\ref{Maple} for implementation details. Examples for $n=2$ and $n=3$ are collected in Examples~\ref{subsec:TL2} and~\ref{subsec:TL3}. 


Our investigation of the pointed Kontsevich spaces begins by identifying $\TL_n$ as a general evaluation fiber in a Kontsevich space. Recall that Micha{\l}ek, Monin, and Wi\'sniewski \cite[Corollary 3.11]{MMW21} realized the space of complete quadrics
as the connected component of the $\mathbb C^*$--fixed locus of the Kontsevich space $\Mbar_{0,0}(\LG(n,2n),n)$, which parallels the constructions of Thaddeus \cite{Th1}. Using Theorem \ref{thm:A}\textup{(2)} and the geometric properties of its fibers established in Lemma \ref{moduli1l}, we lift such subvarieties in $\Mbar_{0,0}(\LG(n,2n),n)$ to $\Mbar_{0,k}(\LG(n,2n),n)$ with $k\leq 3$. This identification is summarized by Theorems~\ref{thm:Fpq-is-CQ}, \ref{thm:Mpq-is-TLn}, Proposition~\ref{prop:TLn-in-M01}, and Remark~\ref{rem:TLn-covers-dense-open-in-M01}, as follows:
\begin{theoremABC}\label{thm:B}
Consider the evaluation morphism at the first two markings
\[
\widetilde{\ev}:=\ \ev_1\times \ev_2\colon
\overline{M}_{0,3}(\LG(n,2n),n)\longrightarrow \LG(n,2n)\times \LG(n,2n).
\]
For a general pair $(p,q)\in \LG(n,2n)\times \LG(n,2n)$, the fiber of $\widetilde{\ev}$ over $(p,q)$ is isomorphic to $\TL_n$. Moreover, $\overline{M}_{0,1}(\LG(n,2n),n)$ contains an open dense subset covered by subvarieties isomorphic to $\TL_n$. 
\end{theoremABC}

Theorem~\ref{thm:B} suggests using the features of $\TL_d$ to probe the birational geometry of $\overline{M}_{0,1}(\LG(d,2d),d)$, an approach that,  via the embedding of the Lagrangian Grassmannians, also applies to $\overline{M}_{0,1}(\LG(n,2n),d)$ with $n\geq d$. To illustrate this principle, we investigate the birational geometry of $\overline{M}_{0,1}(\LG(n,2n),2)$, the moduli space of pointed conics in $\LG(n,2n)$. 

We denote by $H_1:=\ev_1^*H$ the pullback of the Pl\"ucker hyperplane class via the evaluation morphism at the marked point, and by $\Delta_{1|1}$ the boundary divisor parametrizing reducible conics. The unbalanced divisor $D_{\unb}$ is defined as the closure of the locus of conics that are double covers of a line when $n=2$; for $n\ge 3$, it is the closure of the locus of conics whose associated trivial rank $(n-2)$ subbundle meets a fixed $(n+1)$--dimensional subspace in $\PP^{2n-1}$.
On the unpointed space $\overline{M}_{0,0}(\LG(n,2n),2)$, two natural geometric divisor classes play a central role:
the class $H_{\sigma_2}$, defined by imposing incidence with a general codimension--$2$ Schubert cycle, and the class $T$, defined by the tangency condition that the image conic be tangent to a general Pl\"ucker hyperplane section. We view $H_{\sigma_2}$ and $T$ on $\overline{M}_{0,1}(\LG(n,2n),2)$ via pullback under the forgetful morphism. 

Using the geometry of $\TL_2$ together with the computations for the unpointed spaces, we obtain a complete description of the cones of divisors on $\overline{M}_{0,1}(\LG(n,2n),2)$. The following result collects Theorems~\ref{thm:Eff-general-n-k1}, \ref{thm:nef-pointed-conics-LG}, Proposition \ref{-K}, and Corollaries~\ref{cor:Fano-weakFano-MDS}, \ref{rem:Fano-index}.
\begin{theoremABC}\label{thm:C}
Let $n\ge 2$. The effective and nef cone of $\overline{M}_{0,1}(\LG(n,2n),2)$ are given by
\[
\operatorname{Eff}(\overline{M}_{0,1}(\LG(n,2n),2))
=\bigl\langle H_1,\, \Delta_{1|1},\,D_{\unb}\bigr\rangle\,\,\,\,{\rm and}\,\,\, 
\operatorname{Nef}(\overline{M}_{0,1}(\LG(n,2n),2))
=\bigl\langle H_1,\, H_{\sigma_2},\, T\bigr\rangle.
\]
The anticanonical divisor of $\overline{M}_{0,1}(\mathrm{LG}(n,2n),2)$ is given by 
$$
-K_{\overline{M}_{0,1}(\mathrm{LG}(n,2n),2)} = \left\lbrace
\begin{array}{lcll}
H_1 + \frac{5}{2}H_{\sigma_2} + \frac{3}{4}\Delta_{1|1} & \equiv & H_1 + H_{\sigma_2} + \frac{3}{2}T & \text{if } n = 2\\ 
H_1 + \frac{n+2}{2}H_{\sigma_2} + \frac{6-n}{4}\Delta_{1|1} & \equiv & H_1+(n-2)H_{\sigma_2}+\frac{6-n}{2}T & \text{if } n \geq 3
\end{array}\right.. 
$$
Furthermore, $\overline{M}_{0,1}(\LG(n,2n),2)$ is a Mori dream space for all $n\ge 2$. $\overline{M}_{0,1}(\mathrm{LG}(n,2n),2)$ is a Fano variety if and only if $2\le n\le 5$, in which case the Fano index is $1$; it is weak Fano if and only if $2\leq n\leq6$. The automorphism group is
\[
\Aut(\overline{M}_{0,1}(\LG(n,2n),2))\cong\PSp_{2n},\,\,\forall\,n\ge 2.
\]
\end{theoremABC}

Theorem~\ref{thm:C}, to our knowledge, provides one of the few uniform descriptions of the birational cones of a family of pointed Kontsevich spaces targeting a higher--rank homogeneous variety. 
We believe that broader classes of pointed Kontsevich spaces may be accessible by analogous compactifications in other types, arising as general evaluation fibers. As a step in this direction,
we extend the Kausz--type compactification construction to orthogonal Grassmannians in the final section. In particular, we obtain the following (Theorems  \ref{othm:TLn-spherical},   \ref{othm:Fano-weakFano}, \ref{othm:rigidity-absolute},  \ref{all3}):
\begin{corollaryABC}\label{thm:D}
All three Hilbert quotients in \cite{Th1} have smooth universal families.  These universal families are weak Fano varieties with vanishing higher cohomology of the tangent bundle. In particular, they are all locally rigid.
\end{corollaryABC}


\subsection*{Organization of the paper}
In Section~\ref{Sec1}, we fix notation and recall fundamental results.
In Sections~\ref{Sec2} and \ref{Sec2/2} we develop the geometry of the Kausz--type compactifications $\TL_n$ and prove Theorem~\ref{thm:A}.
Section~\ref{Sec3} is devoted to the modular interpretation of $\TL_n$ via evaluation fibers and the proof of Theorem~\ref{thm:B}.
Section~\ref{Sec4} applies this viewpoint to the divisor theory of $\overline{M}_{0,1}(\LG(n,2n),2)$, culminating in Theorem~\ref{thm:C}. In Section~\ref{og}, we introduce the compactification $\TO_n$ related to the orthogonal Grassmannians, and prove Corollary \ref{thm:D}.

\subsection*{Acknowledgments}
Hanlong Fang and Xian Wu are supported by National Key R\&D Program of China (No.~2022YFA1006700).
Alex Massarenti was supported by the PRIN 2022 grant (project 20223B5S8L) \emph{Birational geometry of moduli spaces and special varieties}, is a member of GNSAGA (INdAM), and thanks Peking University for the hospitality during a visiting period in which the collaboration leading to this article began.

\subsection*{Notation and conventions}
We work over $\C$. We set $V:=W\oplus W^\vee$ equipped with the standard symplectic form $\omega$, and we write
$\LG(n,2n):=\LG(V,\omega)$ for the Lagrangian Grassmannian.
\begin{itemize}[leftmargin=*, itemsep=2pt]
    \item The distinguished points $p_\pm\in\LG(n,2n)$ are as in \eqref{p+-}.
    \item $\CQ_n$ denotes the space of complete quadrics, and $\TL_n$ is the Kausz--type compactification considered in this paper.
    \item $\RL_n\colon \TL_n\to \LG(n,2n)$ and $\PL_n\colon \TL_n\to \CQ_n$ are the natural $\mathscr G$--equivariant morphisms (\eqref{rn1} and \eqref{pln}).
    \item $Z_k^\pm\subset \LG(n,2n)$ are the osculating loci defined from $p_\pm$, and $D_k^\pm\subset \TL_n$ are the corresponding exceptional divisors after the iterated blow-ups (Section~\ref{Sec2}).
    \item $H\in\Pic(\LG(n,2n))$ is the Pl\"ucker hyperplane class; for pointed moduli spaces we also use $H_1:=\ev_1^*H$.
    \item $\mathscr G\subset \Aut(\LG(n,2n))$ is the stabilizer of the ordered pair $(p_+,p_-)$.
    \item $\mathcal S$ denotes the tautological rank $n$ subbundle on $\LG(n,2n)$; we will use the homogeneous identification
    $T_{\LG(n,2n)}\simeq \Sym^2(\mathcal S^\vee)$.
\end{itemize}

\section{Preliminaries}\label{Sec1}

We fix basic notation for divisor cones and Mori dream spaces; standard references are
\cite{HK00,ADHL15,Br93,Pe14}.
For a normal projective variety $X$ we write $N^1(X)_\R$ for the N\'eron--Severi space and denote by
$\Eff(X)$, $\Mov(X)$, and $\Nef(X)$ the effective, movable, and nef cones in $N^1(X)_\R$.
If $f\colon X\dasharrow X'$ is a small $\Q$--factorial modification (SQM), then the induced map
$f^*\colon N^1(X')_\R\to N^1(X)_\R$ is an isomorphism and it identifies $\Eff$ and $\Mov$.  Write $N_1(X)_{\mathbb R}$ for the real vector space of $1$--cycles on $X$ modulo the numerical equivalence.


Assume that the divisor class group $\Cl(X)$ is free and finitely generated, and choose a subgroup
$G\subset \Div(X)$ mapping isomorphically onto $\Cl(X)$.
The \emph{Cox ring} of $X$ is
\[
\Cox(X):=\bigoplus_{[D]\in \Cl(X)} H^0\bigl(X,\OO_X(D)\bigr),
\]
with multiplication induced by multiplication of rational sections.

\begin{definition}\label{def:MDS}
A normal projective $\Q$--factorial variety $X$ is a \emph{Mori dream space} (MDS) if:
\begin{itemize}
  \item[(i)] $\Pic(X)$ is finitely generated;
  \item[(ii)] $\Nef(X)$ is a rational polyhedral cone generated by finitely many semiample classes;
  \item[(iii)] there are finitely many SQMs $f_i\colon X\dasharrow X_i$ such that each $X_i$ satisfies (ii) and
  \[
  \Mov(X)=\bigcup_i f_i^*\bigl(\Nef(X_i)\bigr).
  \]
\end{itemize}
Equivalently, the nef cones of the SQMs glue to a fan supported on $\Mov(X)$, which extends canonically to a fan on $\Eff(X)$.
\end{definition}

\begin{definition}\label{MCD}
Let $X$ be a Mori dream space.
The \emph{Mori chamber decomposition} of $\Eff(X)$ is the fan whose maximal cones are
\[
\mathcal{C}_i:=\big\langle g_i^*\bigl(\Nef(Y_i)\bigr),\;\Exc(g_i)\big\rangle\subset N^1(X)_\R,
\]
where $g_i\colon X\dasharrow Y_i$ runs over the finitely many birational contractions to Mori dream spaces and
$\Exc(g_i)$ denotes the set of $g_i$--exceptional prime divisors on $X$.
\end{definition}

\begin{remark}\label{dimCox}
By \cite[Proposition~2.9]{HK00}, for a normal $\Q$--factorial projective variety $X$ the Cox ring is finitely generated
if and only if $X$ is a Mori dream space. See also \cite[Section~1.3]{ADHL15}.
\end{remark}

\begin{definition}\label{def:movable-curve}
An irreducible curve $C$ on a $\Q$--factorial projective variety $X$ is movable
if there exists a flat family of curves
$\pi\colon \mathscr C\to T$ over an irreducible base together with a morphism
$\mathrm{ev}\colon \mathscr C\to X$ such that:
\begin{itemize}
\item[-] for a general $t\in T$, the fiber $\mathscr C_t$ is an irreducible curve
and $\mathrm{ev}(\mathscr C_t)$ is numerically equivalent to $C$;
\item[-] the evaluation map $\mathrm{ev}$ is dominant.
\end{itemize}
The cone of moving curves is the closed convex cone
\[
\Mov_1(X)\ :=\ \overline{\bigl\langle [C]\ \big|\ C\subset X\ \text{movable}\bigr\rangle}
\ \subset\ N_1(X)_\R.
\]
\end{definition}

\begin{remark}\label{rem:bdpp}
If $X$ is smooth and projective, then \cite[Theorem~2.2]{BDPP13} identifies
$\Mov_1(X)$ with the dual of the pseudoeffective cone of divisors:
\[
\Mov_1(X)\ =\ \overline{\Eff}(X)^\vee
\ :=\ \bigl\{\alpha\in N_1(X)_\R\ \big|\ \alpha\cdot E\ge 0\ \text{for every }E\in \Eff(X)\bigr\}.
\]
\end{remark}

\begin{definition}
Let $X$ be a normal projective variety.
\begin{itemize}
  \item[(i)] We say that $X$ is \emph{Fano} if $-K_X$ is ample.
  \item[(ii)] We say that $X$ is \emph{weak Fano} if $-K_X$ is nef and big.
  \item[(iii)] A \emph{log Fano pair} is a pair $(X,\Delta)$ where $X$ is normal projective, $\Delta$ is an effective $\Q$--divisor,
  $(X,\Delta)$ has klt singularities, and $-(K_X+\Delta)$ is ample.
\end{itemize}
\end{definition}

\begin{remark}\label{Fano-logFano-MDS}
By \cite[Corollary~1.3.2]{BCHM10}, every log Fano pair is a Mori dream space.
In particular, a Fano or weak Fano variety is a Mori dream space after adding a suitable effective boundary.
\end{remark}

For a $\QQ$--Gorenstein Fano variety $X$, the Fano index is defined as
\[
\iota(X):= \max\bigl\{m\in\QQ_{>0}\, \big|\, -K_X= mH\ \text{for some ample Cartier divisor }H\in\Pic(X)\,\bigr\},
\]
equivalently, the largest positive integer dividing $-K_X$ in $\Pic(X)$.

We denote by $\PsAut(\TL_n)$ the group of pseudoautomorphisms of $\TL_n$, namely, birational
self--maps which are isomorphisms in codimension $1$.

Let $\mathscr{G}$ be a connected reductive algebraic group and let $\mathscr{B}\subset \mathscr{G}$ be a Borel subgroup.

\begin{definition}
A normal irreducible $\mathscr{G}$--variety $X$ is \emph{spherical} if it contains an open dense $\mathscr{B}$--orbit.
\end{definition}

\begin{definition}
A smooth projective spherical $\mathscr{G}$--variety $X$ is \emph{wonderful} if it has a unique closed $\mathscr{G}$--orbit and the boundary
$X\setminus \mathscr{G}\cdot x$ is a divisor with normal crossings whose strata are exactly the $\mathscr{G}$--orbit closures.
\end{definition}

\begin{definition}
Let $X$ be a spherical $\mathscr{G}$--variety with a fixed Borel subgroup $\mathscr{B}$. The boundary divisors of $X$ are its $\mathscr{G}$--stable prime divisors. The colors of $X$ are its $\mathscr{B}$--stable but not $\mathscr{G}$--stable prime divisors.
\end{definition}

\begin{remark}\label{sphMDS}
Spherical varieties are Mori dream spaces; the Mori chamber decomposition can be read from the combinatorics of colored fans.
We refer to \cite[Section~4]{Pe14} for a convenient overview and further references.
\end{remark}

\begin{definition}\label{def:toroidal}
A normal spherical $\mathscr{G}$--variety $X$ is called toroidal if no color contains a $\mathscr{G}$--orbit closure in $X$.
\end{definition}

Let $\mathscr{G}$ be a semisimple group of adjoint type. Denoted by $\overline{\mathscr{G}}$ its De Concini--Procesi  wonderful compactification.
Then, from the point of view above,  $\overline{\mathscr{G}}$  is a smooth projective spherical $\mathscr{G}\times \mathscr{G}$--variety,
its boundary divisors come from $\mathscr{G}\times \mathscr{G}$--stable and $\mathscr{B}\times \mathscr{B}$--stable degeneracy conditions, and
$\overline{\mathscr{G}}$  is therefore a Mori dream space with Cox ring and Mori chamber decomposition governed by the boundary divisors and colors.

When $\mathscr G$ is reductive but not of adjoint type, the wonderful compactification does not exist. While a general framework as simple as the De Concini-Procesi construction is lacking, Kausz constructed a beautiful compactification for general linear groups. We summarize certain important properties of Kausz compactifications as follows.
\begin{theorem}[\cite{Ka} and \cite{FW}]\label{gwond}
Fix $\mathscr G:=\GL_n$.
$\overline{\mathscr G}$ is a smooth projective $\mathscr G\times \mathscr G$--spherical variety. The complement of the open
$\mathscr G\times \mathscr G$--orbit in $\overline{\mathscr G}$ consists of $2n$ smooth prime divisors with simple normal crossings $D^+_1, 
\cdots,D^+_n,D^-_1$, $\cdots,D^-_n$ such that the following holds.
\begin{enumerate}[label={\rm(\Alph*)}]
\item $D^+_1\cong D^-_1\cong\CC_n$, where $\CC_n$ is the space of complete collineations.

\item There is a $\mathscr G\times \mathscr G$--equivariant flat retraction 
\begin{equation*}
\PL_n:\overline{\mathscr G}\longrightarrow D_1^-\cong\CC_n  
\end{equation*} such that the restriction  $\PL_n|_{D^+_1}:D_1^+\rightarrow D_1^-$ 
is an isomorphism and that for $2\leq i\leq r$,
\begin{equation}\PL_n(D^-_i)=\PL_n(D^+_{n+2-i})=:\check D_i\,\,    
\end{equation}

\item The closures of $\mathscr G\times \mathscr G$--orbits in $\overline{\mathscr G}$ are one-to-one given by
\begin{equation}\label{kinrule}
\bigcap\nolimits_{\,\,i\in I^+}D^+_i\mathbin{\scaleobj{1.1}{\bigcap}} \bigcap\nolimits_{\,\,i\in I^-}D^-_i
\end{equation} 
for subsets $I^+,I^-\subset\{1,2,\cdots,n\}$
such that $\min(I^+)+\min(I^-)\geq n+2$ with the convention that $\min(\emptyset)=+\infty$.

\item  $D_1^-\cong\CC_n$ is wonderful with the $\mathscr G\times \mathscr G$--stable divisors $\check D_i$,  $2\leq i\leq n$.
\end{enumerate}
\end{theorem}

Notice that $\overline{\mathscr G}$ is spherical but not of wonderful type in the strict sense of De Concini--Procesi. In particular, it is not a simple spherical variety: indeed, certain pairs of boundary components do not intersect; for instance $D_0^-\cap D_0^+ = \emptyset$.
Consequently, and in contrast to the case of simple spherical varieties treated in \cite{Br2}, the divisor theory may not be well described by the color fans.




\section{Kausz--Type Compactifications of Spaces of Symmetric Matrices}\label{Sec2}

Let $V$ be a complex vector space of dimension $2n$ endowed with a nondegenerate symplectic form $\omega$. We denote by $\LG(V,\omega)$ the Lagrangian Grassmannian parametrising Lagrangian subspaces of $V$, namely, maximal isotropic subspaces of dimension $n$. Then $\LG(V,\omega) \cong \Sp(V,\omega)/P$ is a rational homogeneous variety of type $C_n$, where $P$ denotes the maximal parabolic subgroup associated to the cominuscule root $\alpha_n$ in the Dynkin diagram
\[
\overset{\alpha_1}{\circ}-\overset{\alpha_2}{\circ}-\cdots-
\overset{\alpha_{n-1}}{\circ}\Leftarrow\overset{\alpha_n}{\circ}.
\]
In particular, $\LG(V,\omega)$ is smooth projective variety of dimension $\frac{n(n+1)}{2}$, $\Pic(\LG(V,\omega))\cong \Z$ is generated by an ample line bundle $\OO_{\LG(V,\omega)}(1)$, and the Fano index of $\LG(V,\omega)$ is $n+1$, namely, the anticanonical class  \(  -K_{\LG(V,\omega)}=(n+1)\OO_{\LG(V,\omega)}(1)\).
Moreover, $\LG(V,\omega)$ is  a spherical $\Sp(V,\omega)$--variety.


We now introduce the Kausz--type compactification of the open linear orbit inside $\LG(V,\omega)$. Let $W$ be an $n$--dimensional complex vector space and set $V:=W\oplus W^\vee$ with its canonical symplectic form
\[
\omega((w,\varphi),(w',\varphi')):= \varphi'(w) - \varphi(w'),
   \qquad
   w,w'\in W,\ \varphi,\varphi'\in W^\vee .
\]
Henceforth, we omit $\omega$ from the notation and denote the Lagrangian Grassmannian by $\LG(n,2n)$. Let $$\mathscr{G}:=\GL(W^\vee)\times\Gm \subset \Sp(V)$$ be the subgroup defined as follows:  $g\in\GL(W^\vee)$ acts on $W^\vee$ by the given representation
and on $W=(W^\vee)^\vee$ by the contragredient representation, while
$\lambda\in\Gm$ acts with weight $+1$ on $W$ and weight $-1$ on
$W^\vee$.  Explicitly,
\begin{equation}\label{Gm}
(g,\lambda)\cdot (w,\varphi)
  := \bigl(\lambda (g^{-1})^{\!\top} w,\ \lambda^{-1} g\varphi\bigr),
  \qquad w\in W,\ \varphi\in W^\vee ,
\end{equation}
where $(g^{-1})^{\!\top}\colon W\to W$ is the transpose inverse of $g$.
Based on the 
Plücker embedding induced by $\OO_{\LG(n,2n)}(1)$ 
\begin{equation}\label{plk}
e_{\LG}\colon\LG(n,2n)\hookrightarrow \Pbb\big(\bigwedge^n (W\oplus W^\vee)\big),
\end{equation}
we define a rational $\mathscr{G}$--equivariant map
\begin{equation}\label{kl1}
\KL_n \colon \LG(n,2n) \dashrightarrow 
\Pbb\big(\bigwedge^n (W\oplus W^\vee)\big)\times\PP\big(\Sym^2 W^\vee\big)\times
\PP\big(\Sym^2\big(\bigwedge^2 W^\vee\big)\big)\times\cdots\times
\PP\big(\Sym^2\big(\bigwedge^n W^\vee\big)\big),
\end{equation}
sending a general Lagrangian subspace to a compatible collection of symmetric forms obtained by taking appropriate minors of its Plücker coordinates.
Following \cite{FW}, we make the next definition.
\begin{definition}\label{def:TLn-Kausz}
The \emph{Kausz--type compactification} $\TL_n$ of the open $\mathscr{G}$--orbit in $\LG(n,2n)$ is defined as the closure of the image of $\KL_n$:
\[
\TL_n := \overline{\KL_n\big(\LG(n,2n)\big)} \subset 
\Pbb\big(\bigwedge^n (W\oplus W^\vee)\big)\times
\PP\big(\Sym^2 W^\vee\big)\times
\PP\big(\Sym^2\big(\bigwedge^2 W^\vee\big)\big)\times\cdots\times
\PP\big(\Sym^2\big(\bigwedge^n W^\vee\big)\big).
\]
We denote by
\begin{equation}\label{rn1}
\RL_n \colon \TL_n \longrightarrow \LG(n,2n)\subset\Pbb\big(\bigwedge^n (W\oplus W^\vee)\big)    
\end{equation}
the projection to the first factor.
\end{definition}
\begin{remark}
Since $\operatorname{GL}(W^\vee) \times \mathbb{G}_m$-- and $\operatorname{GL}(W^\vee)$--orbits are identical in the above construction, we do not distinguish between treating $\TL_n$ as a spherical $\operatorname{GL}(W^\vee) \times \mathbb{G}_m$-- or $\operatorname{GL}(W^\vee)$--variety.
\end{remark}
Consider the following pair of complementary Lagrangian subspaces 
\begin{equation}\label{p+-}
p_+ := 0\oplus W^\vee,\,\,\, p_- :=W\oplus 0 \,\,\,\in\,\LG(n,2n),
\end{equation}
which will play symmetric roles in the paper. For each $k$ with $1\leq k\leq n-1$, we denote by
\[
\Osc^k_{p_\pm}\big(\LG(n,2n)\big) \subset \Pbb\big(\bigwedge^n(W\oplus W^\vee)\big)
\]
the $k$--th osculating space at $p_\pm$ with respect to the Pl\"ucker embedding (see \cite[Section~6]{MR18} and \cite[Section~3.4]{FMR20} for detailed
descriptions of the osculating spaces of Grassmannians and Lagrangian
Grassmannians in Pl\"ucker coordinates, respectively). For later use we record the following
lemma, which packages together the behavior of the osculating spaces at a
point of $\LG(n,2n)$, the loci covered by low--degree rational curves through
that point, and the dimensions of these loci.

\begin{lemma}\label{lem:osculating-curves-LG}
Let 
\(
\LG(n,2n)=\LG(V,\omega)\subset\Pbb(\bigwedge^n V)
\)
be the Lagrangian Grassmannian in its Pl\"ucker embedding.
Fix a point $p\in\LG(n,2n)$ corresponding to a Lagrangian 
$\Lambda_0\subset V$. 
For an integer $d$ with $1\leq d\leq n-1$, set
\[
Z_d(p)\ :=\ \LG(n,2n)\cap\Osc^d_p\big(\LG(n,2n)\big)
\]
and let $R_d(p)\subset \LG(n,2n)$ be the union of all irreducible rational
curves $C\subset\LG(n,2n)$ with $p\in C$ and $\deg_H(C)\leq d$.
Then the following hold:
\begin{itemize}
  \item[\text{(i)}] As closed subsets of $\LG(n,2n)$ one has
  \[
  Z_d(p)
  \ =\
  \big\{\Lambda\in\LG(n,2n)\,\big|\,
      \dim(\Lambda\cap\Lambda_0)\geq n-d\big\}.
  \]
  \item[\text{(ii)}] The closed set $Z_d(p)$ is the Zariski closure of $R_d(p)$; in
  particular $Z_d(p)$ is irreducible and
  \[
  \overline{R_d(p)}\ =\ Z_d(p).
  \]
  \item[\text{(iii)}] The variety $Z_d(p)$ has dimension
  \[
  \dim Z_d(p)
  \ =\
  \frac{d\big(2n-d+1\big)}{2}
  \ =\
  dn-\frac{d(d-1)}{2}.
  \]
\end{itemize}
One has $\Osc^d_p\big(\LG(n,2n)\big)=\Pbb(\bigwedge^n V)$ and hence
$Z_d(p)=\LG(n,2n)$ for $d\geq n$.
\end{lemma}

\begin{proof}
By $\Sp(V,\omega)$--homogeneity we may assume $\Lambda_0=W\oplus 0\subset W\oplus W^\vee=V$.
The big cell of $\LG(n,2n)$ consisting of Lagrangians transverse to $0\oplus W^\vee$ is identified with $\Sym^2 W^\vee$ via graphs of symmetric maps $A\colon W\to W^\vee$, and $p$ corresponds to $A=0$.
In these coordinates, $\dim(\Lambda\cap\Lambda_0)\ge n-d$ is equivalent to $\rk(A)\le d$, giving \textup{(i)}. The locus $\{\rk(A)\le d\}$ is irreducible and equals the closure of the union of degree $\le d$ rational curves through $A=0$, yielding \textup{(ii)}.
Finally, the standard determinantal computation for symmetric matrices gives $\dim\{\rk(A)\le d\}= \frac{d(2n-d+1)}{2}$, which is \textup{(iii)}.
\end{proof}

We define
\[
Z_k^\pm := \LG(n,2n)\cap \Osc^k_{p_\pm}\big(\LG(n,2n)\big).
\]
Each $Z_k^\pm$ is $\mathscr G$--stable and we have chains
\[
Z_0^+:=\{p_+\}\subset Z_1^+\subset\cdots\subset Z_{n-1}^+,\qquad Z_0^-:=
\{p_-\}\subset Z_1^-\subset\cdots\subset Z_{n-1}^-.
\]

\begin{figure}[t]
\centering
\[
\LG(n,2n)=\TL_n^{0}
\xleftarrow{\ \mathrm{Bl}_{Z_0^+}\ }\TL_n^{1,+}
\xleftarrow{\ \mathrm{Bl}_{Z_1^+}\ }\cdots
\xleftarrow{\ \mathrm{Bl}_{Z_{n-2}^+}\ }\TL_n^{n-1,+}
\xleftarrow{\ \mathrm{Bl}_{Z_0^-}\ }\cdots
\xleftarrow{\ \mathrm{Bl}_{Z_{n-2}^-}\ }\TL_n^{n-1}\cong \TL_n.
\]
\caption{Schematic order of the iterated blow-ups along the osculating loci $Z_k^\pm$ producing $\TL_n$.}
\label{fig:blowups-TLn}
\end{figure}

\begin{proposition}\label{prop:ordinary-sing}
For every integer $k$ with $1\leq k\leq n-1$ the following holds:
$$
\dim(Z_k^+) = \frac{n(n+1)}{2}-\frac{(n-k)(n-k+1)}{2}
$$ 
and $Z_k^+$ has ordinary singularities of multiplicity $\mult_{Z_i^+}(Z_k^+) \ =\ k-i+1$ along $Z_i^+$ for every $0\leq i<k$. The analogous statement holds for the chain $Z_0^-\subset\cdots\subset Z_{n-1}^-$.

Let $\TL_n^k$ be obtained from $\LG(n,2n)$ by sequentially blowing up the strict transforms of $Z_0^+, Z_1^+, \dots, Z_{k-1}^+, Z_0^-$, $Z_1^-$, $\dots$, $Z_{k-1}^-$ (in that order).
Then $\TL_n^k$ is smooth, and the strict transforms of $Z_k^+$ and $Z_k^-$ in $\TL_n^k$ are smooth. In particular, $\TL_n^{n-1}$ is smooth. 
\end{proposition}

\begin{proof}
The dimension formula is Lemma~\ref{lem:osculating-curves-LG}\textup{(iii)}.

Fix $k$ and work in the standard big cell $U_+$ around $p_+$, identified with $\Sym^2 W$ by graphs of symmetric maps.
In these coordinates, $Z_k^+$ is the symmetric determinantal locus
\[
Z_k^+ \cap U_+ \ =\ \{A\in \Sym^2 W : \rk(A)\le k\},
\]
cut out by the $(k+1)\times (k+1)$ minors of $A$. It is a standard determinantal computation for symmetric matrices that along the lower rank stratum
$Z_i^+\cap U_+=\{\rk(A)\le i\}$ the singularity is ordinary and the order of vanishing is $k-i+1$, i.e.\ $\mult_{Z_i^+}(Z_k^+)=k-i+1$.
The analogous statement for the $-$--chain follows by working in the opposite big cell $U_-$ around $p_-$.

Finally, let $x$ be a point lying on (strict transforms of) both chains. After a symplectic change of basis we may assume that, in an affine chart around $x$, a Lagrangian is represented by a symmetric block matrix
\[
\begin{pmatrix}
A & B\\
B^{\top} & C
\end{pmatrix},
\qquad
A\in \Sym^2 U,\ \ C\in \Sym^2 U',
\]
for a decomposition $W=U\oplus U'$ with $\dim U=n-\ell$ and $\dim U'=\ell$.
In these coordinates the defining equations for the $+$--chain involve only minors of the block $A$, whereas the defining equations for the $-$--chain involve only minors of the block $C$; the entries of $B$ give transverse coordinates.
In particular, the two chains meet cleanly and are locally defined in independent coordinates. Therefore the successive blow-ups along the strict transforms of
$Z_0^+,Z_1^+,\dots,Z_{k-1}^+,Z_0^-,\dots,Z_{k-1}^-$ resolve the ordinary singularities without introducing new ones, and $\TL_n^k$ as well as the strict transforms of $Z_k^\pm$ are smooth.
\end{proof}

We now compare $\TL_n^{n-1}$ with the compactification $\TL_n$  given in Definition \ref{def:TLn-Kausz}.

\begin{proposition}\label{prop:TLn-blowup-equals-Kausz}
There is a natural $\mathscr{G}$--equivariant
isomorphism $\TL_n^{n-1} \cong \TL_n$ such that the blow up morphism $\beta\colon\TL_n^{n-1}\to\LG(n,2n)$ identifies with the projection
$\RL_n\colon\TL_n\to\LG(n,2n)$ in (\ref{rn1}).
\end{proposition}

\begin{proof}
Set $Z_0^+ := \{p_+\}$ and $Z_0^- := \{p_-\}$.
The morphism $\RL_n$ is projective, birational, and an isomorphism over the locus where $\KL_n$ is regular.
In Pl\"ucker coordinates one checks:
\begin{enumerate}[label=(\alph*)]
    \item $\KL_n$ is regular on the complement of $\bigcup_{k=0}^{n-1}(Z_k^+\cup Z_k^-)$;
  \item  for $0\le k\le n-2$, a general point of $Z_k^\pm$ is an ordinary base point of the linear systems defining $\KL_n$, with multiplicity prescribed by Lemma~\ref{lem:osculating-curves-LG};
  \item  $\KL_n$ is regular at a general point of $Z_{n-1}^\pm$.
\end{enumerate}
Therefore blowing up the centers in Proposition~\ref{prop:ordinary-sing} resolves the indeterminacies of $\KL_n$ and identifies the resulting graph compactification with $\TL_n$, yielding the claimed $\mathscr{G}$--equivariant isomorphism.
\end{proof}

\begin{remark}
As in \cite[Lemma 2.2]{FW}, one can show that the iterated blow--up construction yielding $\TL_n$ is independent of the order in which the blow--ups are performed.     
\end{remark}
\begin{definition}\label{dk}
Denote by $D_k^\pm$ the exceptional divisor over $Z_k^\pm$ for $0\leq k\leq n-2$, and by $D_{n-1}^\pm$
the strict transform of $Z_{n-1}^\pm$ in $\TL_n^{n-1}\cong\TL_n$.  
\end{definition}

\begin{theorem}\label{thm:TLn-spherical}
The variety $\TL_n$ is smooth and projective. The group $\mathscr{G}$
acts on $\TL_n$ with an open dense orbit isomorphic to a homogeneous space $\mathscr{G}/\mathscr H$, where $\mathscr H\subset\mathscr G$ is the stabilizer of a fixed nondegenerate symmetric form on $W^\vee$. The complement of this open orbit is a divisor with $2n$ smooth irreducible components
\begin{equation}\label{boundd}
D_0^-,\dots,D_{n-1}^-,\ D_0^+,\dots,D_{n-1}^+.
\end{equation}
In particular, $\TL_n$ is a smooth projective spherical $\mathscr G$--variety.
\end{theorem}

\begin{proof}
By Proposition~\ref{prop:TLn-blowup-equals-Kausz}, $\TL_n$ is obtained by an iterated blow--up of $\LG(n,2n)$ along smooth centers. Hence $\TL_n$ is smooth and projective. Moreover, the open dense $\mathscr{G}$--orbit is identified with the open subset of $\Sym^2 W^\vee$ consisting of nondegenerate symmetric forms, and the boundary divisors are given by (\ref{boundd}). It is clear that $\mathscr{B}\cdot[p_{-}]$ is an open dense $\mathscr{B}$--orbit in $\mathscr G/\mathscr H$. The proof is complete.
\end{proof}


Let $\Sym^2(W^\vee)$ denote the space of symmetric $n\times n$ matrices and let $\PP(\Sym^2 W^\vee) \cong \PP^{\frac{n(n+1)}2-1}$ be its projectivization. The \emph{space of complete quadrics} in $\PP^{n-1}$ can be constructed by successively blowing up the strata of symmetric matrices of rank at most $k$. We denote this smooth projective variety by $\CQ_n$. Note that the first boundary divisor $D_0^-\subset\TL_n$ is isomorphic to $\CQ_n$, and so is $D_0^+$.

\begin{proposition}\label{prop:D1-MLn}
There are $\mathscr G$--equivariant isomorphisms $D_0^- \ \cong\ \CQ_n \ \cong\ D_0^+$.
\end{proposition}
\begin{proof}
Let $E_0^{+}$ be the exceptional divisor of the blow up of $\LG(n,2n)$ in $p_{+}$. Then $E_{0}^{+}\cong \mathbb{P}^{\frac{n(n+1)}{2}-1}$. Finally, the claim follows from Proposition \ref{prop:TLn-blowup-equals-Kausz} and the blow--up construction of $\CQ_n$ in \cite{Va2}.
\end{proof}

Composing the projection onto $D_0^{\pm}$ with the identification with $\CQ_n$, we can obtain a $\mathscr G$--equivariant flat morphism 
\begin{equation}
\PL_n \colon \TL_n \longrightarrow \CQ_n
\end{equation}
as in \cite{FW}, whose restriction to $D_0^{\pm}$ is an isomorphism onto $\CQ_n$.

\subsection*{Mille Cr\^epes charts on \texorpdfstring{$\TL_n$}{TL\_n}}
\label{subsec:mille-crepes}
Following \cite[Section~3]{FW}, we shall define the Mille Crêpes coordinates for $\TL_n$, by exploiting the above iterated blow--up construction.  For simplicity, we present them here in a concise, abstract form; the full details can be found in \cite{FMW}. 

For each $I=(i_1,i_2,\cdots,i_n)\in\{(i_1,i_2,\cdots,i_n)\in\mathbb Z^n:1\leq i_1<i_2<\cdots<i_n\leq 2n\}=:\mathbb I_n$, the {\it Pl\"ucker coordinate function} $P_{I}$ on ${\rm Spec}\,\mathbb Z\left[x_{ij}(1\leq i\leq n,1\leq j\leq 2n)\right]$ is the determinant of the matrix $(x_{ij})$ consisting of the $i_1$--{th}, $\cdots$, $i_n$--{th} columns. Denote by $\operatorname{Pr}\colon\mathcal G(n,2n)\rightarrow G(n,2n)$  the natural projection from the torsor
\begin{equation*}
\mathcal G(n,2n):=\left\{\mathfrak p\in{\rm Spec}\mathbb Z\left[x_{ij}\right]
:P_{I}\notin\mathfrak p\,\,{\rm for\,\,a\,\,certain\,\,}I\in\mathbb I_{n}\right\} 
\end{equation*} 
to the Grassmannian $G(n,2n)$. For each $x\in G(n,2n)$, let $\widetilde{x}$ be an element in the preimage $(\operatorname{Pr})^{-1}(x)$.

Now, for $0\leq k\leq n$, we
define index sets
\begin{equation*}
\mathbb I_{n}^{k}:=\big\{(i_1,\cdots,i_n)\in\mathbb Z^n:{ 1\leq i_1<\cdots<i_{n-k}\leq n\,;n+1\leq i_{n-k+1}<\cdots<i_n\leq 2n}\},
\end{equation*}
and rational maps
\begin{equation}\label{fk}
\begin{split}
f^k\colon\mathrm{LG}(n,2n)\dashrightarrow\mathbb {P}^{N^k_{n}},\,\,\,x\mapsto[\cdots,P_I(\widetilde x),\cdots]_{I\in\mathbb I^k_{n}}.
\end{split}   
\end{equation}
Then (\ref{kl1}) takes the form
\begin{equation}\label{fskl}
\begin{split}
\KL_{n}:=&(e_{\mathrm{LG}}, f^0,\cdots,f^n)\colon\mathrm{LG}(n,2n)\dashrightarrow\mathbb {P}^{N_{n}}\times\mathbb {P}^{N^0_{n}}\times\cdots\times\mathbb {P}^{N^n_{n}}\\ 
&x\mapsto\big( [\cdots ,P_I(\widetilde x),\cdots]_{I\in\mathbb I_{n}},[\cdots,P_I(\widetilde x),\cdots]_{I\in\mathbb I^0_{n}},\cdots,[\cdots,P_I(\widetilde x),\cdots]_{I\in\mathbb I^n_{n}}\big).
\end{split}
\end{equation}
Here $N_{n}:=\frac{(2n)!}{n!\cdot n!}-1$ and $N_n^k:=|\mathbb I_{n}^{k}|-1$; $e_{\LG}$ is the Pl\"ucker embedding in (\ref{plk}).

On the open subset of Lagrangian subspaces transverse to $W$ every $L$ is the graph of a unique linear map $A\colon W\to W^\vee$ which is symmetric with respect to the natural pairing.  Choosing a basis $\{e_1,\cdots,e_n\}$ for $W$ and letting $\{e_1^*,\cdots,e_n^*\}$ denote its dual basis in $W^\vee$, we may thus identify this open subset with the affine space of symmetric matrices
$Y=(y_{ij})_{1\le i, j\le n}$. 
More generally, for each Lagrangian subspace $L_{12\cdots \ell}$ generated by $e_{\ell+1},\cdots,e_n,e^*_{1},\cdots,e^*_{\ell}$, where $0\leq \ell\leq n$, 
$\LG(W \oplus W^\vee)$  has the following affine coordinate chart around $L_{12\cdots \ell}$: 
\begin{equation}\label{ull}
U_{12\cdots \ell}:=\{\,\,\,\,\underbracedmatrixll{\,\,\,Z\\\,\,\,-X^{\!\top}}{\,\,\,\,\,\,\,\,\,\,\,\,\,\,\,\,\ell\,\,\rm columns}
  \hspace{-.4in}\begin{matrix}
  &\hfill\tikzmark{a}\\
  &\hfill\tikzmark{b}  
  \end{matrix} \,\,\,\,\,
  \begin{matrix}
  0\\
I_{n-\ell}\\
\end{matrix}\hspace{-.11in}
\begin{matrix}
  &\hfill\tikzmark{c}\\
  &\hfill\tikzmark{d}
  \end{matrix}\hspace{-.11in}\begin{matrix}
  &\hfill\tikzmark{g}\\
  &\hfill\tikzmark{h}
  \end{matrix}\,\,\,\,
\begin{matrix}
I_{\ell}\\
0\\
\end{matrix}\hspace{-.11in}
\begin{matrix}
  &\hfill\tikzmark{e}\\
  &\hfill\tikzmark{f}\end{matrix}\hspace{-.25in}\underbracedmatrixrr{X\\Y}{(n-\ell)\,\,\rm columns}\,\,\,\,\}\cong\mathbb A^{\frac{n(n+1)}{2}},
  \tikz[remember picture,overlay]   \draw[dashed,dash pattern={on 4pt off 2pt}] ([xshift=0.5\tabcolsep,yshift=7pt]a.north) -- ([xshift=0.5\tabcolsep,yshift=-2pt]b.south);\tikz[remember picture,overlay]   \draw[dashed,dash pattern={on 4pt off 2pt}] ([xshift=0.5\tabcolsep,yshift=7pt]c.north) -- ([xshift=0.5\tabcolsep,yshift=-2pt]d.south);\tikz[remember picture,overlay]   \draw[dashed,dash pattern={on 4pt off 2pt}] ([xshift=0.5\tabcolsep,yshift=7pt]e.north) -- ([xshift=0.5\tabcolsep,yshift=-2pt]f.south);\tikz[remember picture,overlay]   \draw[dashed,dash pattern={on 4pt off 2pt}] ([xshift=0.5\tabcolsep,yshift=7pt]g.north) -- ([xshift=0.5\tabcolsep,yshift=-2pt]h.south);
\end{equation}
with coordinates
\begin{equation}\label{ulxl}
\begin{split}
&Z:=(\cdots,z_{ij},\cdots)_{1\leq i,j\leq \ell},\,\,X:=(\cdots,x_{ij},\cdots)_{1\leq i\leq \ell,\ell+1\leq j\leq n},\,\,Y:=(\cdots,y_{ij},\cdots)_{\ell+1\leq i,j\leq n},
\end{split}   
\end{equation}
We adopt the convention $Y$ and $Z$ are symmetric with $y_{ij}=y_{ji}$ and $z_{ij}=z_{ji}$. 
One may readily verify that all such coordinate charts cover $\LG(W \oplus W^\vee)$ up to permutations of $\{1,2,\cdots,n\}$. 

Now, it is easy to verify that
\begin{lemma}\label{sepal}
The subvarieties
$Z_k^+\,\,\,
({\rm resp.}\,\, Z_k^-)$
are defined locally in $U_{12\cdots \ell}$ by the vanishing of suitable minors of the matrix $Z=(z_{ij})$ ({\rm resp.}\,\,$Y=(y_{ij})$). More precisely, when $n-\ell\leq k$ ($\ell\leq k$), the restriction of $Z^+_k$ (resp. $Z^-_k$) to $U_{12\cdots \ell}$ is defined by the vanishing of all rank $(k+1-n+\ell)$ (resp. $(k+1-\ell)$) minors of $Z$ (resp. $Y$); otherwise, the restriction is empty.
\end{lemma}

The Mille Cr\^epes construction proceeds as follows. Note that $\TL_n$ is obtained by blowing up successively the centers in order of increasing
dimension. At each blow–up
step one chooses affine charts in which one of the minors defining the center is equal to one and the others are written as products of this
minor and new coordinates. Iterating this procedure one obtains an
open covering
\[
\TL_n \ =\bigcup_{g \in \operatorname{Aut}(\TL_n)} \; \ \bigcup_{\ell=0}^{n} g(A_\ell)
\]
by translates of the affine spaces
\[
A_\ell \ \cong\ \A^{\frac{n(n+1)}{2}},
\qquad 0\le \ell\le n,
\]
with regular coordinates
\[
\bigl\{
a_{\ell,i}^+\bigr\}_{n-\ell\le i\le n-1},
\quad
\bigl\{a_{\ell,i}^-\bigr\}_{\ell\le i\le n-1},
\quad
\bigl\{b_{\ell,1},\dots,b_{\ell,\frac{n(n-1)}{2}}\bigr\},
\]
satisfying the following properties:

\begin{itemize}
\item[-] For every $0\le i\le n-1$ the boundary divisors $D_i^+$ and
$D_i^-$ are smooth. On each $A_\ell$, for $n-\ell\leq i\leq n-1$ (resp. $\ell\leq i\leq n-1$), the intersection
$D_i^+\cap A_\ell$ (resp.\ $D_i^-\cap A_\ell$) is given by the vanishing
of a single coordinate $a_{\ell,i}^+$ (resp.\ $a_{\ell,i}^-$); otherwise, the intersection is empty.

\item[-] The divisor
$D := \bigcup_{i=0}^{n-1}(D_i^+\cup D_i^-)$ on each $A_\ell$ is given by the equation \[\prod_{i=n-\ell}^{n-1} a_{\ell,i}^+\cdot\prod_{i=\ell}^{n-1} a_{\ell,i}^- \ =\ 0.
\] 

\item[-] The remaining coordinates $b_{\ell,1},\dots,b_{\ell,\frac{n(n-1)}{2}}$ are
horizontal coordinates along the fibers of \eqref{pln}
in the sense that they are constant 
along the one–dimensional fibers of
$\PL_n$.
\end{itemize}

To investigate the $\mathbb C^*$--orbits on $\mathcal {TL}_{n}$, we next briefly describe the Bia{\l}ynicki-Birula decomposition on Lagrangian Grassmannians. Note that under the $\mathbb G_m$-action defined in (\ref{Gm}), the connected components of the fixed point scheme of $\mathrm{LG}(n,2n)$ are
\begin{equation*}
\begin{split}
&\mathcal V_{(n-k,k)} :=\left\{\left. \left(
\begin{matrix}
0&X\\
Y&0\\
\end{matrix}\right)\in \mathrm{LG}(n,2n)\right\vert_{}\footnotesize\begin{matrix}
X\,\,{\rm is\,\,a\,\,}k\times n\,\,{\rm matrix\,\,of\,\,rank}\,\,k\,\\
Y\,\,{\rm is\,\,an\,\,}(n-k)\times n\,\,{\rm matrix\,\,of\,\,rank}\,\,(n-k)\\
\end{matrix}
\right\}\,,\,\,\,\,0\leq k\leq n,\\
\end{split}
\end{equation*}
and the stable, unstable submanifolds are respectively
\begin{equation}\label{vani}
\begin{split}
&\mathcal V_{(n-k,k)}^+:= \left\{\left.\left(
\begin{matrix}
0&X\\
Y&W\\
\end{matrix}\right)\in \LG(n,2n)\right\vert_{}\footnotesize\begin{matrix}
X\,\,{\rm is\,\,a\,\,}k\times n\,\,{\rm matrix \,\,of\,\,rank\,\,}k\,\\
Y\,\,{\rm is\,\,an\,\,}(n-k)\times n\,\,{\rm matrix \,\,of\,\,rank\,\,}(n-k)\,\\
\end{matrix}\right\},\\
&\mathcal V_{(n-k,k)}^-:= \left\{\left.\left(
\begin{matrix}
Z&X\\
Y&0\\
\end{matrix}\right)\in \mathrm{LG}(n,2n)\right\vert_{}{\footnotesize\begin{matrix}
X\,\,{\rm is\,\,a\,\,}k\times n\,\,{\rm matrix \,\,of\,\,rank\,\,}k\,\\
Y\,\,{\rm is\,\,an\,\,}(n-k)\times n\,\,{\rm matrix \,\,of\,\,rank\,\,}(n-k)\,\\
\end{matrix}}\right\},    
\end{split}
\end{equation}
for $0\leq k\leq n$. It is clear that $Z^+_{n-k}=\overline{\mathcal V_{(n-k,k)}^+}$ and  $Z^-_{k}=\overline{\mathcal V_{(n-k,k)}^-}$.   

Since $\mathbb G_m={\rm Spec}\,\mathbb C[\lambda,\lambda^{-1}]$ in (\ref{Gm}) takes the form $a^-_{\ell,\ell}\mapsto \lambda^{-1} a^-_{\ell,\ell},\,\,a^+_{\ell,n-\ell}\mapsto \lambda a^+_{\ell,n-\ell}$ in the Mille Cr\^epes coordinate chart $A_{\ell}$, the connected components of the fixed point scheme of $\mathcal {TL}_{n}$ are 
\begin{equation*}
\mathcal D_{(n-k,k)}:=(\RL_{n})^{-1}(\mathcal V_{(n-k,k)}), \,\,0\leq k\leq n,
\end{equation*}
with the stable, unstable subschemes given by
\begin{equation*}
\mathcal D^{\pm}_{(n-k,k)}:=(\RL_{n})^{-1}(\mathcal V^{\pm}_{(n-k,k)}).
\end{equation*}
It is clear that $\mathcal D_{(n,0)}$ and $\mathcal D_{(0,n)}$ are smooth divisors of $\mathcal {TL}_{n}$, and  $\mathcal D_{(n-k,k)}$ are codimension--$2$ smooth closed subschemes for $1\leq k\leq n-1$;  $D_{n-k}^+=\overline{\mathcal D_{(n-k,k)}^+}$ and $D_{k}^-=\overline{\mathcal D_{(k,n-k)}^-}$.

\subsection*{Rational normal curves in Lagrangian Grassmannians}
To identify $\TL_n$ with a subvariety of the Kontsevich moduli space, we study the degree $n$ rational curves in $\LG(n,2n)$ and its Hilbert moduli spaces. We first prove the following lemma by exploiting the Mille Cr\^epes coordinates.
\begin{lemma}\label{moduli1l}
For every closed point $q\in\CQ_n$, the fiber $Z_q := \PL_n^{-1}(q)$ has the following properties.
 
\begin{enumerate}[label=(\alph*)]

\item $Z_q$ consists of a chain of $\mathbb G_m$-stable smooth rational curves from $p_+$ to $p_-$.

\item The morphism $\RL_n$ restricts to an embedding $Z_q\hookrightarrow \LG(n,2n)$.

\item Generically $\RL_{n}(Z_q)$ is a smooth rational curve of degree $n$ with respect to the Pl\"ucker embedding.

\item For any closed points $q,q^{\prime}\in \CQ_{n}$,  $\RL_{n}(Z_q)=\RL_{n}(Z_{q^{\prime}})$ if and only if $q=q^{\prime}$.

\end{enumerate} 
\end{lemma}
\begin{proof} 
The proof is the same as that of \cite[Lemma 4.22]{FW}. Consider the Mille Cr\^epes coordinate chart $A_{\ell}$ with $0\leq \ell\leq n$. Computation yields that when $\ell=0$ (resp. $n$), any fiber $Z_q$ restricted to $A_{\ell}$ is defined by fixing all the variables except $a^-_{0,0}$ (resp. $a^+_{n,0}$); when $1\leq \ell\leq n-1$, $Z_q\cap A_{\ell}$ is defined by fixing $a^+_{\ell,n-\ell}\cdot a^-_{\ell,\ell}$ and all the other variables. We conclude Property (a).

Recall that $\RL_n(Z_q\cap A_{\ell})$ is given by the projection to $\mathbb P^{N_{n}}$ via the Pl\"ucker coordinate functions $P_I$ (see (\ref{fskl})). By computing the corresponding determinants in each coordinate charts $A_{\ell}$, we  conclude Properties (b). Moreover,  in the coordinate chart $A_{0}$, computation yields that for generic $Z_q$, all $P_I(Z_q)$ are polynomials in $a^-_{0,0}$, there is one taking value $1$, and the highest degree of $a^-_{0,0}$ is $n$. Property (c) follows.

Suppose that $\RL_{n}(Z_q)=\RL_{n}(Z_{q^{\prime}})$. We can write $Z_q$ as a chain of a rational curves $\cup_{i=1}^m\gamma_i$ with integers $0< k_1<\cdots<k_{m-1}<n$ such that 
$$\gamma_{1}\subset \overline{\mathcal D_{(n-k_1,k_1)}^-},\,\,\gamma_{i}\subset \overline{\mathcal D_{(n-k_i,k_i)}^-}\cap\overline{\mathcal D_{(n-k_{i-1},k_{i-1})}^+}\,\,{\rm for\,\,}2\leq i\leq m-1,\,\,{\rm and\,\,}\gamma_{m}\subset \overline{\mathcal D_{(n-k_{m-1},k_{m-1})}^+}.$$
Notice that the rational map $f^k$ given by (\ref{fk}) is well--defined on $\gamma_i$ for any $1\leq i\leq m$ and $k_{i-1}\leq k\leq k_i$. We can thus conclude that $q=q^{\prime}$. Property (d) follows.
\end{proof}

Recall that for a projective variety $X$ with an algebraic group action $\mathscr G$, there is an open set $U \subset X$ such that the orbit closures of points in $U$ form a flat family of subschemes of $X$. Hence, there is a morphism from $U$ to ${\rm Hilb}(X)$, the Hilbert scheme of $X$. Following \cite{BS,Kap}, the closure of the image of $U$ in ${\rm Hilb}(X)$ is called the {\it Hilbert quotient} and denoted by $X\! /\!/\mathscr G$. Thaddeus \cite{Th1} proved that the Hilbert quotient $\LG(W \oplus W^\vee)\! /\!/\mathbb C^*$ is isomorphic to $\CQ_n$.
We establish that
\begin{proposition}\label{modulil}  $\PL_{n}\colon\TL_n\rightarrow\CQ_n$ is the universal family over the Hilbert quotient $\LG(W \oplus W^\vee)\! /\!/\mathbb C^*\cong\CQ_n$.
\end{proposition}
\begin{proof}
Let $X:=\LG(n,2n)$ and let $X/\!/\mathbb C^*$ be its Hilbert quotient, i.e.\ the closure of the generic orbit closures inside the Hilbert scheme.
By Lemma~\ref{moduli1l}\textup{(b)} the fibers of $\PL_n$ map isomorphically to $\mathbb C^*$--orbit closures in $X$, hence $\PL_n\colon \TL_n\to\CQ_n$ defines a flat family of subschemes of $X$ parametrized by $\CQ_n$.
By the universal property of the Hilbert scheme (and of the Hilbert quotient) this yields a morphism
\[
j\colon \CQ_n \longrightarrow X/\!/\mathbb C^*
\]
such that $\PL_n$ is the pullback of the universal family along $j$.

Property \textup{(d)} in Lemma~\ref{moduli1l} shows that two points of $\CQ_n$ give distinct orbit closures in $X$, hence $j$ is injective on closed points; since $X/\!/\mathbb C^*$ is defined as the closure of such orbit closures, $j$ is also surjective on closed points. In particular, $j$ is bijective.

The Hilbert quotient $X/\!/\mathbb C^*$ is projective by construction, hence $j$ is proper. Since $j$ is bijective, it is quasi-finite; therefore $j$ is finite. Finally, $X/\!/\mathbb C^*$ is smooth (in particular normal) by Thaddeus' identification $X/\!/\mathbb C^*\cong \CQ_n$ \cite{Th1}. Thus, by Zariski's Main Theorem, $j$ is an isomorphism and $\PL_n$ is the universal family over the Hilbert quotient.
\end{proof}

It is a classical result that any irreducible, nondegenerate rational curve of degree $n$ in $\PP^n$ is projectively equivalent to the Veronese embedding. We derive  similar normal forms for Lagrangian Grassmannians as follows. 
\begin{lemma}\label{2HM}
Each irreducible smooth degree $n$ rational curve passing through $p_+$ and $p_-$ is invariant under the $\mathbb C^*$--action defined by (\ref{Gm}). 

\end{lemma}

\begin{proof}
Let $\gamma$ be such a curve. In the standard local coordinate chart around $p_-$, we may write 
\begin{equation}\label{fij}
\gamma(t)=\begin{pmatrix}
1&0&\cdots&0&f_{11}(t)&f_{12}(t)&\cdots& f_{1n}(t)\\   
0&1&\cdots&0&f_{21}(t)&f_{22}(t)&\cdots& f_{2n}(t)\\ 
&&\cdots&&&&\cdots&\\
0&0&\cdots&1&f_{n1}(t)&f_{n2}(t)&\cdots& f_{nn}(t)\\ 
\end{pmatrix},    
\end{equation}
where $f_{ij}=f_{ji}$ are rational functions with $f_{ij}(0)=0$. Since $\mathcal O_{\LG(n,2n)}\cdot\gamma=n$, the Pl\"ucker coordinate function  $P_{I^*}(\gamma(t))$ for $I^*=(n+1,n+2,\dots,2n)$ vanishes only at $t=0$, with multiplicity exactly $n$. We conclude that $\gamma(\mathbb P^1-\{t=0\})$ is entirely contained in the standard local coordinate chart $U_{12\cdots n}$ around $p_+$ (see (\ref{ull})). Therefore, after a change of coordinates of $\gamma(t)$, we can assume that $f_{ij}$ in (\ref{fij}) are  polynomials without constant terms.
An identical argument with the roles of 
$p_+$ and $p_-$   swapped shows that,  for any $I\subset\{1,2,\cdots,2n\}$ with $|I|=n$ and $I\neq I^*$, the polynomial $P_{I}(\gamma(t))$ is of degree $\leq n-1$.

In the following, we shall prove that all
$f_{ij}(t)=a_{ij}\cdot t$ with $a_{ij}\in\mathbb C$, based on the above properties of the Pl\"ucker coordinate functions $P_{I}(\gamma(t))$.

Expand $f_{1j}(t)=\sum_{k=1}^na_{1j;k}t^k$ for $1\leq j\leq n$, where $a_{1j;k}\in\mathbb C$.
Without loss of generality, we may assume that $a_{11;1}\neq 0$. Notice that for each $1\leq j\leq n$, the minor of $(f_{ij})$ deleting the $1$--st row and the $j$--th column equals $b_{j} t^{n-1}$, and at least one $b_j\in\mathbb C$  nonzero. Moreover,  by the Laplace expansion,
$$\sum\nolimits_{j=1}^nb_j\left(\sum\nolimits_{k=2}^na_{1j;k}t^k\right)=0,$$
for $\det(f_{ij}(t))$ is of degree $n$.
If $b_2=\cdots=b_n=0$, it is clear that $b_1\neq 0$. It follows that the column vector $(f_{i1}(t))_{2\leq i\leq n}$ is zero. This is seen by expanding it in the basis $(f_{i2}(t))_{2\leq i\leq n}, (f_{i3}(t))_{2\leq i\leq n},\cdots, (f_{in}(t))_{2\leq i\leq n}$ and computing the minors. Then Lemma \ref{2HM} follows by induction.

Otherwise, we may assume that $b_n\neq0$. After a congruence transformation by a constant matrix $C$, namely, $(\widetilde f_{ij}):=C^{\!\top}\cdot(f_{ij})\cdot C$,
Such congruence transformations correspond to changing the basis of $W$ (hence to the natural $\GL(W)$--action on the symmetric matrix chart), so they preserve the Lagrangian condition and the incidence with $p_\pm$.
 we can have $\widetilde f_{1n}(t)=\widetilde a_{1n}t$. If $\widetilde a_{1n}=0$, we may apply the same argument repeatedly to eliminate the higher--order terms of $f_{1(n-1)}, f_{1(n-2)},\cdots$  until we reach an index $j$ for which $\widetilde f_{1j}(t)\not\equiv  0$. 
If $j=1$, we have the above case. Otherwise, we may just assume that $f_{1n}(t)=t$. Up to congruence, we further assume that $a_{j1;1}=0$, $2\leq j\leq n-1$. Since $\det(f_{ij}(t))=(f_{11}(t)f_{nn}(t)-t^2)\det(f_{ij}(t))_{2\leq i,j\leq n-1}$ is of degree $n$, $f_{11}(t)=a_{11;1}t$. After a congruence transformation, we conclude Lemma \ref{2HM}. 
\end{proof}

\begin{lemma}\label{gspp}
Let $\gamma$ be an irreducible smooth degree $n$ rational curve on $\LG(n,2n)$.
\begin{enumerate}[label=(\alph*)]
    \item Suppose that $\gamma$ passes through $p_+$ and $p_-$. Then for any two distinct points  $p_1,p_2\in\gamma$, there exists an automorphism $\tau$ of $\LG(n,2n)$ such that $\tau(p_1)=p_+$ and $\tau(p_2)=p_-$. 

    \item Suppose that $\gamma$ passes through $p_-$. If there is a point $p_2\in\gamma$, $p_2\neq p_-$, such that no automorphism $\tau$ of $\LG(n,2n)$ satisfies $\tau(p_-)=p_-$ and $\tau(p_2)=p_+$,  then 
\begin{equation}\label{vanishing}
P_{I^*}(\gamma(t))\equiv0\,\,\forall\,t\in\mathbb C,\,\,{\rm where}\,\,I^*=(n+1,n+2,\cdots,2n).
\end{equation}
In particular, no two points on $\gamma$ can be sent to $(p_+,p_-)$ under an automorphism of $\LG(n,2n)$.
\end{enumerate}

\end{lemma}
\begin{proof} To verify Property (a), it suffices to take $p_1=p_+$. Lemma \ref{2HM} allows us to write $\gamma(t)$ as in (\ref{fij}) with linear entries $f_{ij}(t)=a_{ij}t$. Take $x$ such that $\gamma(x)=p_2$, and define $F=(f_{ij}(x))$. The matrix $\tau=\begin{pmatrix}
        I&-F\\
        0&I
    \end{pmatrix}$
then provides the desired automorphism.

If $P_{I^*}(\gamma(x))\neq0$, there is an automorphism $\sigma$ such that $\sigma(p_-)=p_-$ and $\sigma(\gamma(x))=p_+$. Then by Property (a), there is an automorphism of $\tau$ such that $\tau(p_-)=p_-$ and $\tau(p_2)=p_+$, which is a contradiction.
\end{proof}

Let  $\mathrm{Hilb}_{nt+1}(\LG(n,2n))$ be the Hilbert scheme associated to the Hilbert polynomial $nt+1$, which is irreducible and smooth
by \cite{Pe02}. Denote by $\mathcal H_{p_+,\, p_-}\subset\mathrm{Hilb}_{nt+1}(\LG(n,2n))$ the subscheme parametrizing the  curves passing through the points $p_+$, $p_-$, with the reduced scheme structure. 
\begin{proposition}\label{density}
$\mathcal H_{p_+,\, p_-}$ is the Hilbert quotient $\LG(n,2n)\! /\!/\mathbb C^*$. $\mathcal H_{p_+,\, p_-}$ contains an open and dense subset
\begin{equation*}
 \mathcal H^0_{p_+,\, p_-}:=\{\text {irreducible smooth curves passing through}\,\,p_+,p_-\}.     
\end{equation*}   
\end{proposition}
\begin{proof}
It is clear that the Hilbert quotient $\LG(n,2n)\! /\!/\mathbb C^*$ is contained in $\mathcal H_{p_+,\, p_-}$ by Lemma \ref{moduli1l} and Proposition \ref{modulil}. 
Note that the fiber over a generic point of $\LG(n,2n)\! /\!/\mathbb C^*$ is an irreducible smooth curve. Then by Lemma \ref{2HM},  the only irreducible component of $\mathcal H_{p_+,\, p_-}$ that contains $\LG(n,2n)\! /\!/\mathbb C^*$ is $\LG(n,2n)\! /\!/\mathbb C^*$ itself.

It suffices to show that  $\LG(n,2n)\! /\!/\mathbb C^*$ is 
the only irreducible component of $\mathcal H_{p_+,\, p_-}$.  Let
$\mathrm{Hilb}^0_{nt+1}(\mathrm{LG}(n,2n))$ denote the open dense subset of 
$\mathrm{Hilb}_{nt+1}(\mathrm{LG}(n,2n))$ consisting of irreducible smooth curves. Define
\begin{equation*}
\mathrm{Hilb}^{00}_{nt+1}(\mathrm{LG}(n,2n)):=\left\{\begin{matrix}
 \text {irreducible smooth rational curves of degree}\,\,n\,\,\text{that passes }\\
 \text{ through}\,\,p_+\,\,\text{and}\,\,p_-\,\,\text{up to an automorphism of}\,\,\LG(n,2n)   
\end{matrix}\right\}.  
\end{equation*}

\begin{claim}\label{verygeneral}
$\mathrm{Hilb}^{00}_{nt+1}(\mathrm{LG}(n,2n))$ is an open and dense subset of $\mathrm{Hilb}_{nt+1}(\mathrm{LG}(n,2n))$.   
\end{claim}
\begin{proof}[Proof of Claim~\ref{verygeneral}]
For each $y\in \mathrm{Hilb}^0_{nt+1}(\mathrm{LG}(n,2n))$, let $\gamma_y$ the corresponding curve over $y$.
For any $x\in\mathrm{Hilb}^{00}_{nt+1}(\mathrm{LG}(n,2n))$, we can take an open Euclidean ball $B_x\ni x$ and two sections $\iota_1,\iota_2\colon B_x\rightarrow\LG(n,2n)$ such that for each $y\in B_x$, $\iota_1(y)$ and $\iota_2(y)$ are distinct points on $\gamma_y$. After shrinking $B_x$ if necessary, we have an algebraic family $\tau(\cdot,\cdot)\colon B_x\times\LG(n,2n)\rightarrow\LG(n,2n)$ of automorphisms of $\LG(n,2n)$ such that 
$\tau(y,\iota_1(y))=p_+$ for each $y\in B_x$.

Denote by $\pi\colon\mathcal U\rightarrow \mathrm{Hilb}_{nt+1}(\mathrm{LG}(n,2n))$
the universal family. By (\ref{vanishing}), we define a subvariety  
\begin{equation*}
\mathcal P_x:=\{(y,z)\in B_x\times\LG(n,2n)|\,P_{I^*}(\tau(y,z))=0\},    
\end{equation*}
which is just the image of the product $B_x\times\{z\in\LG(n,2n)|P_{I^*}(z)=0\}$ under the family of automorphisms $\tau^{-1}$. Now, consider the restriction of $\pi$:
\begin{equation*}
\pi|_{\pi^{-1}(B_x)\cap\mathcal P_x}\colon\pi^{-1}(B_x)\cap\mathcal P_x\rightarrow B_x.   
\end{equation*}
The locus
$\mathcal S_x:=\{z\in B_x|\dim\left(\pi|_{\pi^{-1}(B_x)\cap\mathcal P_x}\right)^{-1}(z)\geq 1\}$ is a complex analytic subset of $B_x$. Notice that by Lemma \ref{gspp}, on $B_x\cap\mathrm{Hilb}^{0}_{nt+1}(\mathrm{LG}(n,2n))$, $\mathcal S_x$ is the complement of $\mathrm{Hilb}^{00}_{nt+1}(\mathrm{LG}(n,2n))$. Moreover, the sets $\mathcal S_x$ patch together as a Zariski closed subset of $\mathrm{Hilb}_{nt+1}(\mathrm{LG}(n,2n))$. The proof is complete.
\end{proof} 
Suppose that there is another irreducible component $X$ of $\mathcal H_{p_+,\, p_-}$.  We equip $\mathrm{Hilb}_{nt+1}(\mathrm{LG}(n,2n))$ and $\LG(n,2n)$ with certain Riemannian metrics.
Take $x\in X$ with $\operatorname{dist}(x,\LG(n,2n)\! /\!/\mathbb C^*)>2\epsilon>0$. 
By Claim \ref{verygeneral}, there is a sequence of points $y_1,y_2,\cdots\in\mathrm{Hilb}^{00}_{nt+1}(\mathrm{LG}(n,2n))$ such that $\lim_{i=1}^{\infty}y_i=x$. Since $\pi$ is proper and open, we have
$$\lim_{i\rightarrow\infty}\max_{z\in\gamma_{x}}\{\operatorname{dist}(z,\gamma_{y_i})\}=0.$$

Now for any $\epsilon>0$, there is an $N>0$, such that for any $i\leq N$,
we have $\max_{z\in\gamma_{x}}\{\operatorname{dist}(z,\gamma_{y_i})\}<\epsilon$. In particular, for any $i\leq N$, we can find points $p_1,p_2\in\gamma_i$ such that $\operatorname{dist}(p_1,p_+)<\epsilon$ and $\operatorname{dist}(p_2,p_-)<\epsilon$.
In the standard coordinate charts (\ref{ulxl}) around $p_+$ and $p_-$, we may write $p_1=(E_1\,\vdashline\,I_{n})$ and 
$p_2=(I_{n}\,\vdashline\,E_2)$,
respectively. Note that $E_1$ and $E_2$ are $n\times n$ matrices with the matrix norms $\|E_1\|,\|E_2\|<C\cdot \epsilon$, where $C$ is a constant depending only on the Riemann metric on $\LG(n,2n)$. Define an automorphism $\tau_1$ of $\LG(n,2n)$ by
\begin{equation*}
\tau_1:=\begin{pmatrix}
        I_n&-E_2\\
        0&I_n
    \end{pmatrix}.
\end{equation*}
Then $\tau_1(p_2)=p_-$ and $\tau_1(p_1)=\left(E_1\,\vdashline\,I_n-E_1E_2\right)$. Define
\begin{equation*}
\tau_2:=\begin{pmatrix}
        I_n&0\\
        -(I_n-E_1E_2)^{-1}E_1&I_n
    \end{pmatrix}.
\end{equation*}
It is clear that $\tau:=\tau_2\circ\tau_1$
transform $(p_1,p_2)$ to $(p_+,p_-)$, with $\|\tau-I_{2n}\|<D\cdot\epsilon$ for a universal constant $D$. Notice that the automorphism group of $\LG(n,2n)$ has an algebraic action on $\mathrm{Hilb}_{nt+1}(\LG(n,2n))$; in particular, this action is continuous. Then for $i$ large enough, there exists an automorphism $\tau$ of $\LG(n,2n)$, which is very close to the identity, such that $\tau(\gamma_{y_i})$ passing through $p_+,p_-$ and $\tau(y_i)\notin\LG(n,2n)\! /\!/\mathbb C^*$, which is a contradiction.
\end{proof}

\section{Geometry of \texorpdfstring{$\TL_n$}{TL\_n} as Spherical Varieties}\label{Sec2/2}

In this section, we study the numerical geometry of $\TL_n$ as spherical varieties.

The group
$\mathscr G=\GL(W^\vee)\times\Gm$
acts on $W\oplus W^\vee$ by (\ref{Gm}),
and hence on $\LG(n,2n)$ and on $\TL_n$.
We take as maximal torus
\begin{equation}\label{MaxTor}
T:=\Bigl\{\bigl(\mathrm{diag}(t_1,\dots,t_n),\lambda\bigr)
\ \Big|\ t_1,\dots,t_n,\lambda\in\Gm\Bigr\}\subset\mathscr G.
\end{equation}
In the Mille Cr\^epes coordinates, the variables  $a_{\ell,i}^\pm$ and
$b_{\ell,j}$ are $T$--eigenfunctions, and thus $T$ acts diagonally on each
chart $A_\ell$.
In particular, every coordinate axis is a $T$--invariant affine line, and
its closure in $\TL_n$ is a $T$–invariant irreducible curve meeting the
open $\mathscr G$--orbit.


We briefly recall how the $\mathscr{B}$--stable divisors on $\TL_n$ are produced from Plücker coordinates on $\LG(n,2n)$. Fix a complete flag on $W^\vee$ and let $B\subset\GL(W^\vee)$ be the corresponding
Borel subgroup of upper--triangular matrices; we denote by $\mathscr{B}\subset\mathscr G=\GL(W^\vee)\times\Gm\subset \Sp(V)$ the product of this Borel with the $\Gm$--factor.
In Plücker coordinates on $\LG(n,2n)$ there is a distinguished $(n+1)$--tuple of
$\mathscr{B}$--eigenvectors $P_{I_0},P_{I_1},\dots,P_{I_n}$, corresponding to the multi–indices
\[
I_k=(k+1,k+2,\dots,n+k),\qquad 0\le k\le n,
\]
with respect to the chosen flag.  For $0\le k\le n$, we set
\[
b_k\ :=\ \{L\in\LG(n,2n)\mid P_{I_k}(L)=0\},
\]
and denote by $B_k\subset\TL_n$ the strict transform of $b_k$ under $\RL_n$. 


\subsection{Cones of divisors, positivity of anticanonical bundle, and automorphisms} 
\begin{proposition}\label{prop:picard-effective}
The Picard group of $\TL_n$ is freely generated by
\[
H:=(\RL_n)^*\OO_{\LG(n,2n)}(1),\
D_0^+,\dots,D_{n-2}^+, D_0^-,\dots,D_{n-2}^-.
\]
The boundary divisors are
$D_0^+,\dots,D_{n-1}^+, D_0^-,\dots,D_{n-1}^-$.
Moreover, the colors of $\TL_n$ are $B_1,\dots,B_{n-1}$. Finally, the following linear equivalence relations hold:
\begin{equation}\label{eq:boundary-last}
B_0=D_{n-1}^+=H-\sum\nolimits_{i=0}^{\,n-2}(n-i)\,D_i^+,\qquad
B_n=D_{n-1}^-= H-\sum\nolimits_{i=0}^{\,n-2}(n-i)\,D_i^-,  
\end{equation}
and for $1\le k\le n-1$,
\begin{equation}\label{eq:boundary-mid}
B_k=H-\sum\nolimits_{i=0}^{\,n-k-1}\!(n-k-i)\,D_i^+
-\sum\nolimits_{i=0}^{\,k-1}\!(k-i)\,D_i^-.
\end{equation}
\end{proposition}
\begin{proof}
By Proposition~\ref{prop:TLn-blowup-equals-Kausz}, the morphism $\RL_n\colon\TL_n\longrightarrow\LG(n,2n)$ is obtained by blowing up the strict transforms of $Z_k^\pm$,  all of which are smooth. Thus the divisors listed in the statement form a free $\Z$–basis of $\Pic(\TL_n)$, and the $2n$ divisors $D_i^\pm$, $0\le i\le n-1$, are exactly the irreducible $\mathscr G$–invariant boundary components of $\TL_n$. This shows that the $\mathscr{B}$--stable but not $\mathscr G$–stable divisors are exactly $B_1,\dots,B_{n-1}$.
\end{proof}

\begin{corollary}\label{cor:toroidal-TLn}
For every $n\ge 1$, the compactification $\TL_n$ is a smooth projective spherical toroidal variety.
\end{corollary}

\begin{proof}
Sphericity comes from Theorem \ref{thm:TLn-spherical}. Noticing that every $\mathscr{G}$--invariant closed subvariety has an open subset inside some Mille Cr\^epes coordinate chart $A_{\ell}$ (up to automorphisms of $\TL_n$), we then argue locally. 
On the other hand, computation yields that $B_k\cap A_{\ell}=\emptyset$ for $1\leq k\leq n-1$. Therefore, $\TL_n$ is toroidal.
\end{proof}

\begin{corollary}\label{effcone}
The effective cone $\Eff(\TL_n)$ is generated by the classes $D_0^+,\dots,D_{n-1}^+,\ D_0^-,\dots,D_{n-1}^-$.
\end{corollary}
\begin{proof}
On a smooth complete spherical variety the effective cone is generated by
$\mathscr{B}$-–stable prime divisors \cite[Section 2.6]{Br2}.

We now work out explicit the expression of the colors $B_k$ as nonnegative
linear combinations of the boundary components $D_0^+,\dots,D_{n-1}^+, D_0^-,\dots,D_{n-1}^-$.
We have $B_0=D_{n-1}^+, B_n=D_{n-1}^-$, so $B_0$ and $B_n$ are already boundary components. For
$1\le k\le n-1$ we use the following two relations in $\Pic(\TL_n)$:
\begin{align*}
H=D_{n-1}^+ + \sum\nolimits_{i=0}^{n-2}(n-i)\,D_i^+,\quad
H=D_{n-1}^- + \sum\nolimits_{i=0}^{n-2}(n-i)\,D_i^-,
\end{align*}
coming from the identities $B_0=D_{n-1}^+$ and $B_n=D_{n-1}^-$. For a fixed $k$ with $1\le k\le n-1$ we write $H$ as the convex
combination
\[
H=\frac{n-k}{n}
   \Bigl(D_{n-1}^+ + \sum\nolimits_{i=0}^{n-2}(n-i)\,D_i^+\Bigr)
 + \frac{k}{n}
   \Bigl(D_{n-1}^- + \sum\nolimits_{i=0}^{n-2}(n-i)\,D_i^-\Bigr),
\]
which is valid because both parentheses represent the same class $H$. Substituting this into the expression for $B_k$ in Proposition \ref{prop:picard-effective} and collecting
coefficients of $D_i^\pm$ we obtain
\begin{equation*}
B_k=
\sum\nolimits_{i=0}^{n-1}\alpha_{i,k}^+\,D_i^+
+
\sum\nolimits_{i=0}^{n-1}\alpha_{i,k}^-\,D_i^-,
\qquad 1\le k\le n-1,
\end{equation*}
where the coefficients $\alpha_{i,k}^\pm$ are given explicitly as follows: for $0\le i\le n-1$,
\[
\alpha_{i,k}^+
 =
\begin{cases}
\dfrac{ik}{n} & \text{if } i\le n-k-1\\[6pt]
\dfrac{(n-k)(n-i)}{n} & \text{if } i\ge n-k
\end{cases}\,\,\,{\rm and}\,\,\,\,
\alpha_{i,k}^-
 =
\begin{cases}
\dfrac{i(n-k)}{n} & \text{if } i\le k-1\\[6pt]
\dfrac{k(n-i)}{n} & \text{if } i\ge k
\end{cases}.
\]
Then, the colors lie in the cone generated by the boundary divisors.
\end{proof}


Consider the torus $T$ in~(\ref{MaxTor}). In the Mille Cr\^epes charts, the closure of any coordinate axis is a $T$--invariant curve. We now single out three families of $T$--invariant curves
whose numerical classes generate the Mori cone and will be used throughout.

\begin{definition}\label{def:T-curves-correct}
For $0\le \ell\le n-1$, let $\dot\gamma_\ell\subset A_\ell$ be the
coordinate axis
\[
\dot\gamma_\ell(t)\ :=\ \bigl(a^-_{\ell,\ell}=t,\ \text{all other coordinates }=0\bigr),
\qquad t\in\C,
\]
and denote by $\gamma_\ell\subset \TL_n$ its closure. Equivalently, via the iterated blow--up model of $\TL_n$, $\gamma_\ell$ is the
strict transform of the corresponding coordinate line in the original ambient
projective space; in particular it joins two adjacent fundamental points.

For $1\le j\le n-1$, let $\dot\zeta_j^+\subset A_n$ be the coordinate axis 
\[ \dot\zeta_j^+(t)\ :=\ \bigl(a^+_{n,j}=t,\ \text{all other coordinates }=0\bigr), \qquad t\in\C, \]
and let $\zeta_j^+\subset \TL_n$ be its closure. Here $a^+_{n,j}$ is the Mille Crêpes coordinate that vanishes simply along $D_j^+$. Geometrically, $\zeta_j^+$ is a line in a fiber of the contraction whose exceptional locus is contained in $D_{j-1}^+$, hence it is contracted by the morphism $\RL_n$.



For $1\le j\le n-1$, similarly let $\dot\zeta_j^-\subset A_0$ be the
coordinate axis \[\dot\zeta_j^-(t)\ :=\ \bigl(a^-_{0,j}=t,\ \text{all other coordinates }=0\bigr), \qquad t\in\C,\] with closure $\zeta_j^-\subset \TL_n$. 
$\zeta_j^-$ is 
also contracted by $\RL_n$.
\end{definition}

Each of $\gamma_\ell,\zeta_j^+,\zeta_j^-$ is the closure of a $1$--dimensional
$T$--orbit. We denote their numerical classes by
\[
C_\ell:=[\gamma_\ell],
\qquad
C_j^+:=[\zeta_j^+],
\qquad
C_j^-:=[\zeta_j^-]
\ \in\ N_1(\TL_n)_\R.
\]
\begin{lemma}\label{lem:intersection-table-correct}
In the above notation, the intersection numbers of the divisors
$H,D_i^\pm$ with the curves $C_\ell,C_j^+,C_j^-$ are as follows.
\[
\begin{array}{llll}
H\cdot C_\ell = 1,&
D_i^-\cdot C_\ell = \delta_{i,\ell}-\delta_{i,\ell+1},&
D_i^+\cdot C_\ell = \delta_{i,n-\ell-1}-\delta_{i,n-\ell},
\end{array}
\]
for $0\le i\le n-1,\ 0\le \ell\le n-1$.
\[
\begin{array}{llll}
H\cdot C_j^+ = 0,& D_i^+\cdot C_j^+ = -\delta_{i,j-1}+2\delta_{i,j}-\delta_{i,j+1},& D_i^-\cdot C_j^+ = 0,\\
H\cdot C_j^- = 0,& D_i^-\cdot C_j^- = -\delta_{i,j-1}+2\delta_{i,j}-\delta_{i,j+1},& D_i^+\cdot C_j^- = 0,
\end{array}
\]
for $0\le i\le n-1,\ 1\le j\le n-1$, where $\delta_{ab}$ denotes the Kronecker delta.
\end{lemma}

\begin{proof}
The proof parallels that of \cite[Lemma 3.13]{F26} and is carried out in the Mille-Crêpes coordinates. For simplicity, we omit it here. 
\end{proof}

\begin{proposition}\label{prop:Mori-generators-correct}
The cone of effective curves $\NE(\TL_n)\subset N_1(\TL_n)_\R$ is generated by
the classes
\begin{equation}\label{zetagamma}
\{C_\ell\}_{0\le\ell\le n-1}\ \cup\ \{C_j^+\}_{1\le j\le n-1}\ \cup\
\{C_j^-\}_{1\le j\le n-1}.
\end{equation}
\end{proposition}

\begin{proof}
By \cite[Proposition~2.3 and Corollary~2.4]{Br5}, $\NE(\TL_n)$ is generated by classes of closures of $1$--dimensional $\mathscr{B}$--orbits (for any Borel subgroup $\mathscr{B}\subset \mathscr{G}$). Every $1$--dimensional $\mathscr{B}$--orbit contains a unique closed $T$--orbit, hence $\NE(\TL_n)$ is generated by classes of closures of $1$--dimensional $T$--orbits, i.e.\ by $T$--invariant curves. A direct inspection shows that these axes are exactly the three families listed in Definition~\ref{def:T-curves-correct}.
\end{proof}


\begin{proposition}\label{prop:movcone-TLn-correct}
The moving cone of curves on $\TL_n$ is
\[
\Mov_1(\TL_n)
\ =\
\overline{\Eff}(\TL_n)^\vee
\ =\
\bigl\{\alpha\in N_1(\TL_n)_\R\ \big|\ \alpha\cdot D_i^\pm\ge 0\ \forall\,0\le i\le n-1\bigr\}.
\]
In particular, $\Mov_1(\TL_n)$ is a rational polyhedral cone.

Moreover, its extremal rays are indexed by pairs $(p,q)$ with $0\le p,q\le n-1$
and are spanned by the classes $\widetilde W_{p,q}\in N_1(\TL_n)_\R$ uniquely characterized by
\[
H\cdot \widetilde W_{p,q}=1,\qquad
D_i^+\cdot \widetilde W_{p,q}=\frac{\delta_{ip}}{n-p},\qquad
D_i^-\cdot \widetilde W_{p,q}=\frac{\delta_{iq}}{n-q}
\quad (0\le i\le n-1).
\]
Equivalently, if we set $w_i:=n-i$ and $l_{p,q}:=\lcm(w_p,w_q)$ and define the
integral class $W_{p,q}:=l_{p,q}\,\widetilde W_{p,q}$, then
\[
H\cdot W_{p,q}=l_{p,q},\qquad
D_p^+\cdot W_{p,q}=\frac{l_{p,q}}{w_p},\qquad
D_q^-\cdot W_{p,q}=\frac{l_{p,q}}{w_q},
\]
and all other intersections with $D_i^\pm$ vanish.

Finally, each $W_{p,q}$ is an effective combination of the Mori generators.
One convenient explicit expression is the following: set
\[
a_{p,q}:=\frac{l_{p,q}}{w_p},\qquad b_{p,q}:=\frac{l_{p,q}}{w_q},
\]
define integers $\lambda^{p,q}_r$ for $0\le r\le n-2$ by the recursion
\begin{equation}\label{eq:lambda-recursion}
\lambda^{p,q}_{-1}:=0,\qquad
\lambda^{p,q}_0:=a_{p,q}-b_{p,q}\,\delta_{q0},
\qquad
\lambda^{p,q}_r:=2\lambda^{p,q}_{r-1}-\lambda^{p,q}_{r-2}
-a_{p,q}\,\delta_{r,n-p}-b_{p,q}\,\delta_{rq}
\ \ (1\le r\le n-2),
\end{equation}
and then set
\begin{equation}\label{eq:Mpq-effective}
W_{p,q}
\ :=\
a_{p,q}\sum\nolimits_{\ell=0}^{n-1-p} C_\ell
\ +\
\sum\nolimits_{j=1}^{n-1}\lambda^{p,q}_{j-1}\,C_j^-.
\end{equation}
Moreover, there is a symmetric formula using the $C_j^+$ instead of the $C_j^-$.
\end{proposition}
\begin{proof}
By Remark~\ref{rem:bdpp}, $\Mov_1(\TL_n)=\overline{\Eff}(\TL_n)^\vee$ since $\TL_n$ is smooth. By Corollary~\ref{effcone}, $\overline{\Eff}(\TL_n)$ is generated by the boundary divisors $D_i^\pm$ $(0\le i\le n-1)$, hence the displayed inequalities describe $\Mov_1(\TL_n)$. To see that $W_{p,q}=l_{p,q}\widetilde W_{p,q}$ is effective and equals the explicit combination~\eqref{eq:Mpq-effective}, observe first that, by Lemma~\ref{lem:intersection-table-correct}, the partial sum $\sum_{\ell=0}^{n-1-p} C_\ell$ has $D_i^+$--intersection equal to $\delta_{ip}$, and has $H$--intersection equal to $w_p=n-p$. Since the boundary divisors span $N^1(\TL_n)_\R$, both $\beta_1$ and $\beta_2$ are proportional to $\widetilde W_{p,q}$.
\end{proof}



\begin{proposition}\label{prop:nef-cone}
For integers $p,q$ with $0\le p,q\le n-1$ and $p+q\le n$, define the divisor class $N_{p,q}$ by (\ref{npq}). Then $N_{p,q}$ is nef. Moreover, $\Nef(\TL_n)\subset \Pic(\TL_n)_\R$ is a rational polyhedral cone, with extremal rays precisely
\[
\R_{\ge 0}\,N_{p,q}
\qquad\text{for}\qquad
0\le p,q\le n-1,\ \ p+q\le n.
\]
In particular, $\Nef(\TL_n)$ has $\frac{n^2+3n-2}{2}$ extremal rays.
\end{proposition}

\begin{proof} By Proposition~\ref{prop:Mori-generators-correct}, the Mori cone is generated by (\ref{zetagamma}). Define
\begin{equation}\label{dell}
\begin{split}
&\Delta^+_k:=-\sum\nolimits_{i=0}^{k-1}(k-i)D^+_i,\,\,\,\,\,\,\Delta^-_k:=-\sum\nolimits_{i=0}^{k-1}(k-i)D^-_i,\,\,0\leq k\leq n-1.
\end{split}
\end{equation} 
By the intersection pairing computed in Lemma \ref{lem:intersection-table-correct}, the Nef cone consists of $D:=hH+\sum\nolimits_{i=1}^{n-1}a_i\Delta_i^-+\sum\nolimits_{j=1}^{n-1}b_j\Delta_j^+$ subjected to the conditions
\begin{equation}\label{subj}
\left\{\begin{aligned}
&h-\sum\nolimits_{i=\ell+1}^{n-1}a_i-\sum\nolimits_{j=n-\ell}^{n-1}b_j\geq0,\,\,\,\,\,\,0\leq \ell\leq n-1\\
&a_1,\cdots,a_{n-1},b_1,\cdots,b_{n-1}\geq 0
\end{aligned}\right..
\end{equation}

We prove the statement by induction on $n$. For $n=2$, the cone defined by (\ref{subj}) is generated by the vectors $(h,a_1,b_1) = (1,1,1)$, $(1,1,0)$, $(1,0,1)$, and $(1,0,0)$; therefore Proposition \ref{prop:nef-cone} holds. Now assume that the statement holds for $n=m-1$ with $m-1\geq 2$, and we proceed to prove it for $n=m$.

Take a primitive generator of $\mathrm{Nef}(\mathcal {TL}_{m})$. If $a_k,b_{m-k}\geq 1$ for a certain $1\leq k\leq m-1$, then by (\ref{eq:boundary-mid}), $D-\operatorname{min}\{a_k,b_{m-k}\}\cdot B_k$ also satisfies (\ref{subj}), which is a contradiction. If $a_k=b_{m-k}=0$ for a certain $1\leq k\leq m-1$, we can delete the variables $a_k,b_{m-k}$ and reorder the others so that (\ref{subj}) takes the same form with $m$ replaced by $m-1$. The proof is complete by the induction assumption.  

Now we assume that $D$ is a primitive generator of $\mathrm{Nef}(\mathcal {TL}_{m})$ such that for each $1\leq k\leq m-1$, either $a_k=0$ and $b_{m-k}\geq 1$, or $a_k\geq 1$ and $b_{m-k}=0$.
Note that an extremal ray is an intersection of at least $2m-2$ supporting hyperplanes.  Then at least $m-1$ supporting hyperplanes come from the first $m$ inequalities of (\ref{subj}). Since $m\geq 3$, there are two adjacent ones. If we subtract them, we should have $a_k=b_{m-k}=0$ for a certain $1\leq k\leq m-1$, which is a contradiction.
\end{proof}

Using the blow--up description of $\TL_n$ and the standard formula for the
canonical bundle under blow--ups, one computes the canonical class in terms of
the generators of $\Pic(\TL_n)$:
\begin{equation}\label{eq:canonical-TLn}
\begin{aligned}
K_{\TL_n}=-(n+1)H
 + \sum_{i=0}^{n-2}
   \left(\frac{(n+1-i)(n-i)}{2}-1\right)\bigl(D_i^-+D_i^+\bigr)= -\sum_{m=1}^{n-1} B_m
       -\sum_{i=0}^{n-1}\bigl(D_i^-+D_i^+\bigr).
\end{aligned}
\end{equation}

\begin{theorem}\label{thm:Fano-weakFano}
The variety $\TL_n$ is weak Fano for every $n\ge 1$. Moreover, $\TL_n$ is Fano if
and only if $n\le 2$.
\end{theorem}

\begin{proof}
By \eqref{eq:canonical-TLn} and Lemma~\ref{lem:intersection-table-correct}, the intersection numbers of $-K_{\TL_n}$ with the Mori generators $C_\ell,C_j^+,C_j^-$ are all nonnegative. 
Hence $-K_{\TL_n}$ is nef. According to \cite[Theorem 1.3]{FZ},  
$-K_{\TL_n}$ is big and $\TL_n$ is weak Fano. For $n=1,2$, one checks that $-K_{\TL_n}$ lies in the interior of $\Nef(\TL_n)$.
\end{proof}

\begin{corollary}\label{prop:log-Fano}
For every $n\ge 1$ the variety $\TL_n$ is log Fano.
\end{corollary}

\begin{proof}
By Theorem~\ref{thm:Fano-weakFano}, the divisor $-K_{\TL_n}$ is nef and big. A standard argument in the Minimal Model Program produces an effective $\Q$--divisor $\Delta$ such that $-(K_{\TL_n}+\Delta)$ is ample and the pair $(\TL_n,\Delta)$ is klt. 
\end{proof}

We conclude this section by describing the local deformations of $\TL_n$. Our first step is to establish a generalization of \cite[Proposition 4.2\textup{(iii)}]{BienBrion} under weaker positivity.
\begin{proposition}\label{gbb}
Let $X$ be a smooth projective toroidal spherical variety. Assume that the restriction of the anticanonical bundle of $X$ to each prime boundary divisor is nef and big. Then $H^k(X,T_X)=0$ for all $k>0$.
\end{proposition}
\begin{proof}
Let $\partial X=\sum_{i=1}^{m}D_i$ be the boundary divisor, and let $S_{X}\subset T_{X}$ be the subsheaf of vector fields tangent to $\partial X$, so $S_{X}=T_{X}(-\log\partial X)$ in logarithmic language. By \cite[Proposition~2.3.2]{BienBrion} there is an exact sequence 
\begin{equation}\label{eq:BB-exact}
0\longrightarrow S_{X}\longrightarrow T_{X}\longrightarrow
\bigoplus\nolimits_{i=1}^{m}\OO_{X}(D_i)|_{D_i}
\longrightarrow 0.
\end{equation}
Knop's vanishing theorem \cite[Theorem~4.1]{Knop94}, as used in \cite[Section 4]{BienBrion}, yields
$H^k(X,S_{X})=0$ for all $k>0$.
For any boundary component $D_i$,  we have by the adjunction formula that $\OO_{X}(D)|_D\cong K_D\otimes (-K_{X})|_{D}$.
Then the Kawamata--Viehweg vanishing theorem yields that $H^k\left(D,\ \OO_{X}(D)|_{D}\right)=0$ for all $k>0$. We complete the proof by taking cohomology in \eqref{eq:BB-exact}.
\end{proof}

\begin{lemma}\label{lem:anticanonical-restriction}
For every $n\ge 3$, every $0\le j\le n-1$, and each sign $\pm$, the restriction $(-K_{\TL_n})|_{D_j^\pm}$ is nef and big on the boundary divisor $D_j^\pm$.
\end{lemma}

\begin{proof}
Since $-K_{\TL_n}$ is nef on $\TL_n$ by Theorem~\ref{thm:Fano-weakFano}, its restriction to any subvariety is nef. 
Then it suffices to show that $(-K_{\mathcal {TL}_{n}})|_{D^{\pm}_j}$ is in the interior of the effective cone. 
The proof is similar to that of \cite[Theorem 1.2]{FZ21}. Without loss of generality, we may consider the restriction of (\ref{eq:canonical-TLn}) to $D^-_j$, for $0\le j\le n-1$.  

Denote by $B^{-j}_{m}$, $0\leq m\leq n$, the restriction of the line bundle $B_m$ to $D^{-}_j$, and by $D^{-j}_{\pm i}$, $0\leq i\leq n-1$, the restriction of $D^{\pm}_i$ to $D^-_j$. Note that $D^{-j}_{-i}$ is given by the intersection $D^{-}_j\cap D^{-}_i$, when $i\neq j$; $D^{-j}_{+i}$ is given by $D^{-}_j\cap D^{+}_i$, which is empty when $0\leq i\leq n-1-j$. 
It is easy to verify that $D^{-}_j$ is a spherical variety with boundary divisors $\{D^{-j}_{-i}\}_{0\leq i\leq n-1, i\neq j}\cup\{D^{-j}_{+i}\}_{n-j\leq i\leq n-1}$ and colors $\{B^{-j}_m\}_{1\leq m\leq n-1}$. Moreover, one can show that its effective cone is generated by $\{D^{-j}_{-i}\}_{0\leq i\leq n-1, i\neq j}\cup\{D^{-j}_{+i}\}_{n-j\leq i\leq n-1}\cup\{B^{-j}_{j}\}$, where $B^{-0}_0 = \emptyset$ if $j = 0$.  

When $2\leq j\leq n-2$, we have 
\begin{equation*}
\sum\nolimits_{0\leq i\leq j}D^{-j}_{-i}=B^{-j}_{j}-B^{-j}_{j+1}\,\,\,\,\,\,\,\,{\rm and}\,\,\,\,\,\,\sum\nolimits_{0\leq i\leq j-1}D^{-j}_{-i}=B^{-j}_{j-1}-B^{-j}_{j}+D^{-j}_{+(n-j)}.
\end{equation*}
Setting $0<t<1$, we get
\begin{equation*}
\begin{split}
-K_{\mathcal {TL}_{n}}|_{D^{-}_j}&= t\cdot\sum\nolimits_{0\leq i\leq j-1}D^{-j}_{-i}+
\sum\nolimits_{j+1\leq i\leq n-1} D^{-j}_{-i}+\sum\nolimits_{n-j+1\leq i\leq n-1}D^{-j}_{+i}+(1-t)\cdot D^{-j}_{+(n-j)}\\
&\,\,\,\,\,\,\,\,\,+(2+t)\cdot B_{j}^{-j}+(1-t)\cdot B_{j-1}^{-j}+\sum\nolimits_{1\leq m\leq n-1,\,m\neq j-1,j,j+1}B_{m}^{-j}.
\end{split} 
\end{equation*}

The same argument works for $j=0,1,n-1$. The proof is complete. 
\end{proof}

\begin{theorem}\label{thm:rigidity-absolute}
For every $n\ge 1$ we have
\[
H^k\bigl(\TL_n,T_{\TL_n}\bigr)=0\qquad\text{for all }k>0.
\]
In particular, $\TL_n$ is locally rigid, namely, it has no nontrivial local deformations.
\end{theorem}
\begin{proof}
Since $\mathcal {TL}_{n}$ are toroidal, Theorem \ref{thm:rigidity-absolute} holds when they are Fano (see \cite[Corollary 3.3.11]{Pe14}).  When $n\geq 3$, it follows from Proposition \ref{gbb} and Lemma \ref{lem:anticanonical-restriction}.
\end{proof}

\begin{remark}\label{rem:FuLi}
Fu and Li \cite{Fu24} studied the global rigidity of wonderful group compactifications under
Fano deformations. Their strategy uses, among other inputs,
local rigidity of Bien--Brion type for regular Fano compactifications. Theorem~\ref{thm:rigidity-absolute}
establishes the same kind of vanishing for $\TL_n$ even when $\TL_n$ is only weak Fano. 
A natural question is whether $\TL_n$ exhibit global rigidity under weak Fano deformations.
\end{remark}

\subsection*{Automorphisms and pseudoautomorphisms of \texorpdfstring{$\TL_n$}{TLn}}

We determine the biregular automorphism group and show that every pseudoautomorphism is actually biregular.


\begin{proposition}\label{prop:Aut-LG}
For every $n\ge 1$ we have
\[
\Aut(\LG(n,2n)) \cong \PSp_{2n}.
\]
\end{proposition}

\begin{proof}
For a rational homogeneous space $G/P$ the connected component $\Aut^0(G/P)$ coincides with the adjoint group $G^{\rm ad}$, except for a short explicit list of exceptional pairs $(G,P)$ \cite[Theorem~1]{Dem77}. The Lagrangian Grassmannian $\LG(n,2n)$ is the homogeneous space $\Sp_{2n}/P_{\alpha_n}$, and this pair is not among the exceptional cases; moreover the Dynkin diagram of type $C_n$ has no nontrivial automorphisms. Hence $\Aut(\LG(n,2n))=\Aut^0(\LG(n,2n))\cong \PSp_{2n}$.
\end{proof}

\begin{remark}
For $n=2$ one has $\LG(2,4)\cong Q^3\subset \PP^4$.  In this case
$\Aut(Q^3)\cong \PO_5\cong \PSO_5$, and the exceptional isomorphism $\PSp_4\cong \PSO_5$ identifies
this group with the one in Proposition~\ref{prop:Aut-LG}.
\end{remark}

\begin{lemma}\label{lem:stab-two-points}
Let $G:=\Aut(\LG(n,2n))\cong \PSp_{2n}$ and let $\Aut(\LG(n,2n))_{p_+,p_-}\subset G$ be the subgroup
fixing both $p_+ := 0\oplus W^\vee,\: p_- := W\oplus0\in \LG(n,2n)$.  Then
\[
\Aut(\LG(n,2n))_{p_+,p_-} \cong \GL_n/\{\pm I\}.
\]
\end{lemma}
\begin{proof}
An element of $\Sp_{2n}$ fixes $p_+$ and $p_-$ if and only if it preserves the direct sum decomposition $W\oplus W^\vee$, hence it is block diagonal. The symplectic condition forces $h=g^{-\top}$, so the stabilizer in $\Sp_{2n}$ is isomorphic to $\GL_n$ via $g\mapsto \mathrm{diag}(g,g^{-\top})$. Passing to $\PSp_2n=\Sp_2n/\{\pm I\}$ yields $\GL_n/\{\pm I\}$.
\end{proof}

For $f\in \Aut(\TL_n)$, the pullback $f^*$ preserves $\Eff(\TL_n)$, hence permutes its
extremal rays.  We thus obtain a group homomorphism
\[
\rho \colon\Aut(\TL_n)\longrightarrow \mathfrak{S}_{2n},
\qquad
f\longmapsto \text{the induced permutation of }\{[D_i^\pm]\}_{0\le i\le n-1}.
\]

\begin{lemma}\label{lem:image-is-S2}
The image of $\rho$ is the subgroup $S_2\subset \mathfrak{S}_{2n}$ generated by the swap
\[
[D_i^+]\longleftrightarrow [D_i^-]\qquad\text{for all }0\le i\le n-1.
\]
Equivalently, any $f\in\Aut(\TL_n)$ either preserves each of the two sets
$\{D_i^+\}_{i=0}^{n-1}$ and $\{D_i^-\}_{i=0}^{n-1}$, or swaps them simultaneously.
\end{lemma}

\begin{proof}
By the intersection numbers in Lemma \ref{lem:intersection-table-correct} either $D_0^+\mapsto D_0^+$, which forces $D_i^+\mapsto D_i^+$ and $D_i^-\mapsto D_i^-$ for all $i$, or $D_0^+\mapsto D_0^-$, which forces $D_i^+\mapsto D_i^-$ and $D_i^-\mapsto D_i^+$ for all $i$. Now, let $\sigma\colon V=W\oplus W^\vee\to V$ be the linear symplectic automorphism given in block form by so that $\sigma(W)=W^\vee$ and $\sigma(W^\vee)=W$. By functoriality of osculating spaces (and hence of the osculating loci $Z_k^\pm$), this automorphism exchanges the two blow--up towers $\{Z_k^+\}$ and $\{Z_k^-\}$, and therefore lifts to an involution of the iterated blow--up $\TL_n$. Then there exists an involution $\iota\in \Aut(\TL_n)$ inducing the nontrivial element of $\mathrm{Im}(\rho)\cong S_2$, and hence $\mathrm{Im}(\rho)\cong S_2$.
\end{proof}

\begin{theorem}\label{thm:Aut-TLn}
There is a split exact sequence
\[
1\longrightarrow \Aut(\LG(n,2n))_{p_+,p_-}\longrightarrow \Aut(\TL_n)
\overset{\rho}{\longrightarrow} S_2\longrightarrow 1,
\]
and hence
\[
\Aut(\TL_n)\ \cong\ \bigl(\GL_n/\{\pm I\}\bigr)\rtimes S_2,
\]
where $S_2$ acts on $\GL_n/\{\pm I\}$ by $g\mapsto g^{-\top}$.
\end{theorem}

\begin{proof}
By Lemma~\ref{lem:image-is-S2}, $\rho$ is surjective and has a section. Let $f\in \Ker(\rho)$. Then $f$ fixes each boundary divisor $D_i^\pm$, hence in particular it fixes every $\RL_n$--exceptional divisor $D_i^\pm$ for $0\le i\le n-2$. As in the proof of \cite[Lemma 4.6]{F26}, $f$ descends to a biregular automorphism $\bar f$ of $\LG(n,2n)$ fixing the points $p_+$ and $p_-$. The description of the semidirect product action follows from the explicit swap matrix.
\end{proof}


\begin{corollary}\label{cor:H0T}
We have $h^0\big(\TL_n,T_{\TL_n}\big)=\dim \Aut(\TL_n)=n^2$.
\end{corollary}
\begin{proof}
The identity component of $\Aut(\TL_n)$ is $\GL_n/\{\pm I\}$, hence its Lie algebra has
dimension $n^2$.
\end{proof}

\begin{theorem}\label{thm:PsAut-equals-Aut}
Every pseudoautomorphism of $\TL_n$ is biregular.  Equivalently,
\[
\PsAut(\TL_n)=\Aut(\TL_n) \cong\ \bigl(\GL_n/\{\pm I\}\bigr)\rtimes S_2.
\]
\end{theorem}

\begin{proof}
Let $\varphi:\TL_n\dashrightarrow \TL_n$ be a pseudoautomorphism.  Since $\varphi$ is an
isomorphism in codimension one, it induces a linear automorphism $\varphi^*:N^1(\TL_n)_\R\longrightarrow N^1(\TL_n)_\R$ sending effective divisors to effective divisors, and hence preserving
$\Eff(\TL_n)$.

In particular, $\varphi$ preserves the full exceptional locus of $\RL_n$ (the divisors
$D_i^\pm$ for $0\le i\le n-2$).  Therefore $\varphi$ descends to a birational self--map $\bar\varphi:\LG(n,2n)\dashrightarrow \LG(n,2n)$ which does not contract any divisor: indeed, if a divisor on $\LG(n,2n)$ were contracted by
$\bar\varphi$, then its strict transform on $\TL_n$ would be a divisor contracted by $\varphi$,
contradicting that $\varphi$ is a pseudoautomorphism. The same argument shows that $\varphi^{-1}$ does not contract any divisor. Now, since $\Pic(\LG(n,2n))\cong \Z$ \cite[Proposition~7.2]{Mas1} yields that $\bar\varphi$ is in fact a biregular automorphism of $\LG(n,2n)$.

Finally, since $\bar\varphi$ is biregular, it preserves the osculating loci $Z_k^\pm$
intrinsically associated with the points $p_\pm$ and the Pl\"ucker embedding, and therefore it
lifts step--by--step to the iterated blow--up $\TL_n$.  Because $\varphi$ already preserves every
exceptional divisor, this lift coincides with $\varphi$ on the common open subset where $\RL_n$
is an isomorphism; hence $\varphi$ is biregular.
\end{proof}


\subsection{Cox ring and movable cone of \texorpdfstring{$\TL_n$}{TLn}}

We now describe the Cox ring of $\TL_n$. This allows us to express its movable cone in terms of the $\mathscr{B}$--stable divisors introduced above.

\begin{lemma}\label{cohBC}
For every $0\le i\le n-1$ and for each sign $\pm$ we have
\[
h^0\bigl(\TL_n,\OO_{\TL_n}(D_i^\pm)\bigr)=1,
\]
and the unique (up to scalars) nonzero section is the canonical section
$s_{D_i^\pm}$ vanishing exactly along $D_i^\pm$.

Moreover, for every $1\le k\le n-1$ set
\[
V_k:=H^0\bigl(\TL_n,\OO_{\TL_n}(B_k)\bigr),\qquad r_k:=\dim V_k.
\]
Then $V_k$ is an irreducible $\mathscr{G}$--module generated by the
$\mathscr{B}$--eigenvector $s_{B_k}$, and it can be identified as follows.
Let
$f\colon\bigwedge\nolimits^n (W\oplus W^\vee) \longrightarrow \bigwedge\nolimits^{n-2} (W\oplus W^\vee)$
be contraction with the standard symplectic form on $W\oplus W^\vee$. 
The map $f$ restricts to the $\GL(W^\vee)$--equivariant contractions
\[
\contr_k\colon\bigwedge\nolimits^{n-k}W\otimes \bigwedge\nolimits^k W^\vee
\longrightarrow
\bigwedge\nolimits^{n-k+1}W\otimes \bigwedge\nolimits^{k-1} W^\vee,
\qquad 1\le k\le n.
\]
For $1\le k\le n-1$, there are isomorphisms of $\mathscr{G}$--modules $V_k^\vee \ \cong\ \Ker(\contr_k)$, 
and consequently $r_k =\binom{n}{k}^2-\binom{n}{k-1}\binom{n}{k+1}$ 
where $\binom{n}{-1}=\binom{n}{n+1}=0$.
\end{lemma}

\begin{proof}
Fix $0\le i\le n-1$. We know that there exists a $T$--invariant curve class $C$ moving in $D_i^+$ (resp.\ in $D_i^-$) such that $D_i^\pm\cdot C<0$. Then $E\cdot C=D_i^\pm\cdot C<0$, hence $C\subset \Supp(E)$ (otherwise one would have $E\cdot C\ge0$). Finally, for $1\le k\le n-1$ the map $\contr_k$ is surjective, we derive the dimension of $V_k$.
\end{proof}

\begin{remark}\label{cohBC_Geom}
Notice that the linear system of $B_k$ is cut on $\LG(n,2n)$ by the hyperplanes containing the span of certain osculating spaces of $\LG(n,2n)$ at the points $p_{\pm}$ (see \cite[Lemma 3.9]{F26} for the case of Grassmannians). Hence, the dimension of $V_k$ could also be deduced from \cite[Proposition 3.9]{FMR20}.
\end{remark}

\begin{proposition}\label{prop:Cox-TLn}
The Cox ring $\Cox(\TL_n)$ is finitely generated as a $\C$--algebra by:
\begin{itemize}
\item[-] the canonical sections of the boundary divisors $D_i^\pm$,
$0\le i\le n-1$, and
\item[-] for each color $B_k$ with $1\le k\le n-1$, a basis of the vector space
$V_k=H^0(\TL_n,\OO_{\TL_n}(B_k))$, whose dimension is
$r_k=\binom{n}{k}^2-\binom{n}{k-1}\binom{n}{k+1}$.
\end{itemize}

Concretely, let
\[
s_{D_i^+}\in H^0\bigl(\TL_n,\OO_{\TL_n}(D_i^+)\bigr),\qquad
s_{D_i^-}\in H^0\bigl(\TL_n,\OO_{\TL_n}(D_i^-)\bigr)
\qquad (0\le i\le n-1)
\]
be the canonical sections vanishing exactly along $D_i^\pm$, and for each
$1\le k\le n-1$ choose a basis
\[
s_{k,1},\dots,s_{k,r_k}\in V_k=H^0\bigl(\TL_n,\OO_{\TL_n}(B_k)\bigr).
\]
Then $\Cox(\TL_n)$ is generated by the set
\[
\{\,s_{D_i^+},s_{D_i^-}\mid 0\le i\le n-1\,\}
\ \cup\
\{\,s_{k,\ell}\mid 1\le k\le n-1,\ 1\le \ell\le r_k\,\},
\]
so the total number of generators is
$2n+\sum_{k=1}^{n-1}\left(\binom{n}{k}^2-\binom{n}{k-1}\binom{n}{k+1}\right)$.

With respect to the basis of $\Pic(\TL_n)$ given by
$H,\ D_0^+,\dots,D_{n-2}^+,\ D_0^-,\dots,D_{n-2}^-$,
the $\Pic(\TL_n)$--grading of these generators is given by
\[
\deg(s_{D_i^+})=[D_i^+],\ \deg(s_{D_i^-})=[D_i^-]\qquad (0\le i\le n-2),
\]
and, using Proposition~\ref{prop:picard-effective},
\[
\deg(s_{D_{n-1}^+})=[D_{n-1}^+]
=
H-\sum\nolimits_{i=0}^{n-2}(n-i)\,D_i^+,
\qquad
\deg(s_{D_{n-1}^-})=[D_{n-1}^-]
=
H-\sum\nolimits_{i=0}^{n-2}(n-i)\,D_i^-,
\]
while for $1\le k\le n-1$ and $1\le \ell\le r_k$ one has
\[
\deg(s_{k,\ell})=[B_k]
=
H
-\sum\nolimits_{i=0}^{\,n-k-1}(n-k-i)\,D_i^+
-\sum\nolimits_{i=0}^{\,k-1}(k-i)\,D_i^-.
\]
Equivalently, the grading matrix is the $(2n-1)\times\bigl(2n+\sum_{k=1}^{n-1}r_k\bigr)$
integer matrix whose columns are precisely the degree vectors listed above, with the
degree vector of $[B_k]$ repeated $r_k$ times.

\end{proposition}

\begin{proof}
For a spherical variety $X$ the $\Cl(X)$--graded Cox ring is generated by the canonical sections of the $\mathscr{G}$--stable prime divisors together with finitely many simple $\mathscr{G}$--submodules of $H^0(X,\OO(D))$ for the colors $D$ \cite[Theorem 4.5.4.6]{ADHL15}. Lemma~\ref{cohBC} shows that each $H^0(\TL_n,\OO(D_i^\pm))$ is one--dimensional, hence contributes a single generator, while each color $B_k$ contributes exactly $r_k=\dim H^0(\TL_n,\OO(B_k))$ generators, i.e.\ a basis $s_{k,1},\dots,s_{k,r_k}$. The explicit degrees follow from Proposition~\ref{prop:picard-effective}, and the grading matrix is obtained by arranging these degree vectors as columns.
\end{proof}

Since $\TL_n$ is a Mori dream space its movable cone is
polyhedral and may be described in terms of the $\operatorname{Pic}(\TL_n)$--degrees
of the Cox ring generators. We use \cite[Proposition~3.3.2.3]{ADHL15} stating what follows: let $X$ be an irreducible normal complete variety with finitely generated
$\mathrm{Cl}(X)$ and Cox ring. Choose a system
$f_1,\dots,f_t$ of pairwise nonassociated $\mathrm{Cl}(X)$--prime generators of
$\Cox(X)$. Then
\[
\Mov(X)=\bigcap\nolimits_{i=1}^t \cone\big(\deg(f_j)\mid j\ne i\big)
\subset \mathrm{Cl}(X)_\QQ.
\]

\begin{remark}
\label{rem:multiplicity}
For $1\le k\le n-1$, the divisor $B_k$ has at least two independent sections,
hence there are at least
two pairwise nonassociated $\Pic(\TL_n)$--prime Cox ring generators of the same
degree $\deg(B_k)$. Therefore the column $\deg(B_k)$ must be inserted at
least twice in the grading matrix used to compute $\Mov(\TL_n)$ from Cox
degrees. By contrast, $B_0$ and $B_n$ are rigid and are inserted only once.
\end{remark}

Specializing to $X=\TL_n$, we choose the Cox generators given by canonical
sections of the $\mathscr{B}$--stable prime divisors, with the multiplicity convention of
Remark~\ref{rem:multiplicity}. Concretely, define the list of degree vectors
\[
\mathsf{Deg}_n:=\big(\deg(B_i),\deg(B_i),\;
\deg(D_{n-1}^+),\deg(D_{n-1}^-),\;
\deg(D_j^+),\;
\deg(D_j^-)\big),
\]
for $1\leq i\leq n-1$, $0\leq j\leq n-2$, where all degrees are written in the basis
$(H;D_0^+,\dots,D_{n-2}^+;D_0^-,\dots,D_{n-2}^-)$.

\begin{proposition}\label{prop:mov-by-degrees}
Let $\mathsf{Deg}_n=(v_1,\dots,v_m)$, with $m = 4n-2$, be the degree list above and set
\[
\Mov_n := \bigcap\nolimits_{i=1}^m \cone(v_j\mid j\ne i)\subset \Pic(\TL_n)_\QQ.
\]
Then $\Mov(\TL_n)=\Mov_n$. In particular, the extremal rays of $\Mov(\TL_n)$ are
precisely the rays of the polyhedral cone $\Mov_n$.
\end{proposition}
\begin{proof}
Since $\TL_n$ is a Mori dream space, its Cox ring is finitely generated and $\Mov(\TL_n)\subset \Pic(\TL_n)_\QQ$ is a rational polyhedral cone. Let $f_1,\dots,f_t$ be a system of pairwise nonassociated $\Pic(\TL_n)$--prime generators of $\Cox(\TL_n)$. By \cite[Proposition~3.3.2.3]{ADHL15} one has $\Mov(\TL_n)=\bigcap_{k=1}^{t}\cone\bigl(\deg(f_\ell)\mid \ell\neq k\bigr)$. Therefore the intersection above may be indexed by the elements of $\mathsf{Deg}_n$, yielding $\Mov(\TL_n)=\bigcap_{i=1}^{m}\cone\bigl(v_j\mid j\neq i\bigr)$.
\end{proof}

\begin{example}[The case $n=2$]\label{subsec:TL2}
Recall that $\LG(2,4)\cong Q^3\subset \PP^4$ is a smooth quadric threefold.
$\RL_2\colon \TL_2\to Q^3$ is the blow--up of $Q^3$ at the two $\mathscr{B}$--fixed points
$p_+,p_-$. Then $\Pic(\TL_2)\cong \ZZ H\oplus \ZZ D_0^+\oplus \ZZ D_0^-$, where
\[
H=(\RL_2)^*\OO_{Q^3}(1),\qquad
D_0^+=\RL_2^{-1}(p_+)\cong \PP^2,\qquad
D_0^-=\RL_2^{-1}(p_-)\cong \PP^2.
\]
The remaining
$\mathscr{B}$--stable boundary divisors are
$D_1^{\pm}= H-2D_0^{\pm}$, the unique non--rigid color is $B_1=H-D_0^+-D_0^-$, and the anticanonical class is given by $-K_{\TL_2}= 3H-2D_0^+-2D_0^-$. In $\Pic(\TL_2)_\QQ$ one has
\[
\Eff(\TL_2)=\cone(D_0^+,D_1^+,D_0^-,D_1^-),\qquad
\Mov(\TL_2)=\Nef(\TL_2)=\cone\bigl(H,\ H-D_0^+,\ H-D_0^-,\ B_1\bigr).
\]
The Mori chamber decomposition of $\Mov(\TL_2)$ consists of a single
chamber, namely $\Mov(\TL_2)=\Nef(\TL_2)$.

$\NE(\TL_2)$ is generated by the $T$--invariant curve classes $C_0=[\gamma_0]$, $C_1=[\gamma_1]$, $C_1^{\pm}=[\zeta_1^\pm]$ as in Definition~\ref{def:T-curves-correct}. The curves $C_1^+\subset D_0^+$ and $C_1^-\subset D_0^-$ are contracted by $\pi$ and are represented by lines in the corresponding exceptional planes. Let $\epsilon^+ := C_1^+, \epsilon^- := C_1^-$, and let $\ell$ denote the strict transform of a general line. Then
\[
\NE(\TL_2)=\cone(\epsilon^+,\epsilon^-,\ell-\epsilon^+,\ell-\epsilon^-),
\,\,\,\,\,
\Mov_1(\TL_2)=\cone\bigl(\ell,\ 2\ell-\epsilon^+,\ 2\ell-\epsilon^-,\
2\ell-\epsilon^+-\epsilon^-\bigr).
\]

Geometrically, $2\ell-\epsilon^+$ (resp.\ $2\ell-\epsilon^-$) is represented by the
strict transform of a general plane--section conic of $Q^3$ passing
through $p_+$ (resp.\ $p_-$) and avoiding the other point. These conics move and
cover $\TL_2$. The class $2\ell-\epsilon^+-\epsilon^-$ is represented by the strict
transform of a general plane--section conic passing through both $p_+$ and $p_-$;
varying the plane yields a covering family, hence this class spans an extremal ray
of $\Mov_1(\TL_2)$.

\begin{figure}[ht]
\centering
\begin{subfigure}[b]{0.48\textwidth}
\centering
\begin{tikzpicture}[scale=0.27,
  every node/.style={font=\fontsize{1.5}{1.5}\selectfont, inner sep=0.1pt, outer sep=0pt}]

  \fill[gray!15]
    (-12,0) -- (0,6) -- (6,0) -- (0,-12) -- cycle;
  \draw[thick]
    (-12,0) -- (0,6) -- (6,0) -- (0,-12) -- cycle;

  \fill[gray!55]
    (0,0) -- (-3,0) -- (-6,-6) -- (0,-3) -- cycle;
  \draw[thick]
    (0,0) -- (-3,0) -- (-6,-6) -- (0,-3) -- cycle;

  \draw[thin]
    (0,6) -- (0,0)
    (0,6) -- (-3,0)
    (6,0) -- (0,0)
    (6,0) -- (0,-3)
    (-12,0) -- (-3,0)
    (-12,0) -- (-6,-6)
    (0,-12) -- (-6,-6)
    (0,-12) -- (0,-3);

  \fill (-12,0) circle (0.22);
  \node[above left, xshift=1pt, yshift=-1pt] at (-12.3,0) {$D_1^+$};

  \fill (0,6) circle (0.22);
  \node[above left, xshift=1pt, yshift=-1pt] at (-0.3,6) {$D_0^-$};

  \fill (6,0) circle (0.22);
  \node[above right, xshift=-1pt, yshift=-1pt] at (6.3,0) {$D_0^+$};

  \fill (0,-12) circle (0.22);
  \node[below right, xshift=-1pt, yshift=1pt] at (0.3,-11.7) {$D_1^-$};

  \fill (0,0) circle (0.22);
  \node[above right, xshift=-1pt, yshift=-1pt] at (0.3,0.3) {$H$};

  \fill (-3,0) circle (0.22);
  \node[above, yshift=-1pt] at (-4.6,0.3) {$H-D_0^+$};

  \fill (0,-3) circle (0.22);
  \node[left, xshift=1pt] at (3.6,-3) {$H-D_0^-$};

  \fill (-6,-6) circle (0.22);
  \node[below left, xshift=1pt, yshift=1pt] at (-6.2,-6) {$B_1$};

  \fill (-2,-2) circle (0.22);
  \node[right, xshift=1pt] at (-4.1,-2.67) {$-K_{\TL_2}$};

\end{tikzpicture}
\caption{}
\label{fig:sub1}
\end{subfigure}
\hfill
\begin{subfigure}[b]{0.48\textwidth}
\centering
\begin{tikzpicture}[scale=0.27,
  every node/.style={font=\fontsize{1.5}{1.5}\selectfont, inner sep=0.1pt, outer sep=0pt}]

  \fill[gray!15]
    (-12,0) -- (0,6) -- (6,0) -- (0,-12) -- cycle;
  \draw[thick]
    (-12,0) -- (0,6) -- (6,0) -- (0,-12) -- cycle;

  \fill[gray!55]
    (0,0) -- (-3,0) -- (-6,-6) -- (0,-3) -- cycle;
  \draw[thick]
    (0,0) -- (-3,0) -- (-6,-6) -- (0,-3) -- cycle;

  \draw[thin]
    (0,6) -- (0,0)
    (0,6) -- (-3,0)
    (6,0) -- (0,0)
    (6,0) -- (0,-3)
    (-12,0) -- (-3,0)
    (-12,0) -- (-6,-6)
    (0,-12) -- (-6,-6)
    (0,-12) -- (0,-3);

  \fill (-12,0) circle (0.22);
  \node[above left, xshift=1pt, yshift=-1pt] at (-12.3,0) {$\scriptscriptstyle \ell-\epsilon^+$};

  \fill (0,6) circle (0.22);
  \node[above left, xshift=1pt, yshift=-1pt] at (-0.3,6) {$\scriptscriptstyle \epsilon^-$};

  \fill (6,0) circle (0.22);
  \node[above right, xshift=-1pt, yshift=-1pt] at (6.5,0) {$\scriptscriptstyle \epsilon^+$};

  \fill (0,-12) circle (0.22);
  \node[below right, xshift=-1pt, yshift=1pt] at (0.4,-11.7) {$\scriptscriptstyle \ell-\epsilon^-$};

  \fill (0,0) circle (0.22);
  \node[above right, xshift=-1pt, yshift=-1pt] at (0.3,0.3) {$\scriptscriptstyle \ell$};

  \fill (-3,0) circle (0.22);
  \node[above, yshift=-1pt] at (-4.3,0.3) {$\scriptscriptstyle 2\ell-\epsilon^+$};

  \fill (0,-3) circle (0.22);
  \node[left, xshift=1pt] at (3.2,-3) {$\scriptscriptstyle 2\ell-\epsilon^-$};

  \fill (-6,-6) circle (0.22);
  \node[below left, xshift=1pt, yshift=1pt] at (-6.3,-6) {$\scriptscriptstyle 2\ell-\epsilon^+-\epsilon^-$};

\end{tikzpicture}
\caption{}
\label{fig:sub2}
\end{subfigure}
\caption{(a) A $2$--dimensional section of $\Eff(\TL_2)$. The light grey region is the section of the effective cone; the dark grey quadrilateral is the movable (and nef) cone. (b) A $2$--dimensional section of $\NE(\TL_2)$ (light grey) and of the cone $\Mov_1(\TL_2)$ of moving curves (dark grey).}
\label{fig:combined}
\end{figure}
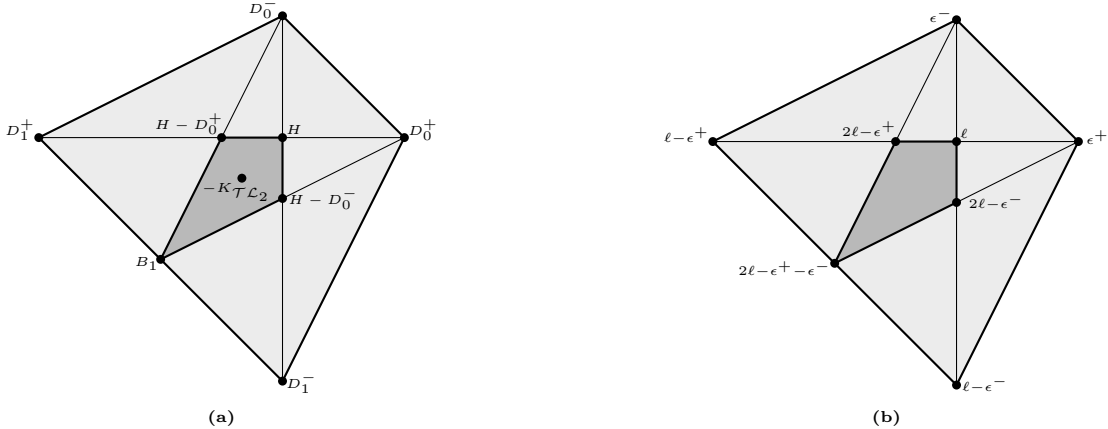

By Proposition \ref{prop:Cox-TLn} we have the generators of $\Cox(\TL_2)$. Let $X_{\mathrm{tor}}$ be the canonical toric embedding associated to the set of
$\mathscr{B}$--stable prime divisors and colors (so that its Cox ring is a polynomial ring
graded by $\Pic(\TL_2)$ with weights given by the classes of
$D_0^\pm,D_1^\pm,B_1$).
The Mori chamber decomposition of $X_{\mathrm{tor}}$ coincides with the GKZ
decomposition (secondary fan) of the corresponding weight configuration.
In particular, the GKZ fan subdivides $\Mov(X_{\mathrm{tor}})$ into finitely
many maximal chambers, and the number of such chambers gives an upper bound on that of Mori chambers of $\TL_2$, since walls may collapse when restricting
from $X_{\mathrm{tor}}$ to $\TL_2$.

To illustrate the wall--collapsing phenomenon, consider the wall between 
$H-D_0^+$ and $H-D_0^-$, for example. Let $X:=\Bl_{\{p_+,p_-\}}\PP^4$ with blow--up map $\pi$ and exceptional divisors
$E_+=\pi^{-1}(p_+)$ and $E_-=\pi^{-1}(p_-)$. Consider the divisor $B^{\mathrm{amb}}_1 \ :=\ H-E_+-E_-$. The complete linear system $|B^{\mathrm{amb}}_1|$ is the strict transform of the hyperplanes in $\PP^4$ passing through both $p_+$ and $p_-$. 
The intersection of all such hyperplanes is the line $L:=\left\langle p_+p_-\right\rangle\subset \PP^4$, hence $|B^{\mathrm{amb}}_1|$ has stable base locus equal to the strict transform $\widetilde L\subset X$ of $L$. So $B^{\mathrm{amb}}_1$ is not nef on $X$. In the GKZ/Mori--chamber decomposition of the ambient toric variety $X$, this non--nef behavior produces a genuine wall:
crossing it changes the stable base locus (and corresponds to the birational
operation dictated by the curve $\widetilde L$). Now, restricting to the quadric threefold $Q^3$, the base locus of the restricted linear system is $\bigcap_{H'\supset \langle p_+,p_-\rangle}(Q^3\cap H') = Q^3\cap L=\{p_+,p_-\}$.
Therefore, after blowing up $p_\pm$, the base points are removed and the divisor $B_1 = H-D_0^+-D_0^-$ on  
$\TL_2$ becomes base--point--free. In particular, the ambient wall coming from the curve $\widetilde L$ disappears upon restriction to $\TL_2$ for $\widetilde L\not\subset\TL_2$.

\end{example}

\begin{example}[The case $n=3$]\label{subsec:TL3}
We have
$\Pic(\TL_3)\cong \ZZ H\oplus \ZZ D_0^+\oplus \ZZ D_1^+\oplus \ZZ D_0^-\oplus \ZZ D_1^-$, and
$D_2^{\pm} = H-3D_0^+-2D_1^+$.
The colors are 
$B_1=H-2D_0^+-D_1^+-D_0^-$ and
$B_2= H-D_0^+-2D_0^--D_1^-$. 
Note that
\[
\Eff(\TL_3)=\cone\bigl(D_0^+,D_1^+,D_2^+,D_0^-,D_1^-,D_2^-\bigr).
\]
In the basis $(H,D_0^+,D_1^+,D_0^-,D_1^-)$, the movable cone $\Mov(\TL_3)$ has
$16$ extremal rays, given by the vectors
\[
\begin{aligned}
&(1,-1,0,-2,-1),\ (1,-2,-1,-1,0),\ (1,0,0,-2,-1),\ (2,-2,0,-3,-2),\ (2,0,0,-3,-2),\ (2,-3,-2,-2,0),\\
&(1,-2,-1,0,0),\ (2,-3,-2,0,0),\ (1,0,0,0,0),\ (1,-1,0,0,0),\ (1,0,0,-1,0),\ (1,-1,0,-1,0),\\
&(3,-4,-2,-4,-2),\ (3,-4,-2,-3,-2),\ (3,-3,-2,-4,-2),\ (3,-3,-2,-3,-2).
\end{aligned}
\]
By Proposition \ref{prop:nef-cone}, $\Nef(\TL_3)$ is generated by the first $8$ rays in this list. 
A direct GKZ/secondary--fan computation yields $88$ maximal chambers for
$X_{\mathrm{tor}}$. Hence $\Mov(\TL_3)$ has at least $88$ Mori chambers.

By Proposition \ref{prop:Mori-generators-correct}, the Mori cone $\NE(\TL_3)\subset N_1(\TL_3)_\R$ is generated by the
$T$--invariant curve classes
$C_0=[\gamma_0]$, $C_1=[\gamma_1]$, $C_2=[\gamma_2]$,
$C_1^\pm=[\zeta_1^\pm]$, $C_2^\pm=[\zeta_2^\pm]$.
Let us construct these generators geometrically. 

Fix a symplectic vector space $(V,\omega)$ with symplectic basis
$(e_1,e_2,e_3,f_1,f_2,f_3)$ such that  $\omega(e_i,f_j)=\delta_{ij}$ and $\omega(e_i,e_j)=\omega(f_i,f_j)=0$. Let $T\subset \Sp(V,\omega)$ be the diagonal torus acting on   $\LG(V,\omega)\cong\LG(3,6)$. 
We have the chain of $T$--fixed points
\[
p_0:=\langle f_1,f_2,f_3\rangle=p_-,\quad
p_1:=\langle e_1,f_2,f_3\rangle,\quad
p_2:=\langle e_1,e_2,f_3\rangle,\quad
p_3:=\langle e_1,e_2,e_3\rangle=p_+.
\]
For each $k=0,1,2$, there is a unique $T$--invariant line $L_k\cong \PP^1\subset \LG(3,6)$ joining $p_k$ to $p_{k+1}$, obtained by varying the $(k+1)$--st direction in the
symplectic $2$--plane $\langle e_{k+1},f_{k+1}\rangle$ and keeping the other
two basis directions fixed. 
$\gamma_1$ is  the strict transform of  $L_1$. In particular,  $\gamma_1$ meets both
second--level boundary divisors $D_1^+$ and $D_1^-$ (transversely, at one point each).

Consider the class $C_0=[\gamma_0]$. The line $L_0$ is contained in the second center $Z_1^-$ on the $-$ side. Let $\pi_1\colon X_1\rightarrow\LG(3,6)$ be the blow--up along the strict transforms of $Z_0^+,Z_1^+,Z_0^-$ so that $\TL_3\rightarrow X_1$ is the last step in the iterated blow--up construction. Denote by $\widetilde Z_1^-$ the strict transform of $Z_1^-$ under $\pi_1$. 
It is clear that
\begin{equation}\label{pj}
D_1^+|_{L_2}
\;\cong\;\PP\bigl(N_{\widetilde Z_1^+/\;X_1}\big|_{L_2}\bigr)\ \longrightarrow\ L_2 
\end{equation}
is a $\mathbb P^2$--bundle over $\mathbb P^1$. The $T$--action splits the normal bundle
into weight--subbundles, and there is a distinguished $T$--invariant quotient line bundle, of which the projectivization gives a $T$--invariant section of (\ref{pj}).
This section is the curve $\gamma_2$. 
Symmetrically, the class $C_0=[\gamma_0]$ can be obtained on the $-$ side.

$C_1^+=[\zeta_1^+]$ (resp.\ $C_1^-=[\zeta_1^-]$) is represented by a line of the projectivized tangent space $\PP(T_{p_\pm}\LG(3,6))$ before subsequent blow--ups, and we keep the same notation for its strict transform in $\TL_3$. $C_2^+=[\zeta_2^+]$ (resp.\ $C_2^-=[\zeta_2^-]$) is represented by a line in a fiber
of the exceptional divisor $D_1^+\to \widetilde Z_1^+$ (resp.\ $D_1^-\to \widetilde Z_1^-$), namely a line inside the projectivized normal space to the second center. All $\zeta$--curves are contracted by $\RL_3$.

Let $\ell\in N_1(\TL_3)$ be the pullback of a general line in $\LG(3,6)$, set $e_0^{\pm}: = C_1^{\pm}, e_1^{\pm}: = C_2^{\pm}$, and fix the basis $(\ell,e_0^+,e_1^+,e_0^-,e_1^-)$ of $N_1(\TL_3)$. By Proposition \ref{prop:movcone-TLn-correct}, the extremal rays of $\Mov_1(\TL_3)$ are generated by 
\begin{equation}
\begin{aligned}
&(1,0,0,0,0),\ (2, 0, -1, 0, 0),\ (2, 0, 0, 0, -1),\ (2,0,-1,0,-1),\ (3, -1, 0, 0, 0),\\
&(3, 0, 0, -1, 0),\ (3, -1, 0, -1, 0),\ (6, 0, -3, -2, 0),\ (6, -2, 0, 0, -3).
\end{aligned}
\end{equation}
\end{example}

\begin{remark}[Maple implementation]\label{Maple}
The Maple worksheets used for the polyhedral computations on $\TL_n$ are available at
\begin{center}
\url{https://github.com/msslxa/Kausz}
\end{center}
They implement the grading matrix of the Cox ring of $\TL_n$ and compute the cones
$\Eff(\TL_n)$, $\Mov(\TL_n)$, and $\Nef(\TL_n)$, together with the rays of the cones of effective and moving curves.

\end{remark}

\section{Stable Maps in Lagrangian Grassmannians}\label{Sec3}

Based on \cite{KontsevichManin94,FP97,KV}, we recall certain basic notation and properties of the moduli spaces of stable maps to Lagrangian Grassmannians.

Let $\beta \in H_2(\LG(n,2n),\mathbb{Z})$ be an effective curve class.  Let
$M_{0,k}(\LG(n,2n),\beta)$ be the set of isomorphism classes of pointed maps $(C,p_1,\cdots, p_k,f)$, where $C$ is a nonsingular rational curve, the markings $p_1,\cdots, p_k$ are
distinct points of $C$, and $f$ is a morphism from $C$ to $\LG(n,2n)$ satisfying $f_*([C]) =\beta$.
Let $\Mbar_{0,k}(\LG(n,2n),\beta)$ denote the Kontsevich moduli space of $k$--pointed, genus $0$ stable maps to $\LG(n,2n)$ of class $\beta$. 

Since $\LG(n,2n)$ is convex, namely,  $H^1(\Pbb^1,f^*T_{\LG(n,2n)})=0$ for every morphism $f\colon \Pbb^1 \rightarrow \LG(n,2n)$,  $\Mbar_{0,k}(\LG(n,2n),\beta)$ is normal, irreducible and has the expected dimension:
\begin{equation}\label{eq:dim-M0k}
\dim \Mbar_{0,k}(\LG(n,2n),\beta) \;=\; \dim \LG(n,2n) + c_1(T_{\LG(n,2n)})\cdot \beta + k - 3.
\end{equation}

We will concern the case where $\beta = \beta_n$ is the homology class of the closure of a general $\C^*$--orbit. Equivalently, $\beta_n$ is the unique effective class such that $c_1(\OO_{\LG(n,2n)}(1))\cdot \beta_n = n$.
For notational simplicity, following the usual convention in Gromov--Witten theory, we write $\Mbar_{0,k}(\LG(n,2n),n) := \Mbar_{0,k}(\LG(n,2n),\beta_n)$.


For each $k\geq 1$, there are evaluation morphisms
\[
\ev_i \colon \Mbar_{0,k}(\LG(n,2n),n) \longrightarrow \LG(n,2n),\qquad 1\leq i\leq k,
\]
defined by $\ev_i\bigl((C,p_1,\dots,p_k,f)\bigr) := f(p_i)$. We then have for $k=2$ the evaluation morphism
\[
\ev = (\ev_1,\ev_2) \colon \Mbar_{0,2}(\LG(n,2n),n) \longrightarrow \LG(n,2n)\times \LG(n,2n).
\]
\begin{definition}
For $p,q \in \LG(n,2n)$ we denote by $F_{p,q} := \ev^{-1}(p,q) \subset \Mbar_{0,2}(\LG(n,2n),n)$ the fiber of the two--point evaluation map over $(p,q)$ with the reduced scheme structure.
\end{definition}


In the following, we show that for a general pair $(p,q)$ the fiber $F_{p,q}$ is isomorphic to $\CQ_n$.

Before proceeding, we recall modular properties of the Kontsevich spaces more precisely. A $k$--pointed map $(C,p_1,\cdots, p_k,f)$ is a morphism $f\colon C \to \LG(n,2n)$, where $C$ denotes a tree of projective lines with $k$ distinct marked points that are smooth points of $C$. $(C,p_1,\cdots, p_k,f)$ is isomorphic to $(C^{\prime},p_1^{\prime},\cdots, p_k^{\prime},f^{\prime})$ if there is a scheme isomorphism $\tau\colon C\rightarrow C^{\prime}$ taking $p_i$ to $p^{\prime}_i$ with $f^{\prime}\circ\tau =f$. 
A $k$-pointed map $f\colon C\rightarrow \LG(n,2n)$ is called Kontsevich stable if any irreducible component mapped to a point is stable as a pointed curve.


A family of $k$-pointed maps over a base scheme $B$ is a diagram
\begin{equation}\label{testf}
\begin{CD}
\mathfrak{X} @>{f}>> \LG(n,2n) \\
@V{\pi}VV \\
B
\end{CD}
\end{equation}
together with $k$ disjoint sections $\sigma_1,\dots,\sigma_k \colon B \to \mathfrak{X}$ such that:
\begin{itemize}
    \item[-] $\pi$ is a flat family of trees of smooth rational curves;
    \item[-] the sections $\sigma_i$ do not meet the singularities of the fibers of $\pi$.
\end{itemize}
Thus for each $b \in B$, the restriction $f_b\colon \mathfrak{X}_b \to \LG(n,2n)$ is a $k$-pointed map, with marked points $\sigma_1(b),\dots,\sigma_k(b)$.  Two families of maps over \( B \),
$(\pi\colon\mathfrak X \to B, \{\sigma_i\}, f)$,  $(\pi' : \mathfrak X' \to B, \{\sigma'_i\}, f')$,
are isomorphic if there exists a scheme isomorphism \(\tau : \mathfrak X \to \mathfrak X'\) satisfying: \(\pi = \pi' \circ \tau\), \(\sigma'_i = \tau \circ \sigma_i\), \(f = f' \circ \tau\). 

Let \(\Mbar_{0,n}^*(\LG(n,2n), n) \subset \Mbar_{0,n}(\LG(n,2n), n) \) denote the open locus of stable maps with no nontrivial automorphisms. By \cite[Theorem 2]{FP97},
\( \overline{M}_{0,n}^*(\LG(n,2n), n) \) is a nonsingular, fine moduli space (for automorphism--free stable maps) equipped with a universal family
\begin{equation}
 \Pi\colon \mathcal U\rightarrow \overline{M}_{0,n}^*(\LG(n,2n), n).  
\end{equation}
This means that for any family \( \mathcal{X} \to B \) of such curves, there exists a unique morphism  $\kappa\colon B \to M$ such that the pullback \( \kappa^*\mathcal{U} \) is isomorphic to \( \mathcal{X} \) as families over \( B \).

Applying Lemma \ref{moduli1l}, Proposition \ref{modulil} and marking $p_1 = p_+$, $p_2 = p_-$, we derive from the morphism $\PL_n \colon \TL_n \rightarrow \CQ_n$ a family of $2$--pointed stable maps with fibers embedded into $\LG(n,2n)$ and with the morphism $f$ in (\ref{testf}) given by $\RL_n\colon\TL_n\rightarrow\LG(n,2n)$. Then there is a unique morphism \begin{equation}\label{iem}
i\colon\CQ_n\rightarrow\overline{M}_{0,2}^*(\LG(n,2n), n)\subset \Mbar_{0,2}(\LG(n,2n), n)   
\end{equation} that pulls back the family. The morphism is bijective onto its image by Property (d) in Lemma \ref{moduli1l}. 

\begin{lemma}\label{lem:normal-Zq}
Let $q\in \CQ_n$ be general and let $Z_q\simeq \Pbb^1$ be the corresponding fiber of $\PL_n$.
Then the restriction of the tautological subbundle $\mathcal S$ on $\LG(n,2n)$ satisfies
\[
\mathcal S|_{\RL_n(Z_q)} \ \cong\ \OO_{\Pbb^1}(-1)^{\oplus n}.
\]
Consequently,
\[
T_{\LG(n,2n)}|_{\RL_n(Z_q)} \ \cong\ \OO_{\Pbb^1}(2)^{\oplus \dim \LG(n,2n)}
\qquad\text{and}\qquad
N_{\RL_n(Z_q)/\LG(n,2n)} \ \cong\ \OO_{\Pbb^1}(2)^{\oplus(\dim \LG(n,2n)-1)}.
\]
\end{lemma}

\begin{proof}
By Lemma~\ref{moduli1l}\textup{(c)} the curve $\RL_n(Z_q)$ is a smooth degree $n$ embedding $\Pbb^1\hookrightarrow \LG(n,2n)$ passing through $p_-$ and $p_+$.
By Lemma~\ref{2HM}, after a change of parameter on $\Pbb^1$ we may write it in the big cell around $p_-$ as the row span of a matrix of the form \eqref{fij} with
$f_{ij}(t)=a_{ij}\,t$ for a symmetric matrix $A=(a_{ij})$.
Since the Pl\"ucker coordinate $P_{I^*}$ vanishes to order $n$ at $t=0$, we have $\det(A)\neq 0$, hence $A$ is invertible and the point $t=\infty$ maps to $p_+$.

Writing homogeneous coordinates $[u:v]$ on $\Pbb^1$ (so $t=v/u$), the family of Lagrangian subspaces along the curve is the image of the rank $n$ subbundle
\[
\OO_{\Pbb^1}(-1)^{\oplus n}\ \longrightarrow\ \OO_{\Pbb^1}^{\oplus 2n},
\qquad
(u,v)\ \longmapsto\ \bigl(u\,I_n\ \ \ v\,A\bigr),
\]
whose fiber over $[u:v]$ is exactly the corresponding $n$--plane in $\C^{2n}$.
Therefore $\mathcal S|_{\RL_n(Z_q)}\cong \OO_{\Pbb^1}(-1)^{\oplus n}$.

Finally, for the Lagrangian Grassmannian one has $T_{\LG(n,2n)}\simeq \Sym^2(\mathcal S^\vee)$, hence
$T_{\LG(n,2n)}|_{\RL_n(Z_q)}\cong \OO_{\Pbb^1}(2)^{\oplus \dim \LG(n,2n)}$,
and the normal bundle is obtained by quotienting out $T_{\Pbb^1}\cong \OO_{\Pbb^1}(2)$.
\end{proof}

\begin{theorem}\label{thm:Fpq-is-CQ}
For generic $p,q\in \LG(n,2n)$, $F_{p,q}\subset\overline{M}_{0,2}^*(\LG(n,2n), n)$. Moreover, (\ref{iem}) induces an isomorphism from $\CQ_n$ to $F_{p,q}$.
\end{theorem}
\begin{proof}
There exists an automorphism $g$ of $\LG(n,2n)$ such that $g\cdot p_+=p,g\cdot p_-=q$, where $p_+,p_-$ are defined by (\ref{p+-}). Then the induced action on $\Mbar_{0,2}(\LG(n,2n),n)$ restricts to an isomorphism
between $F_{p_+,p_-}$ and $F_{p,q}$.
Hence, in what follows, we concentrate on $F_{p_{+},p_{-}}$.

It is clear that $i(\CQ_n)\subset F_{p_{+},p_{-}}$. We proceed to prove that
$F_{p_{+},p_{-}}\subset i(\CQ_n)$. 

Take a generic geometric fiber $Z_q\simeq \Pbb^1$ of $\PL_n$, and set $C_q:=\RL_n(Z_q)\subset \LG(n,2n)$.
By \cite{Pe02}, ${\rm Hilb}_{nt+1}(\LG(n,2n))$ is irreducible and smooth.
By Lemma~\ref{lem:normal-Zq} we have
\[
N_{C_q/\LG(n,2n)} \ \cong\ \mathcal{O}(2)^{\oplus\left( \frac{n(n+1)}{2} - 1\right)}.
\]
Hence the tangent space to the Hilbert scheme at $[C_q]$, namely $H^0(C_q, N_{C_q/\LG(n,2n)})$, has dimension $\frac{3n(n+1)}{2} - 3$, which equals $\dim {\rm Hilb}_{nt+1}(\LG(n,2n))$.

Let $M^0_{0,2}(\LG(n,2n),n)\subset\Mbar_{0,2}(\LG(n,2n),n)$ denote the locus of  degree $n$ embeddings $\mathbb P^1\rightarrow \LG(n,2n)$.
Noticing that $\Mbar_{0,2}(\LG(n,2n),n)$ is irreducible and $\dim \Mbar_{0,2}(\LG(n,2n),n)=3\dim \LG(n,2n)  + 2 - 3$ by (\ref{eq:dim-M0k}), we conclude that $M^0_{0,2}(\LG(n,2n),n)$ is an open dense subset of $\Mbar_{0,2}(\LG(n,2n),n)$.

Suppose that there is a point  $x\in F_{p_{+},p_{-}}$ with $x\notin i(\CQ_n)$. Then there exists a sequence of points $y_1,y_2,\cdots\in M^0_{0,2}(\LG(n,2n),n)$ such that $\lim_{i=1}^{\infty}y_i=x$. Following the same argument as in Lemma \ref{density}, we can conclude that for $i$ large enough, there exists an automorphism $\tau$ of $\LG(n,2n)$, which is very close to the identity, such that $\tau(\mathfrak X_{y_i})$ passing through $p_+,p_-$ and $\tau(y_i)\notin i(\CQ_n)$. However, by Lemma \ref{2HM}, we must have $\tau(y_i)\in i(\CQ_n)$, which is a contradiction.

The action defined by \eqref{Gm} induces a $\C^*$--action on $\overline{M}_{0,2}^*(\LG(n,2n), n)$.
Set-theoretically, the $\C^*$--fixed locus in $\overline{M}_{0,2}^*(\LG(n,2n), n)$ is precisely $F_{p_{+},p_{-}}$.

Since $\overline{M}_{0,2}^*(\LG(n,2n), n)$ is smooth and $\C^*$ is linearly reductive, its fixed locus is smooth: at a fixed point $x$, the tangent space of the fixed locus is the invariant subspace $T_x(\overline{M}_{0,2}^*)^{\C^*}$.
Hence $F_{p_{+},p_{-}}$ is smooth.

Finally, the morphism $i\colon\CQ_n\to F_{p_{+},p_{-}}$ is bijective by construction, and it is proper because $\CQ_n$ is projective and $\overline{M}_{0,2}^*(\LG(n,2n), n)$ is separated.
Therefore $i$ is quasi-finite and proper, hence finite.
Since $F_{p_{+},p_{-}}$ is smooth (in particular normal), Zariski's Main Theorem implies that $i$ is an isomorphism.

\end{proof}

Forgetting the $k$--th marked point and stabilizing the domain curve when necessary, we obtain forgetful maps
\[
\phi_k \colon \Mbar_{0,k}(\LG(n,2n),d) \longrightarrow \Mbar_{0,k-1}(\LG(n,2n),d).
\]
Micha{\l}ek--Monin--Wi\'sniewski \cite[Corollary 3.11]{MMW21} realized $\CQ_n$
as the connected component of the $\mathbb C^*$--fixed locus of $\Mbar_{0,0}(\LG(n,2n),n)$. We recover such construction by forgetting the marked points as follows.
\begin{theorem}\label{thm:forgetful-iso-image}
Let $i$ be given by (\ref{iem}). The composition  \begin{equation}\label{12i}
\phi_1\circ\phi_2\circ i\colon \CQ_n\rightarrow\Mbar_{0,0}(\LG(n,2n),n)   
\end{equation}
is a closed embedding. In other words, $\CQ_n$ is realized as a closed subvariety of the unpointed Kontsevich moduli space by forgetting the two marked points of the stable maps passing through $p_+$ and $p_-$.  Consequently, for any generic points $p,q\in\LG(n,2n)$, we have the closed embedding
\[
(\phi_1\circ\phi_2)|_{F_{p,q}}\colon F_{p,q}\longrightarrow \Mbar_{0,0}(\LG(n,2n),n).
\]
\end{theorem}

\begin{proof}
By Theorem~\ref{thm:Fpq-is-CQ}, $F_{p,q}\subset\overline{M}_{0,2}^*(\LG(n,2n), n)$. Since set-theoriatically $\phi_1\circ\phi_2$ is given by forgetting the marked points and stabilizing the domain curve, we conclude that $(\phi_1\circ\phi_2)(F_{p,q})\subset\overline{M}_{0,0}^*(\LG(n,2n), n)$.
Let $\overline{M}_{0,0}^{emb}(\LG(n,2n), n)$ denote the open locus of $\overline{M}_{0,0}(\LG(n,2n), n)$ parametrizing stable maps $f\colon C\rightarrow\LG(n,2n)$ where $f$ is a closed embedding. Then by sending a stable map to its image, we have a natural forgetful morphism 
\begin{equation}\label{00H}
\Phi : \overline{M}_{0,0}^{emb}(\LG(n,2n), n) \longrightarrow \operatorname{Hilb}_{nt+1}(\LG(n,2n)).    
\end{equation}
By Proposition \ref{modulil},   \begin{equation*}
\Phi\circ\phi_1\circ\phi_2\circ i\colon \CQ_n\rightarrow \operatorname{Hilb}_{nt+1}(\LG(n,2n))   
\end{equation*}
is a closed embedding. Then (\ref{12i}) is a closed embedding.
\end{proof}

\begin{corollary}\label{swap}
Let $U\subset \LG(n,2n)\times \LG(n,2n)$ be the open locus of pairs $(p,q)$ for which
Theorems~\ref{thm:Fpq-is-CQ} and~\ref{thm:forgetful-iso-image} hold. Varying $(p,q)\in U$ the images $(\phi_1\circ\phi_2)(F_{p,q})$ sweep an open dense locus in $\Mbar_{0,0}(\LG(n,2n),n)$.
\end{corollary}
\begin{proof}
Since the morphism (\ref{00H}) induces an open embedding on $\overline{M}_{0,0}^*(\LG(n,2n), n)$, Corollary \ref{swap} follows directly from Claim \ref{verygeneral}.
\end{proof}


Set $M_{p,q}:=(\phi_3)^{-1}(F_{p,q})\subset \Mbar_{0,3}(\LG(n,2n),n)$, which parametrizes $3$--pointed stable maps $(C,p_1,p_2,p_3,f)$ such that $f(p_1)=p$ and $f(p_2)=q$. Here $\phi_3\colon \Mbar_{0,3}(\LG(n,2n),n)\rightarrow \Mbar_{0,2}(\LG(n,2n),n)$
is the morphism forgetting the last marked point.

\begin{theorem}\label{thm:Mpq-is-TLn}
For a general pair $(p,q)$ there is an isomorphism $M_{p,q}\ \xrightarrow{\ \sim\ }\ \TL_n$ over $\LG(n,2n)$. Under this identification, the restricted evaluation morphism $\ev_3|_{M_{p,q}}\colon M_{p,q}\to \LG(n,2n)$ equals the blow--up
$\RL_n\colon \TL_n\to \LG(n,2n)$.
\end{theorem}
\begin{proof} Without loss of generality, we may assume that $p=p_+$ and $q=p_-$. Since $F_{p_+,p_-}\subset\Mbar^*_{0,2}(\LG(n,2n),n)$, then by \cite[Theorem 4.5 of Chapter V]{Man}, $M_{p_+,p_-}$ is isomorphic to the universal family over $F_{p_+,p_-}\cong\CQ_n$, which is $\TL_n$. By consider the stabilization process pointwise, it is clear that $ev_3|_{M_{p_+,p_-}}$ coincides with the morphism from the universal family $\TL_n$ to its image in $\LG(n,2n)$, namely the blow--up $\RL_n$. 
\end{proof}

\begin{remark}\label{geom03}
For each $x\in\TL_n$, the
corresponding stable map in $\overline{M}_{0,3}(\LG(n,2n),n)$ is represented by $[\alpha\colon C\rightarrow\LG(n,2n)]$, where $C$ is the fiber of $\PL_n$ over $\PL_n(x)$, $\alpha$ is the restriction of $\RL_n$ to $C$, and the first marking is the unique intersection point $(\PL_n)^{-1}(\PL_n(x))\cap D^+_0$, the second one is $(\PL_n)^{-1}(\PL_n(x))\cap D^-_0$, and the third one is $x$. When $x$ is a node or coincides with the first two markings, one needs to stabilize the domain curve.    
\end{remark}

By Theorem~\ref{thm:Mpq-is-TLn}, we may identify $\TL_n$ with $M_{p,q}\subset \Mbar_{0,3}(\LG(n,2n),n)$. Denote by $\phi\colon\Mbar_{0,3}(\LG(n,2n),n)\rightarrow\Mbar_{0,1}(\LG(n,2n),n)$ the morphism forgetting the first two markings. Then, restricting $\phi$ to $M_{p,q}$ yields a morphism 
$\phi|_{M_{p,q}}\colon \TL_n\cong M_{p,q}\longrightarrow \Mbar_{0,1}(\LG(n,2n),n)$.

\begin{corollary}\label{prop:TLn-in-M01}
For a general pair $(p,q)$, $\phi|_{M_{p,q}}$
is a closed embedding.
\end{corollary}
\begin{proof}
We may assume that $p=p_+$ and $q=p_-$.
Then, we have the following commutative diagram:
\begin{equation}\label{mfpq}
\begin{tikzcd}
\TL_n\ar[d,"\PL_n"]\ar[r,"\sim"]& M_{p_+,p_-}\subset\overline{M}_{0,3}(\LG(n,2n),n) \ar[r, "\phi"] \ar[d, "\phi_3"'] 
& \Mbar_{0,1}(\LG(n,2n),n) \ar[d, "\phi_1"] \ar[r, "\ev_1"]
& \LG(n,2n) \\
\CQ_n\ar[r,"i","\sim"']&F_{p_+,p_-}\subset \overline{M}_{0,2}(\LG(n,2n),n) \ar[r, "\phi_1\circ\phi_2"]
& \Mbar_{0,0}(\LG(n,2n),n)
&
\end{tikzcd}.    
\end{equation}
Since $\phi_1\circ\phi_2\circ i$ is a closed embedding by Theorem \ref{thm:forgetful-iso-image}, then Corollary \ref{prop:TLn-in-M01} follows.
\end{proof}

\begin{corollary}\label{rem:TLn-covers-dense-open-in-M01}
As the pair $(p,q)$ varies, the images $\phi(M_{p,q})$, each isomorphic to $\TL_n$, cover a open dense subset of $\Mbar_{0,1}(\LG(n,2n),n)$.
\end{corollary}
\begin{proof}
By (\ref{mfpq}), $\phi(M_{p,q})=(\phi_1)^{-1}(F_{p,q})$. We then complete the proof by Corollary \ref{swap}.   
\end{proof}

\section{Birational Geometry of Moduli Spaces of Pointed Conics in \texorpdfstring{$\LG(n,2n)$}{LG(n,2n)}}\label{Sec4}

We begin by recalling several divisor classes on the Kontsevich moduli space $\overline{M}_{0,1}(\LG(n,2n),2)$. 

For the Lagrangian Grassmannian $\LG(n,2n)=\LG(V,\omega)$, fix a complete isotropic flag of  subspaces
$\{0\}=F^0\subset F^1 \subset F^2 \subset \dots \subset F^n \subset V$.  Let $\mathcal{D}_n$ be the set of strict partitions $\lambda=(\lambda_1, \dots, \lambda_l)$ with $0 < \lambda_l < \dots < \lambda_1 \le n$ and denote by $\left|\lambda \right|= \lambda_1 + \dots + \lambda_l$ the weight of $\lambda$. For each $\lambda \in \mathcal{D}_n$ there is a codimension--$\left|\lambda \right|$ Schubert variety $\Sigma^n_\lambda \subset \LG(n,2n)$ defined by 
$$\Sigma^n_\lambda:=\{W \in \LG(n,2n),\, \dim(W \cap F^{n+1-\lambda_i}) \ge i,\; i=1, \dots,l\}.$$
Denote by $\sigma_2\in A^2(\LG(n,2n))$ the unique codimension--$2$ Schubert class in the Chow group with rational coefficients. Define the tautological class 
\[
H_{\sigma_2}\ :=\ (\phi_{2})_*(\ev_{2}^*\sigma_2)\ \in\ A^1(\overline{M}_{0,1}(\LG(n,2n),2))\cong\Pic (\overline{M}_{0,1}(\LG(n,2n),2)\otimes\QQ,
\]
where $\phi_{2}\colon\overline{M}_{0,2}(\LG(n,2n),2)\to \overline{M}_{0,1}(\LG(n,2n),2)$ is the morphism forgetting the last marking, and \(\ev_2: \Mbar_{0,2}(\LG(n,2n),2) \to \LG(n,2n)\) is the evaluation morphism at the second marking.  

Define the evaluation class
\begin{equation*}
H_1:=\ev_1^*c_1(\OO_{\LG(n,2n)}(1))\in\Pic (\overline{M}_{0,1}(\LG(n,2n),2)\otimes\QQ,
\end{equation*}
where \(\ev_1\) is the evaluation morphism and \(\OO_{\LG(n,2n)}(1)\) is the Plücker line bundle on $\LG(n,2n)$.

The boundary of $\overline{M}_{0,k}(\LG(n,2n),d)$ is formed by maps whose domains are nodal (reducible) genus $0$ curves, which is a divisor with normal crossings (up to a finite group quotient) by \cite{FP97}. Specifically, $\overline{M}_{0,1}(\mathrm{LG}(n,2n),2)$ contains a unique boundary divisor, $\Delta_{1|1}$. This divisor is the closure of the locus of stable maps $[\alpha\colon(C,x) \to \mathrm{LG}(n,2n)]$ such that
$C=C_1\cup C_2$ has two irreducible components meeting in one node and $\deg(\alpha|_{C_1})=\deg(\alpha|_{C_2})=1$.  Analogously, the unmarked moduli space $\overline{M}_{0,0}(\mathrm{LG}(n,2n),2)$ possesses a unique boundary divisor, which we denote by $\Delta^n$.


In what follows, we shall define the unbalanced divisor $D_{\mathrm{unb}}$ on $\overline{M}_{0,0}(\mathrm{LG}(n,2n),2)$ as a natural effective $\mathbb{Q}$--Cartier divisor. 
Note that since $\overline{M}_{0,0}(\LG(n,2n),2)$ possesses only quotient singularities, it suffices to define $D_{\unb}$ as a Weil divisor.

\begin{lemma}\label{lem:rank-two-factorization}
Let $\alpha\colon\PP^1\to \LG(n,2n) = \LG(V,\omega)$ be a morphism of Pl\"ucker degree $2$. Then there exists an isotropic subspace
$H_{\alpha}\subset V$ of dimension $n-2$ such that $H_{\alpha}\subset \alpha(t)$ for all $t\in \PP^1$. Equivalently, the image of $\alpha$
is contained in the closed subvariety
\[
\LG(n,2n)_{H_{\alpha}}:=\{\,L\in \LG(n,2n)\mid H_{\alpha}\subset L\,\}\ \subset\ \LG(n,2n),
\]
and there is a canonical identification $\LG(n,2n)_{H_{\alpha}} \cong \LG(H_{\alpha}^\perp/H_{\alpha}) \cong \LG(2,4)$.
\end{lemma}

\begin{proof}
Denote by $\mathcal{S}$ the tautological bundle of $\LG(n,2n)$. The pullback of the tautological bundle $E_\alpha := \alpha^*\mathcal{S}$ is a rank $n$ subbundle of the trivial bundle $\mathcal{O}_{\mathbb{P}^1} \otimes \mathbb{C}^{2n}$. Since the degree of $\alpha$ is $2$, $E_\alpha$ splits as $\mathcal{O}_{\mathbb P^1}(-1)^{\oplus2}\oplus\mathcal{O}_{\mathbb P^1}^{\oplus(n-2)}$ or $\mathcal{O}_{\mathbb P^1}(-2)\oplus\mathcal{O}_{\mathbb P^1}^{\oplus(n-1)}$. In particular,  $E_\alpha$ contains a trivial subbundle $\mathcal{O}_{\mathbb P^1}^{\oplus(n-2)}$. Then the linear independent constant sections of $E_{\alpha}$ yield the desired $(n-2)$--dimensional isotropic subspace $H_{\alpha}\subset V$.

Now for each $t\in\mathbb P^1$,  $\alpha(t)$ corresponds to an isotropic $n$--plane $\Lambda_t$ satisfying $H_{\alpha} \subset \Lambda \subset H_{\alpha}^\perp$.
By the symplectic reduction, the space of such isotropic $n$--planes is isomorphic to the Lagrangian Grassmannian of the $4$--dimensional symplectic quotient $H^\perp / H$:
\begin{equation}\label{hper}
\{ \Lambda \in \LG(n,2n) \mid H \subset \Lambda \} \cong \LG(H^\perp/H) \cong \LG(2, 4).     
\end{equation}
We derive the isomorphism $\LG(n,2n)_{H_{\alpha}} \cong \LG(H_{\alpha}^\perp/H_{\alpha}) \cong \LG(2,4)$.
\end{proof}

When \(n=2\), we define \(D_{\unb}\subset \overline{M}_{0,0}(\LG(2,4),2)\) as the closure of the locus of the stable maps $[\alpha\colon\mathbb P^1\to \LG(2,4)]$ such that $E_\alpha \not\cong \mathcal{O}_{\mathbb{P}^1}(-1)^{\oplus 2}$.

When \(n\ge 3\), by associating $\alpha$ to the $(n-2)$--plane
\(H_\alpha\subset V\cong \mathbb{C}^{2n}\),  we derive a rational map to the isotropic Grassmannian of $(n-2)$--planes in $\mathbb C^{2n}$:
\begin{equation}\label{xi}
\begin{split}
\xi\ \ \colon\ \  \Mbar_{0,0}(\LG(n,2n),2) &\dashrightarrow \operatorname{SG}(n-2,2n)\\
[\alpha\colon\mathbb P^1\to \LG(n,2n)]&\mapsto H_{\alpha}   
\end{split}\ \ \ .   
\end{equation}

Recall that the indeterminacy locus of (\ref{xi}) in $M_{0,0}(\mathrm{LG}(n,2n),2)$ consists precisely of stable maps with splitting type $E_\alpha \cong \mathcal{O}(-2) \oplus \mathcal{O}^{\oplus (n-1)}$. Such a map $[\alpha]$ is necessarily a double cover $\mathbb{P}^1 \to \ell \cong \mathbb{P}^1$ of a line $\ell \subset \mathrm{LG}(n,2n)$.
By (\ref{eq:dim-M0k}), the dimension of the space of lines in $\mathrm{LG}(n,2n)$ is $\frac{1}{2}n(n+1)+n-2$. Accounting for the two degrees of freedom in choosing branch points for a double cover $\mathbb{P}^1 \to \mathbb{P}^1$, we find the codimension of this locus in $M_{0,0}(\mathrm{LG}(n,2n),2)$ to be $n-1$, which is at least $2$ for $n \geq 3$.

Moreover, the rational map (\ref{xi}) is well--defined at a generic point of the boundary divisor $\Delta^n$ in the following manner. 
Suppose $[\alpha]$ is a stable map whose domain is a union of two irreducible components $C_1, C_2$ meeting at a node, and whose image is a pair of distinct intersecting lines $L_1, L_2 \subset \mathrm{LG}(n,2n)$.
The intersection of all isotropic subspaces contained in $L_1 \cup L_2$ defines a uniqe $(n-2)$--dimensional isotropic subspace. 

Now we define the unbalanced divisor \(D_{\mathrm{unb}}\) as the pullback of the Schubert divisor class via \(\xi\), namely,
$D_{\mathrm{unb}} := \xi^* \OO_{\operatorname{SG}(n-2,2n)}(1)$.
More explicitly, fix a general \((n+2)\)-dimensional linear subspace \(V_{n+2} \subset V\cong \mathbb C^{2n}\) and let \(\Sigma \subset \operatorname{SG}(n-2,2n)\) be the Schubert divisor defined by the incidence condition \(\dim( H \cap V_{n+2}) \geq1\).  
Away from a closed subset of codimension at least \(2\), the rational map \(\xi\) is a morphism; thus, the preimage \(\xi^{-1}(\Sigma)\) on the regular locus extends uniquely to a Weil divisor on \(\Mbar_{0,0}(\LG(n,2n),2)\).  Notice that the definition of $D_{\mathrm{unb}}$ is independent of the choice of $V_{n+2}$.

It is clear that the restriction of (\ref{xi}) to the boundary divisor $\Delta^n$ is dominant onto $\mathrm{SG}(n-2,2n)$, and hence $D_{\mathrm{unb}}\cap\Delta^n$ is of codimension at least $2$. Therefore, \(D_{\mathrm{unb}} \subset \Mbar_{0,0}(\LG(n,2n),2)\) is the closure of the locus of stable maps \([\alpha \colon \PP^1 \to \LG(n,2n)]\) such that \(\dim(H_\alpha\cap V_{n+2})\geq1\).

\begin{definition}\label{def:Dunb-pointed}
For \(n\ge 2\), we define the unbalanced divisor on $\overline{M}_{0,1}(\LG(n,2n),2)$ as the pullback $\phi_1^*(D_{\unb})$ of the unbalanced locus in $\overline{M}_{0,0}(\LG(n,2n),2)$, where $\phi_1$ is the forgetful morphism. With a slight abuse of notation, we also denote this pullback as $D_{\unb}$.
\end{definition}

\begin{proposition}\label{prop:Dunb-divisor}
For \(n\ge 2\) and $k=0,1$, the locus \(D_{\unb}\subset \overline{M}_{0,k}(\LG(n,2n),2)\) is an irreducible divisor.
\end{proposition}
\begin{proof}
For \(k=0\), we provide a detailed proof based on  \cite[Section~6]{CM} as follows.

When \(n=2\), the divisor \(D_{\unb}\subset \overline{M}_{0,0}(\LG(2,4),2)\) coincides with the
closure of the locus of degree \(2\) stable maps that are double covers of a line
\(\ell\subset \LG(2,4)\). Equivalently, such a map factors as a double cover
\(\mathbb P^1\to \ell\simeq \mathbb P^1\) branched along $2$ points, followed by the inclusion \(\ell\hookrightarrow \LG(2,4)\). The Fano variety of lines in  $\LG(2,4)$ is isomorphic to $\mathbb P^3$. 
For each line $\ell$,
the space of double covers is parameterized by unordered pairs of distinct branch points on $\ell$, which is the irreducible variety 
$((\mathbb P^1\times\mathbb P^1)\backslash\Delta)/\mathbb Z_2$. Thus, $D_{\unb}$ contains an open dense subset that is a  $((\mathbb P^1\times\mathbb P^1)\backslash\Delta)/\mathbb Z_2$--bundle over $\mathbb P^3$.  In particular, $D_{\unb}$ is irreducible.

Assume $n \geq 3$. To establish the irreducibility of $D_{\unb}$, it suffices to show that the locus $D_{\unb}^0$, consisting of stable maps $[\alpha \colon \mathbb{P}^1 \to \mathrm{LG}(n,2n)]$ with splitting type $E_\alpha \cong \mathcal{O}(-1)^{\oplus 2} \oplus \mathcal{O}^{\oplus (n-2)}$, is irreducible. First notice that the Schubert divisor $\Sigma\subset\operatorname{SG}(n-2, 2n)$, defined by the fixed generic vector space $V_{n+2}$, is irreducible. 

Fix an arbitrary point $H \in \Sigma$. Then, $\xi$ is well--defined on $D_{\unb}^0$, and the fiber $\xi^{-1}(H)\cap D_{\unb}^0$ consists of all degree $2$ stable maps $[\alpha \colon \mathbb{P}^1 \to \mathrm{LG}(n,2n)]$ such that  $\alpha(t)$ corresponds to an isotropic $n$--plane $\Lambda_t$ satisfying $H \subset \Lambda \subset H^\perp$, for any $t\in\PP^1$.
By Lemma \ref{lem:rank-two-factorization}, the fiber $\xi^{-1}(H)\cap D_{\unb}^0$ is isomorphic to an open subset of $\Mbar_{0,0}(\LG(2, 4), 2)$. We thus conclude that $D_{\unb}$ is irreducible.

Now consider the pullback $\phi_1^*(D_{\unb})$. Notice that the boundary divisor $\Delta_{1|1}$ is mapped onto the boundary divisor $\Delta^n$ by the forgetful morphism $\phi_1\colon\overline{M}_{0,1}(\LG(n,2n),2)\rightarrow\overline{M}_{0,0}(\LG(n,2n),2)$ .
Since $\Delta^n$ is not contained in $D_{\unb}$,   $\Delta_{1|1}$ is not contained in $\phi_1^*(D_{\unb})$. On the other hand, over the open locus ${M}_{0,0}(\LG(n,2n),2)$, the fiber of the forgetful morphism 
$\phi_1$ is isomorphic to $\PP^1$: forgetting the marked point from a smooth domain curve does not require stabilization in this situation. Then,  $\phi_1^*(D_{\unb})$ is also irreducible.
\end{proof}

Recall the construction in Corollary \ref{rem:TLn-covers-dense-open-in-M01}. Consider the evaluation morphism
\begin{equation*}
\widetilde{\ev}:= (\ev_1,\ev_2)\colon\overline{M}_{0,3}(\LG(2,4),2)\to \LG(2,4)\times \LG(2,4).    
\end{equation*} 
Set $U:=(\LG(2,4)\times \LG(2,4))\setminus \Sigma_1$, where $\Sigma_1\subset \LG(2,4)\times \LG(2,4)$ is the locus of collinear pairs $(p,q)$, namely,  pairs lying on a line
$\ell\subset \LG(2,4)\cong Q^3$.
For $(p,q)\in U$, the fiber $M_{p,q}= f^{-1}(p,q)$ is isomorphic to the blow--up $\TL_2$ by Theorem \ref{thm:Mpq-is-TLn} and Property (b) in Lemma \ref{gspp}. For convenience, via Example \ref{subsec:TL2} we  write $\RL_2\colon \TL_2\to \LG(2,4)$
as $\RL_2\colon\Bl_{\{p,q\}}(Q^3) \rightarrow Q^3$, and the exceptional divisors as $E_p,E_q$. Write $H:=
(\RL_2)^* \mathcal{O}_{\LG(2,4)}(1)$.

By Corollary \ref{prop:TLn-in-M01}, the restriction of the morphism $\phi\colon\Mbar_{0,3}(\LG(2,4),2)\rightarrow\Mbar_{0,1}(\LG(2,4),2)$ forgetting the first two markings
yields a closed embedding
\begin{equation*}
\iota_{p,q}:\Bl_{\{p,q\}}(Q^3) \cong M_{p,q}\ \hookrightarrow\ \overline{M}_{0,1}(\LG(2,4),2).    
\end{equation*}
Furthermore, the family 
\(
\left\{ \iota_{p,q} ( \operatorname{Bl}_{\{p,q\}}(Q^3)) \right\}_{(p,q) \in U}
\)
covers an open dense subset of $\overline{M}_{0,1}(\mathrm{LG}(2,4), 2)$  by Corollary \ref{rem:TLn-covers-dense-open-in-M01}. It is clear that this subset consists of stable maps whose image is not a line.

\begin{lemma}\label{lem:restrictions-k1}
For any $(p,q)\in U$, we have the following equalities in $\Pic(\Bl_{\{p,q\}}(Q^3))\otimes\mathbb{Q}$:
\begin{equation}\label{testn}
\iota_{p,q}^*(H_1)=H,\  \iota_{p,q}^*(\Delta_{1|1})=2\bigl(H-E_p-E_q\bigr),\  \iota_{p,q}^*(D_{\unb})=0.
\end{equation}
\end{lemma}

\begin{proof}
According to  (\ref{mfpq}), we have $\phi^*(H_1)=\ev_3^*\mathcal{O}_{\LG(2,4)}(1)$ on $\overline{M}_{0,3}(\LG(2,4),2)$, where $\ev_3$ is the evaluation at the third marking. By Theorem \ref{thm:Mpq-is-TLn}, on the fiber $M_{p,q}\cong \Bl_{\{p,q\}}(Q^3)$, the map $\ev_3$ coincides with the blow--up $\RL_2$. Therefore, $\iota_{p,q}^*(H_1)=\ev_3^*\mathcal{O}_{\LG(2,4)}(1)|_{M_{p,q}}=
(\RL_2)^* \mathcal{O}_{\LG(2,4)}(1)=H$. 

By Remark \ref{geom03}, for each point $x\in\TL_2\cong\Bl_{\{p,q\}}(Q^3)$, the
corresponding stable map of $\iota_{p,q}(x)\in\overline{M}_{0,1}(\LG(2,4),2)$ is represented by $[\alpha\colon C\rightarrow\LG(2,4)]$, where $C=(\PL_2)^{-1}(\PL_2(x))$, $\alpha=\RL_2|_C$, and the marking is $x$; when $x$ is a node, one needs to stabilize the domain curve. Then, we conclude that $\iota_{p,q}^{-1}(\Delta_{1|1})$ is supported on the union of the boundary divisors $D^+_1$ and $D^-_1$ (see Definition \ref{dk}). Due to the symmetry between $p$ and $q$, the pullback of the divisor class $\Delta_{1|1}$ takes the form $\iota_{p,q}^*(\Delta_{1|1}) = m(D^+_1 + D^-_1)$ for some positive integer $m$.

We claim that $m=1$. It suffices to check that at a generic point of $\iota_{p,q}(D^+_1)$, the image $ \iota_{p,q}( \operatorname{Bl}_{\{p,q\}}(Q^3))$ intersects the boundary divisor $\Delta_{1,1}$ transversely. Take a generic point $x\in D^+_1$. Note that $\Delta_{1|1}$ is smooth at $\iota_{p,q}(x)$, for the corresponding stable map has no nontrivial automorphisms. It then suffices to produce a smooth analytic curve in $\operatorname{Bl}_{\{p,q\}}(Q^3)$ through $x$ whose image under $\iota_{p,q}$ intersects $\Delta_{1,1}$ transversely. Take a smooth analytic curve $\gamma\colon\mathbb D\rightarrow\TL_2$ from a unit disc such that $\gamma(0)=x$ and $\gamma$ intersects with $D^+_1$ transversely at $x$. By shrinking $\mathbb D$, we may further assume that the restriction of the projection $\PL_2$ on $\gamma$ is an embedding. 
Noticing that the universal family $(\PL_2)^{-1}(\PL_2(\gamma(\mathbb D)))$ is smooth, we conclude that $ \iota_{p,q}( \operatorname{Bl}_{\{p,q\}}(Q^3))$ intersects the boundary divisor $\Delta_{1,1}$ transversely at $x$.

Therefore, by (\ref{eq:boundary-last}), we have that $\iota_{p,q}^*(\Delta_{1|1})=D^+_1 + D^-_1=(H-2E_p)+(H-2E_q)=2(H-E_p-E_q)$.

Furthermore, since for $(p,q)\in U$ the image conic of any map in $M_{p,q}$ is not a line, $M_{p,q}$ is disjoint from the unbalanced locus. It follows that $\iota_{p,q}^*(D_{\unb})=0$.
\end{proof}

\begin{proposition}\label{prop:Pic012}
$\{\Delta_{1|1},H_1,H_{\sigma_2}\}$ is a $\QQ$--basis of
$\operatorname{Pic}(\overline{M}_{0,1}(\LG(2,4),2))\otimes\QQ$.
\end{proposition}

\begin{proof}
By \cite[Proposition~1]{Op05}, every class in $\operatorname{Pic}(\overline{M}_{0,1}(\LG(2,4),2))\otimes\QQ$ is generated by the boundary divisor $\Delta_{1|1}$, the evaluation class $H_1$, and the tautological class $H_{\sigma_2}$. 

By a slight abuse of notation, we denote by  $\Delta_{1|1}$, $H_1$, and $D_{\unb}$ their images in the Néron--Severi space $N^1(\overline{M}_{0,1}(\mathrm{LG}(2,4),2))_{\mathbb{R}}$, respectively. It suffices to show that they are $\QQ$--linearly independent.  

Suppose there is a linear relation
\(
aH_1 + b\Delta_{1|1} + cD_{\unb} = 0
\)
in $N^1(\overline{M}_{0,1}(\mathrm{LG}(2,4),2))_{\mathbb{R}}$. By Lemma \ref{lem:restrictions-k1}, for a general pair $(p,q)\in U$, pulling back by $\iota_{p,q}$ yields
\(
a H + 2b (H - E_p - E_q) = 0
\).
Since the classes $H$, $E_p$, $E_q$ are linearly independent in $\mathrm{Pic}(\Bl_{\{p,q\}}(Q^3)) \otimes \mathbb{Q}$, we must have $a=b=0$. On the other hand, since $D_{\unb}$ is effective and nonempty, we have $c=0$. The proof is complete.
\end{proof}

\begin{theorem}\label{thm:Eff-M01-LG24-k1}
The cone of effective divisor classes of $\overline{M}_{0,1}(\LG(2,4),2)$ is
\[
\operatorname{Eff}(\overline{M}_{0,1}(\LG(2,4),2))
=
\langle H_1, \Delta_{1|1}, D_{\unb}\rangle
\ \subset\ N^1(\overline{M}_{0,1}(\LG(2,4),2))_{\mathbb{R}}.
\]
\end{theorem}
\begin{proof}
Since $H_1$, $\Delta_{1|1}$, and $D_{\unb}$ are effective,  $\langle H_1, \Delta_{1|1}, D_{\unb}\rangle\subset\operatorname{Eff}(\overline{M}_{0,1}(\LG(2,4),2))$.
To prove the reverse inclusion, we take an arbitrary prime effective divisor 
$D\subset \overline{M}_{0,1}(\LG(2,4),2)$ that is different from $D_{\unb}$ and $\Delta_{1|1}$. According to the proof of Proposition \ref{lem:rank-two-immersion}, we express $D$ in the Néron--Severi space as the linear combination 
$D\equiv \alpha H_1 + \beta \Delta_{1|1} + \gamma D_{\mathrm{unb}}$.

Fix a smooth conic $C$ in $\LG(2,4)$. Define a curve  $\Gamma\subset\overline{M}_{0,1}(\LG(2,4),2)$ as the locus of stable maps $[f\colon (\PP^1,x)\rightarrow \LG(2,4)]$ such that $f$ is an isomorphism from $\PP^1$ to $C$. Computing the dimensions of the Kontsevich spaces and the Hilbert schemes as in the proof of Theorem \ref{thm:Fpq-is-CQ}, we can conclude by Claim \ref{verygeneral} and (\ref{00H}) that 
the locus of stable maps $[f\colon (\PP^1,x)\rightarrow \LG(2,4)]$ such that $f$ is an embedding and $f(\PP^1)$ passes through a pair of generic points is open in $\overline{M}_{0,1}(\mathrm{LG}(2,4),2)$. Then, it is clear that $\Gamma$ is a moving curve with $\Gamma\cdot H=2$ and $\Gamma\cdot \Delta_{1|1}=\Gamma\cdot D_{\unb}=0$. Therefore $\alpha\geq0$.

Choose a general pair $(p,q)$ so that $\iota_{p,q}(\Bl_{\{p,q\}}(Q^3))\not\subset D$. Then the pullback $\iota_{p,q}^*(D)$ is an effective divisor on $\Bl_{\{p,q\}}(Q^3)$.
We claim that $\beta\geq0$. Otherwise, we take the effective curve $\zeta^+_1$ as in Definition \ref{def:T-curves-correct}. By Lemma \ref{lem:intersection-table-correct}, we get that
$\zeta^+_1\cdot H=0$, $\zeta^+_1\cdot (H-E_p-E_q)=1$, and hence
$\zeta^+_1\cdot\iota_{p,q}^*(D)=\beta<0$. Then $\iota_{p,q}(\zeta^+_1)\subset D$. 
Take a generic point $[f\colon (\PP^1\cup\PP^1,x)\rightarrow \LG(2,4)]$ of $\Delta_{1|1}$ which is not in $D$, and a pair of points $(f(x),q)$ that can be transformed to $(p_+,p_-)$ by an automorphism of $\LG(2,4)$. Then $\iota_{x,q}([\alpha\colon (\PP^1,x)\rightarrow \LG(2,4)])\in D$, which is a contradiction.

Finally, we shall prove that $\gamma\geq 0$ by the following morphism from the Hirzebruch surface $F_1$ to the quadric threefold $Q\subset\PP^4$ defined by $z_0z_4-z_1z_3+z_2^2=0$.

Let $\Sigma\subset N_\mathbb{R}$ be the fan of $X_\Sigma\cong F_1$, with primitive generators of rays $\mathbf{v}_a=(1,1)$, $\mathbf{v}_v=(0,1)$, $\mathbf{v}_b=(-1,0)$ and $\mathbf{v}_u=(0,-1)$. The lattice $N\cong\mathbb{Z}^2$ is a sublattice of $N'=N+(\frac12\mathbb{Z},0)$ of index $2$. Deonte $\Sigma'$ the same fan but in $N'_\mathbb{R}$, then the toric variety $X_{\Sigma'}$ associated to $\Sigma'$ is a weighted blow up of $\mathbb{P}(1,1,2)$. $(N,\Sigma)\to(N',\Sigma')$ is a double cover in the sense of \cite[Definition~3.2]{AlexeevPardini2013ramified}, so on geometry aspect we obtain the double cover $f\colon X_{\Sigma}\to X_{\Sigma'}$. For $\Sigma'$, remove the ray generated by $\mathbf{v}_v$, which corresponds to the contraction of the $(-\frac12)$-section on $X_{\Sigma'}$, denoted by $g\colon X_{\Sigma'}\to\mathbb{P}(1,1,2)$. Since $\mathbb{P}(1,1,2)=\mathbf{V}(z_1z_3-z_2^2)\subset\mathbb{P}^3_{[z_0:z_1:z_2:z_3]}$, $g\circ f$ maps $F_1$ to $\mathbb{P}(1,1,2)\subset Q\subset\mathbb{P}^4$. Here we view $\mathbb{P}(1,1,2)$ as a generic conic ruling over a line instead of a ruling over a conic curve.

Equivalently, by Cox quotient construction of toric varieties, the map $F_1\to \mathbb{P}^4$ can also be interpreted via the lift $\mathbb{A}^4_{a,b,u,v}\to\mathbb{P}^4$. Explicitly, the quotient map is $(\mathbb{A}^4_{a,b,u,v}\setminus\mathbf{V}(B(\Sigma)))/G\to F_1$, where $G\cong(\mathbb{C}^*)^2$ and the unstable locus $\mathbf{V}(B(\Sigma))=\mathbb{A}_{a,b}^2\cup\mathbb{A}^2_{u,v}$. The weights of the $G$-action are $\mathrm{wt}(a)=\mathrm{wt}(b)=(1,0)$, $\mathrm{wt}(u)=(0,1)$ and $\mathrm{wt}(v)=(-1,1)$. So the lifted morphism
\begin{align*}
\mathbb{A}^4_{a,b,u,v}\setminus\mathbf{V}(B(\Sigma))&\to\mathbb{P}^3_z,\\
(a,b,u,v)&\mapsto[u^2:0:buv:a^2v^2:b^2v^2]
\end{align*}
descends to $F_1\to Q\subset\mathbb{P}^4$. It is a double cover and the contraction is $(\mathbb{A}^3_{a,b,u}\setminus\mathbf{V}(B(\Sigma)))/G\mapsto[u^2:0:0:0:0]$. Therefore, we can construct a family of stable map $C\colon\PP^1\rightarrow \Mbar_{0,0}(\LG(2,4),2)$ with the contracted curve as the marking. Since the generic fiber is a smooth conic and a special fiber is a double cover of a line,  $\gamma\geq 0$.

The proof is complete.
\end{proof}

\begin{remark}
In general, determining an extremal ray of the effective cone requires a suitable moving curve, as in \cite[Lemma 2.4]{CS06}. However, such a moving curve may not admit a constant marking. In the proof above, we therefore use a non--moving test curve to establish that $\beta \geq 0$, by exploiting the geometry of $\TL_2$.
\end{remark}

The following results aim to establish a link between the divisor theory of $\overline{M}_{0,1}(\LG(2,4),2)$ and that of $\overline{M}_{0,1}(\LG(n,2n),2)$ for $n\geq 3$.  
Fix an isotropic $(n-2)$--plane $A_0\subset V$. The inclusion $\imath_{A_0}\colon\LG(n,2n)_{A_0}\hookrightarrow \LG(n,2n)$ yields a morphism
\[
\jmath_{A_0,k}\colon\overline{M}_{0,k}(\LG(2,4),2)\ \cong\ \overline{M}_{0,k}(\LG(n,2n)_{A_0},2)\ \longrightarrow\ \overline{M}_{0,k}(\LG(n,2n),2)
\]
for each $k\ge 0$, 
whose image is the locus of stable maps whose image is contained in $\LG(n,2n)_{A_0}$.
Then, 
\begin{proposition}\label{lem:rank-two-immersion} The pullback homomorphism
\[
\jmath_{A_0,1}^*\colon\operatorname{Pic}(\overline{M}_{0,1}(\LG(n,2n),2))\otimes\QQ
\ \longrightarrow\
\operatorname{Pic}(\overline{M}_{0,1}(\LG(2,4),2))\otimes\QQ
\]
is an isomorphism which sends the divisor classes $H_1$, $\Delta_{1|1}$, $H_{\sigma_2}$ on $\overline{M}_{0,1}(\LG(n,2n),2)$ to $H_1$, $\Delta_{1|1}$, $H_{\sigma_2}$ on $\overline{M}_{0,1}(\LG(2,4),2)$, respectively. Consequently, $\{\Delta_{1|1},H_1,H_{\sigma_2}\}$ is a $\QQ$--basis of $\operatorname{Pic}(\overline{M}_{0,1}(\LG(n,2n),2))\otimes\QQ$. The induced map on numerical divisor classes 
\begin{equation*}
\overline\jmath^*_{A_0,1}:N^1(\overline{M}_{0,1}(\LG(n,2n),2))_{\R}\longrightarrow N^1(\overline{M}_{0,1}(\LG(2,4),2))_{\R}  
\end{equation*} 
is an isomorphism that sends the divisor class  $D_{\unb}$ on $\overline{M}_{0,1}(\LG(n,2n),2)$ to the class  $D_{\unb}$ on $\overline{M}_{0,1}(\LG(2,4),2)$. 
\end{proposition}
\begin{proof}
The proof is analogous to the argument for unmarked moduli spaces given in \cite[Proposition 6.4]{CM}.

By \cite[Proposition~1]{Op05}, $\operatorname{Pic}\bigl(\overline{M}_{0,1}(\LG(n,2n),2)\bigr)\otimes\QQ$ is generated by $[\Delta_{1|1}]$,  $H_1$, and  $H_{\sigma_2}$.

Since  $\imath_{A_0}$ pulls back the unique codimension--$2$ Schubert class of $\LG(n,2n)$ to the unique codimension--$2$ Schubert class of $\LG(2,4)$, $\jmath_{A_0,1}$ pulls back the tautological class $H_{\sigma_2}$ of $\LG(n,2n)$ to that of $\LG(2,4)$.

It is clear that
the pullback of the boundary divisor $\Delta_{1|1}$ of $\LG(n,2n)$ is supported on the boundary divisor $\Delta_{1|1}$ of $\LG(2,4)$. By the same argument as in  the proof of Lemma
\ref{lem:restrictions-k1}, we can show that the multiplicity is $1$. Hence, $\jmath_{A_0,1}$ pulls back the boundary class $[\Delta_{1|1}]$ of $\LG(n,2n)$ to that of $\LG(2,4)$.

According to commutative diagram
\begin{equation*}
\begin{tikzcd}
\overline{M}_{0,1}(\LG(n,2n)_{A_0},2) \ar[r, "\ev_1"] \ar[d, "\jmath_{A_0,1}"'] 
& \LG(n,2n)_{A_0} \ar[d, "\imath_{A_0}"'] \\
 \overline{M}_{0,1}(\LG(n,2n),2) \ar[r, "\ev_1"]
& \LG(n,2n)
\end{tikzcd},   
\end{equation*}
we conclude that $\jmath_{A_0,1}$ pulls back the evaluation class $H_1$ of $\LG(n,2n)$ to that of $\LG(2,4)$.

Then, by Proposition \ref{prop:Pic012}, we conclude that both $\jmath_{A_0,1}^*$ and $\overline{\jmath}_{A_0,1}^*$ are isomorphisms. Consequently, the set $\{\Delta_{1|1}, H_1, H_{\sigma_2}\}$ forms a basis for $\operatorname{Pic}\bigl(\overline{M}_{0,1}(\mathrm{LG}(n,2n),2)\bigr) \otimes \mathbb{Q}$.

Recall that $D_{\unb}$ on $\overline{M}_{0,1}(\LG(n,2n),2)$ is defined as the pullback of the unbalanced divisor on $\overline{M}_{0,0}(\LG(n,2n),2)$ under the forgetful morphism. Fix an $(n-2)$--dimensional  isotropic subspace $A_0$. Since forgetful morphisms commute with $\jmath_{A_0,k}$, it suffices to prove that $\overline\jmath_{A_0,0}^*(D_{\unb})$ equals the unbalanced class $D_{\unb}$ on $\overline{M}_{0,0}(\LG(n,2n)_{A_0},2)$.

Choose a generic $(n+2)$--dimensional vector subspace $V_{n+2}\subset V\cong\mathbb C^{2n}$ such that $A_0\cap V_{n+2}=\{0\}$. We define \(D_{\mathrm{unb}} \subset \Mbar_{0,0}(\LG(n,2n),2)\) as the closure of the locus of stable maps \([\alpha \colon \PP^1 \to \LG(n,2n)]\) such that \(\dim(H_\alpha\cap V_{n+2})\geq1\). Let $[\beta \colon \PP^1 \to \LG(n,2n)_{A_0}]$ be a generic point of $D_{\unb}$ on $\overline{M}_{0,0}(\LG(n,2n)_{A_0},2)$. Then $\beta$ determines an $(n-1)$--dimensional  isotropic subspace $\widehat A_0$ that contains $A_0$. Take another $(n-2)$--dimensional  isotropic subspace $\check A_0\subset \widehat A_0$ such that $\dim (\check A_0\cap V_{n+2})\geq 1$. Then $[\beta \colon \PP^1 \to \LG(n,2n)_{A_0}]$ is also contained in the image of $\overline{M}_{0,0}(\LG(n,2n)_{\check A_0},2)$ under $\jmath_{\check A_0,0}$. Notice that the generic point of
$\jmath_{\check A_0,0}(\overline{M}_{0,0}(\LG(n,2n)_{\check A_0},2))$ is contained in the locus of stable maps \([\alpha \colon \PP^1 \to \LG(n,2n)]\) such that \(\dim(H_\alpha\cap V_{n+2})\geq1\). We conclude that $\jmath_{\check A_0,0}(\beta)=\jmath_{A_0,0}(\beta)$ is contained in \(D_{\mathrm{unb}} \subset \Mbar_{0,0}(\LG(n,2n),2)\). Since \(D_{\unb}\subset \overline{M}_{0,0}(\LG(n,2n),2)\) is irreducible for any $n\geq2$ by Proposition \ref{prop:Dunb-divisor}, we have that the divisor class
$\overline\jmath_{A_0,0}^*(D_{\unb})$ equals a multiple of $D_{\unb}$ on $\overline{M}_{0,0}(\LG(n,2n)_{A_0},2)$.

Observing that $\jmath_{A_0,0}^*$ maps the basis elements $\Delta_{1|1}$ and $H_{\sigma_2}$ of $\overline{M}_{0,0}(\mathrm{LG}(n,2n),2)$ to their respective counterparts on $\overline{M}_{0,0}(\mathrm{LG}(2,4),2)$, we conclude that $\overline\jmath_{A_0,0}^*(D_{\unb})$ must equal $D_{\unb}$ on $\overline{M}_{0,0}(\LG(n,2n)_{A_0},2)$.

The proof of Proposition \ref{lem:rank-two-immersion} is complete. 
\end{proof}

\begin{theorem}\label{thm:Eff-general-n-k1}
For every $n\ge 2$, the cone of effective divisors on $\overline{M}_{0,1}(\LG(n,2n),2)$ is given by
\[
\operatorname{Eff}(\overline{M}_{0,1}(\LG(n,2n),2))\ =\ \bigl\langle H_1, \Delta_{1|1}, D_{\unb}\bigr\rangle.
\]
\end{theorem}

\begin{proof}
Let $E_n\subset N^1(\overline{M}_{0,1}(\LG(n,2n),2))_{\mathbb{R}}$ be the cone generated by $H_1$, $\Delta_{1|1}$, and $D_{\unb}$, whose imaages under the isomorphism $\jmath_{A_0,1}^*$ are $H_1$, $\Delta_{1|1}$, $H_{\sigma_2}$ on $\overline{M}_{0,1}(\LG(2,4),2)$, respectively. It is clear that $E_n\subset \Eff(\overline{M}_{0,1}(\LG(n,2n),2))$. 

By Lemma \ref{lem:rank-two-factorization} $\overline{M}_{0,1}(\LG(n,2n),2)$ is covered by the subvarieties
\[
Y_{A_0}\ :=\ \jmath_{A_0,1}(\overline{M}_{0,1}(\LG(2,4),2))\ \subset\ \overline{M}_{0,1}(\LG(n,2n),2),
\]
where $A_0$ are isotropic $(n-2)$--planes.
For any prime divisors $[D]\in \Eff(\overline{M}_{0,k}(\LG(n,2n),2))$, we may choose $A_0$ such that $Y_{A_0}\not\subset \Supp(D)$.
Then $\jmath_{A_0,1}^*[D]\ =\ [D\cap Y_{A_0}]\in \Eff(Y_{A_0}) \cong\Eff(\overline{M}_{0,1}(\LG(2,4),2))$. By Theorem \ref{thm:Eff-M01-LG24-k1} $\Eff(\overline{M}_{0,1}(\LG(2,4),2)) =\bigl\langle H_1,\Delta_{1|1},D^{\mathrm{unb}}\bigr\rangle$, and by Proposition \ref{lem:rank-two-immersion}  the pullback $\jmath_{A_0,1}^*$ is an isomorphism sending $E_n$ onto this cone.
Thus $\jmath_{A_0,1}^*[D]\in \jmath_{A_0,1}^*(E_n)$ implies $[D]\in E_n$. Therefore $\Eff(\overline{M}_{0,1}(\LG(n,2n),2))\subset E_n$.
\end{proof}



Fix a general hyperplane section $H\subset \LG(n,2n)$. Let $T\subset \overline{M}_{0,0}(\LG(n,2n),2)$ be the divisor of conics tangent to $H$. We will also denote by $T$ its pullback to $\overline{M}_{0,1}(\LG(n,2n),2)$. Equivalently, $T$ is the locus of pointed maps $[f\colon(C,x_1)\to \LG(n,2n)]$ such that the image curve
$f(C)$ is tangent to the fixed hyperplane section.

\begin{theorem}\label{thm:nef-pointed-conics-LG}
For every $n\ge 2$ the nef cone of $\overline{M}_{0,1}(\LG(n,2n),2)$ is given by
\[
\operatorname{Nef}(\overline{M}_{0,1}(\LG(n,2n),2))
=
\langle H_1, H_{\sigma_2}, T\rangle.
\]
\end{theorem}
\begin{proof}
The class $H_1=\ev_1^*\OO_{\LG(n,2n)}(1)$ is nef for $\OO_{\LG(n,2n)}(1)$ is ample. By \cite[Proposition~6.11]{CM} the nef cone of $\overline{M}_{0,0}(\LG(n,2n),2)$ is generated by $H_{\sigma_2}$ and $T$. Thus $\langle H_1, H_{\sigma_2}, T\rangle \subset \Nef(\overline{M}_{0,1}(\LG(n,2n),2))$.

We first assume that $n = 2$. Take an arbitrary nef divisor class $[D]\in\Pic(\overline{M}_{0,1}(\LG(2,4),2))$. By (6.12)
in \cite{CM}, $H_{\sigma_2}=\frac{1}{2}\Delta_{1|1}+D_{unb}$ and $T=\Delta_{1|1}+D_{unb}$, and hence we can write 
\begin{equation*}
    D\equiv aH_1+bH_{\sigma_2}+cT.
\end{equation*}
According to Theorem \ref{thm:Eff-M01-LG24-k1}, it is clear that $a\geq 0$.

By \cite[Proposition~6.11]{CM}, the divisor $H_{\sigma_2}$ induces a birational contraction $f_{H_{\sigma_2}}\colon\overline{M}_{0,0}(\LG(2,4),2)\rightarrow Z$ whose exceptional locus is the divisor $Q^2(1)$ parameterizing double covers of a line in $\LG(2,4)$.
Let $\ell\subset \LG(2,4)$ be a general line, and let $\Gamma\subset Q^2(1)$ be a curve contained in the fiber of
$f_{H_{\sigma_2}}$ over the point of $Z$ corresponding to $\ell$, so $\Gamma$ varies the double cover while keeping
the image line $\ell$ fixed. Choose a lift $\widetilde{\Gamma}\subset \overline{M}_{0,1}(\LG(2,4),2)$ of $\Gamma$ such that the marked point
maps to a fixed point of $\ell$. Then $H_1\cdot \widetilde{\Gamma}=0$ and $H_{\sigma_2}\cdot \widetilde{\Gamma}=0$, and hence $c\geq0$.

By \cite[Proposition~6.11]{CM} again, the divisor $T$ induces a birational morphism $f_{T}\colon\overline{M}_{0,0}(\LG(2,4),2)\rightarrow W$ which contracts the locus of maps with reducible 
domain $[C_1\cup C_2, \alpha]$ to $\alpha(C_1\cap C_2)$.
Take a point $p\in f_T(\Delta_{1|1})\subset \LG(2,4)$, and let $\Gamma\subset\Delta_{1|1}$ be a curve contained in the fiber of
$f_T$ over $p$. For the family of curves $\Gamma$,
take the mark points to be $C_1\cap C_2$ and then stabilize the universal curves. Then we have a lift $\widetilde{\Gamma}\subset \overline{M}_{0,1}(\LG(2,4),2)$ of $\Gamma$ such that the marked point
maps to the fixed point  $p$. It is clear that $H_1\cdot \widetilde{\Gamma}=T\cdot \widetilde{\Gamma}=0$, and hence $b\geq0$.

Finally, to get the claim for any $n\geq 2$ it is enough to argue as in the proof of Theorem \ref{thm:Eff-general-n-k1}.
\end{proof}

\noindent
\begin{minipage}{\linewidth}
    \centering
    \vspace{1pt}
    \begin{tikzpicture}[scale=0.27,
      every node/.style={font=\fontsize{1.5}{1.5}\selectfont, inner sep=0.1pt, outer sep=0pt}]

      \pgfmathsetmacro{\n}{3} 
      \pgfmathsetmacro{\Kx}{24/(\n+9)}
      \pgfmathsetmacro{\Ky}{12*(\n+3)/(\n+9)}

      \coordinate (H1)   at (-12,0);
      \coordinate (Del)  at (12,0);
      \coordinate (Dunb) at (0,12);

      \coordinate (T)    at (6,6);  
      \coordinate (Hsig) at (4,8);  

      \coordinate (K) at (\Kx,\Ky);

      \fill[gray!15] (H1) -- (Dunb) -- (Del) -- cycle;
      \draw[thick]   (H1) -- (Dunb) -- (Del) -- cycle;

      \fill[gray!55] (H1) -- (Hsig) -- (T) -- cycle;
      \draw[thick]   (H1) -- (Hsig) -- (T) -- cycle;

      \draw[thin]
        (H1) -- (Hsig)
        (H1) -- (T)
        (Del) -- (T)
        (Dunb) -- (Hsig);

      \fill (H1) circle (0.22);
      \node[below left, xshift=-1pt, yshift=1pt] at (H1) {$H_1$};

      \fill (Del) circle (0.22);
      \node[below right, xshift=3pt, yshift=1pt] at (Del) {$\Delta_{1|1}$};

      \fill (Dunb) circle (0.22);
      \node[above, yshift=2.9pt] at (Dunb) {$D_{\unb}$};

      \fill (Hsig) circle (0.22);
      \node[left, xshift=17pt] at (Hsig) {$H_{\sigma_2}$};

      \fill (T) circle (0.22);
      \node[right, xshift=3.5pt] at (T) {$T$};


    \end{tikzpicture}
    \vspace{-5pt}
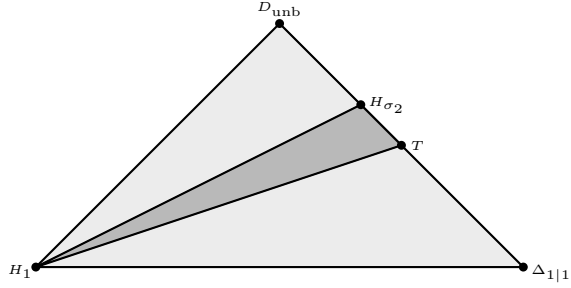
    \captionof*{figure}{A $2$--dimensional section of $\Eff(\overline M_{0,1}(\mathrm{LG}(n,2n),2))$ (light grey) and $\Nef(\overline M_{0,1}(\mathrm{LG}(n,2n),2))$ (dark grey).}    
\end{minipage}


\begin{proposition}\label{-K}
The anticanonical divisor of $\overline{M}_{0,1}(\mathrm{LG}(n,2n),2)$ is given by 
$$
-K_{\overline{M}_{0,1}(\mathrm{LG}(n,2n),2)} = \left\lbrace
\begin{array}{lcll}
H_1 + \frac{5}{2}H_{\sigma_2} + \frac{3}{4}\Delta_{1|1} & \equiv & H_1 + H_{\sigma_2} + \frac{3}{2}T & \text{if } n = 2\\ 
H_1 + \frac{n+2}{2}H_{\sigma_2} + \frac{6-n}{4}\Delta_{1|1} & \equiv & H_1+(n-2)H_{\sigma_2}+\frac{6-n}{2}T & \text{if } n \geq 3
\end{array}\right. .
$$
\end{proposition}
\begin{proof}
Let $\overline{\mathcal{M}}_{0,0}(\mathrm{LG}(n,2n),2)$ be the smooth Kontsevich stack of conics in $\mathrm{LG}(n,2n)$,  $\overline H_{\sigma_2}$, $\overline T$, $\overline\Delta_{1|1}$, $\overline D_{\mathrm{unb}}$ the divisors on $\overline{\mathcal{M}}_{0,0}(\mathrm{LG}(n,2n),2)$  corresponding to $H_{\sigma_2}$, $T$, $\Delta_{1|1}$, $D_{\mathrm{unb}}$ respectively.
By \cite[Proposition 6.15]{CM}, 
$$
\begin{array}{ll}
\overline\Delta_{1|1} \equiv 2(\overline T-\overline H_{\sigma_2}),\quad \overline D_{\unb} \equiv \overline H_{\sigma_2}-\frac{1}{2}\overline T,\quad -K_{\overline{\mathcal{M}}_{0,0}(\mathrm{LG}(n,2n),2)} \equiv 5\overline H_{\sigma_2}-5 \overline D_{\unb} & \text{if } n = 2;\\
\overline \Delta_{1|1} \equiv 2(\overline T-\overline H_{\sigma_2}),\quad \overline D_{\unb} \equiv 2\overline H_{\sigma_2}-\overline T,\quad -K_{\overline{\mathcal{M}}_{0,0}(\mathrm{LG}(n,2n),2)} \equiv 5\overline H_{\sigma_2}+\frac{n-7}{2} \overline D_{\unb} & \text{if } n \geq 3.
\end{array}
$$

Let $\phi_2:\overline{M}_{0,2}(\LG(n,2n),2)\to \overline{M}_{0,1}(\LG(n,2n),2)$ be the morphism forgetting the second marking, and let $s_1$ be the section corresponding to the first marking. The cotangent line class is $\psi_1:=c_1(s_1^*\omega_{\phi_2})$, where $\omega_{\phi_2}$ is the relative dualizing sheaf of $\phi_2$.
Let $\pi:\overline{\mathcal{M}}_{0,1}(\mathrm{LG}(n,2n),2)\longrightarrow \overline{\mathcal{M}}_{0,0}(\mathrm{LG}(n,2n),2)$ be the morphism forgetting the marking. By \cite[Theorem 1.1]{dJS17}, the canonical bundle formula reads
$$
K_{\overline{\mathcal{M}}_{0,1}(\mathrm{LG}(n,2n),2)}\ \equiv\ \pi^*K_{\overline{\mathcal{M}}_{0,0}(\mathrm{LG}(n,2n),2)}+\psi_1
\qquad\text{in }\Pic(\overline{\mathcal{M}}_{0,1}(\mathrm{LG}(n,2n),2))\otimes\QQ.
$$
On a generic fiber \(\pi^{-1}([f])\cong \mathbb P^1\) (moving only the marking on a fixed degree $2$ map \(f\)), we have \(\deg(H_1)=2\) and \(\deg(\psi_1)=\deg(\omega_{\mathbb P^1})=-2\), while \(T\) is pulled back from the unpointed space so it has degree \(0\) on the fiber. Hence \(\psi_1\equiv-H_1+\pi^*(D)\) for some divisor \(D\) downstairs.

To identify \(D\), we fix a Plücker hyperplane section $H\subset \LG(n,2n)$ with defining linear form $\ell$. On the open locus where the domain is irreducible, for a pointed map $(f,x)$ we can choose a local coordinate $z$ on the source at the marked point $x$ (so $z(x)=0$) and write the pullback
\(
(\ell\circ f)(z)=a+bz+cz^2 .
\)
Intrinsicly, $a$ is the evaluation at the marking, hence a local trivialization of $\ev_1^*\mathcal O_{\LG(n,2n)}(1)$ (i.e. of $H_1$); $b$ is the first jet at the marking, hence lies in $\ev_1^*\mathcal O_{\LG(n,2n)}(1)\otimes \psi_1$; $c$ is the second jet, hence lies in $\ev_1^*\mathcal O_{\LG(n,2n)}(1)\otimes \psi_1^{\otimes 2}$. Equivalently, under the reparametrization $z\mapsto \lambda z$ (automorphisms fixing the marked point) we have $b\mapsto \lambda b$, $c\mapsto \lambda^2 c$, and under rescaling $\ell\mapsto \kappa \ell$ we have $(a,b,c)\mapsto \kappa(a,b,c)$.
Therefore the discriminant
\(
\operatorname{Disc}(\ell\circ f):=b^2-4ac
\)
is a well-defined global section of
\[
(\ev_1^*\mathcal O_{LG}(1))^{\otimes 2}\otimes \psi_1^{\otimes 2}
\cong \mathcal O(2H_1+2\psi_1).
\]
By definition of the tangency divisor $T$ (the locus where $\ell\circ f$ does not have two distinct smooth zeros on the domain), we have
\(
T = \operatorname{div}(\operatorname{Disc}(\ell\circ f)),
\)
hence $\mathcal O(T)\cong \mathcal O(2H_1+2\psi_1)$ and so \(\psi_1\equiv -H_1+\tfrac12 T\).

Note that the locus of points of $\overline{\mathcal{M}}_{0,1}(\mathrm{LG}(n,2n),2)$ with non trivial stabilizer has codimension $>1$.  These numerical equivalences descend unchanged to the coarse moduli space $\overline{M}_{0,1}(\LG(n,2n),2)$, and rewriting them in the basis $\{H_1,H_{\sigma_2},\Delta_{1|1}\}$ gives the stated formulas.
\end{proof}

\begin{corollary}\label{cor:Fano-weakFano-MDS}
The Kontsevich space $\overline{M}_{0,1}(\mathrm{LG}(n,2n),2)$ is Fano if and only if $2\le n\le 5$, and weak Fano for $n=6$. Furthermore, $\overline{M}_{0,1}(\mathrm{LG}(n,2n),2)$ is a Mori dream space for every $n\ge 2$.
\end{corollary}
\begin{proof}
From the explicit determination of $\Eff(\overline{M}_{0,1}(\mathrm{LG}(n,2n),2))$ and $\Nef(\overline{M}_{0,1}(\mathrm{LG}(n,2n),2))$ in Theorems \ref{thm:Eff-general-n-k1}, \ref{thm:nef-pointed-conics-LG} it follows that $\Mov(\overline{M}_{0,1}(\mathrm{LG}(n,2n),2))=\Nef(\overline{M}_{0,1}(\mathrm{LG}(n,2n),2))$. Indeed, if $D$ is effective but not nef, then $D$ has negative intersection with one of the curve classes used
to cut out the nef cone in Theorem~\ref{thm:nef-pointed-conics-LG}, which move in a divisor (either
$\Delta_{1|1}$ or $D_{\mathrm{unb}}$);
hence $D$ is not movable.
Moreover, $\Eff(\overline{M}_{0,1}(\mathrm{LG}(n,2n),2))$ admits a Mori chamber decomposition with exactly three chambers. Therefore $\overline{M}_{0,1}(\mathrm{LG}(n,2n),2)$ is a Mori dream space for every $n\ge 2$. The rest of the claim follows from Proposition \ref{-K}.
\end{proof}

By Corollary \ref{cor:Fano-weakFano-MDS}, we have:
\begin{corollary}\label{rem:Fano-index}
For $2\le n\le 5$, the Kontsevich space $\overline{M}_{0,1}(\mathrm{LG}(n,2n),2)$ is a Fano variety of index $1$.    
\end{corollary}
\begin{proof}
Let \(X = \overline{M}_{0,1}(\LG(n,2n),2)\) with \(2\leq n\leq 5\). 
It suffices to prove that \(-K_{X}\) is not divisible by any integer \(m\geq 2\) in \(\Pic(X)\). 

Consider the test subvarieties $\iota_{p,q}:\Bl_{\{p,q\}}(Q^{3})\hookrightarrow X$ in Lemma \ref{lem:restrictions-k1}, where  $p,q$ are two general points. 
Then, \begin{equation*}
\iota_{p,q}^*(H_1) \equiv H,\qquad\iota_{p,q}^*(H_{\sigma_2})\equiv H-E_p-E_q,\qquad\iota_{p,q}^*(T)\equiv 2(H-E_p-E_q).  \end{equation*} 
By Proposition \ref{-K},  we have that
\(
\iota_{p,q}^{*}(-K_{X}) = 5H - 4E_{p} - 4E_{q}.
\) The class \(5H - 4E_{p} - 4E_{q}\in \Pic(\Bl_{\{p,q\}}(Q^{3}))\) is primitive since \(\gcd(5,4,4) = 1\). The proof is complete. 
\end{proof}

The following is the generalization of \cite[Corollary 6.18]{CM} (the case $n=2$) to every $n\geq 2$.

\begin{proposition}\label{prop:AutKontsevichConicsLG}
Automorphisms of $\overline{M}_{0,0}\big(\LG(n,2n),2\big)$ are induced by that of the target $\LG(n,2n)$:
\[
\Aut(\overline{M}_{0,0}(\LG(n,2n),2))\cong\PSp_{2n},\,\,\forall\, n\geq2.
\]
\end{proposition}

\begin{proof}
We may assume $n\ge 3$. By \cite[Theorem 6.14]{CM}, the effective cone $\Eff(\overline{M}_{0,0}(\LG(n,2n),2))$ is generated by the two extremal rays
$\mathbb{R}_{\ge 0}\langle \Delta^n\rangle$ and $\mathbb{R}_{\ge 0}\langle D^n_{\mathrm{unb}}\rangle$,
and the nef cone $\Nef(\overline{M}_{0,0}(\LG(n,2n),2))$ is generated by 
$\mathbb{R}_{\geq 0}\langle H^n_{\sigma_2}\rangle$ and $\mathbb{R}_{\geq 0}\langle T^n\rangle$. For any $\varphi\in \Aut(\overline{M}_{0,0}(\LG(n,2n),2))$, $\varphi^*$ preserves both $\Eff(\overline{M}_{0,0}(\LG(n,2n),2))$ and $\Nef(\overline{M}_{0,0}(\LG(n,2n),2))$, hence it must fix each extremal ray. 

Consider the birational model $X_n$ associated in \cite[Theorem 6.14]{CM} to the Mori chamber delimited by
$H^n_{\sigma_2}$ and $D^n_{\mathrm{unb}}$; by functoriality of the chamber decomposition under $\varphi^*$,
the automorphism $\varphi$ induces an automorphism of $X_n$ over the symplectic Grassmannian parametrizing $(n-2)$--dimensional isotropic subspaces of a $2n$--dimensional symplectic vector space $\operatorname{SG}(n-2,2n)$. Moreover, \cite[Theorem 6.14]{CM} identifies $X_n$ as a relative Hilbert scheme $X_n:=\operatorname{Hilb}_{2t+1}(\LG(2, \mathcal Q_n))  \rightarrow \operatorname{SG}(n-2,2n)$, where $\mathcal{Q}_n$ is the rank $4$ universal quotient bundle on $\operatorname{SG}(n-2,2n)$. Since $\varphi^*[D^n_{\mathrm{unb}}]=[D^n_{\mathrm{unb}}]$, $\varphi$ descends to an automorphism
$\psi\in \Aut(\operatorname{SG}(n-2,2n))$. This defines a group homomorphism
\[
\rho \colon \Aut(\overline{M}_{0,0}(\LG(n,2n),2))\longrightarrow \Aut(\operatorname{SG}(n-2,2n)).
\]  

Take an arbitrary automorphism $\phi\in\Aut(\overline{M}_{0,0}(\LG(n,2n),2))$ such that $\rho(\phi)={\rm Id}$.
Since the fibers of $\operatorname{Hilb}_{2t+1}(\LG(2, \mathcal Q_n))\rightarrow\operatorname{SG}(n-2, 2n)$ are isomorphic to $G(3, 5)$, the induced automorphism of $X_n$ yields an automorphism of the underlying rank $4$ bundle $\mathcal{Q}_n$. By \cite{Ram66} $\mathcal{Q}_n$ is a simple homogeneous vector bundle, hence $\Aut(\mathcal{Q}_n)=\mathbb{C}^*$. We can thus conclude that $\rho$ is injective.

Finally, since the Dynkin diagram of type $C_n$ has no nontrivial automorphisms, $\Aut(\operatorname{SG}(n-2,2n))=\PSp_{2n}$ by \cite{Dem77}. On the other hand, $\PSp_{2n}=\Aut(\LG(n,2n))$ acts naturally on $\overline{M}_{0,0}(\LG(n,2n),2)$. Hence $\Aut(\overline{M}_{0,0}(\LG(n,2n),2))\cong \PSp_{2n}$.
\end{proof}

The following is the pointed analogue of Proposition~\ref{prop:AutKontsevichConicsLG}.

\begin{proposition}\label{prop:AutKontsevichConicsLGpointed}
Automorphisms of $\overline{M}_{0,1}(\LG(n,2n),2)$ are induced by that of $\LG(n,2n)$:
\[
\Aut(\overline{M}_{0,1}(\LG(n,2n),2))\cong\PSp_{2n},\,\,\forall\,n\geq 2.
\]
\end{proposition}

\begin{proof}
By Theorems \ref{thm:Eff-general-n-k1}, \ref{thm:nef-pointed-conics-LG}, any
$\varphi\in\Aut(\overline{M}_{0,1}(\LG(n,2n),2))$ preserves the extremal rays of these cones. In particular, $\varphi^{*}$ preserves the ray
$\mathbb{R}_{\ge 0}\langle H_1\rangle$, hence the complete linear system $|mH_1|$ for $m\gg 0$.
Because $H_1=\ev_1^{*}H$, the morphism defined by $|mH_1|$ is exactly $\ev_1$, so there is a unique
$\widehat\varphi\in \Aut(\LG(n,2n))$ such that $\ev_1\circ \varphi=\widehat\varphi\circ \ev_1$.
Moreover, $H_{\sigma_2}$ and $T$ on $\overline{M}_{0,1}(\LG(n,2n),2)$ are pullbacks via the forgetful morphism
$\phi_1$ from $\overline{M}_{0,0}(\LG(n,2n),2)$, so preserving the corresponding contraction yields a unique
$\bar\varphi\in \Aut(\overline{M}_{0,0}(\LG(n,2n),2))$ with $\phi_1\circ \varphi=\bar\varphi\circ \phi_1$.  Therefore, $\varphi$ induces an automorphism $\bar\varphi\in\Aut(\overline{M}_{0,0}(\LG(n,2n),2))$.

Identifying $\bar\varphi$ with an element of $\PSp_{2n}$ via Proposition~\ref{prop:AutKontsevichConicsLG}, we derive a natural homomorphism
\[
\chi\colon\Aut(\overline{M}_{0,1}(\LG(n,2n),2))\rightarrow \PSp_{2n}
\qquad
\chi(\varphi):=\bar\varphi.
\]
Moreover, 
$\varphi\in\Aut(\overline{M}_{0,1}(\LG(n,2n),2))$ and  $\bar\varphi=\chi(\varphi)$ fit into the following commutative diagrams.
\[
\begin{tikzcd}
\overline{M}_{0,1}(\LG(n,2n),2) \arrow[r,"\varphi"] \arrow[d,"\phi_1"'] &
\overline{M}_{0,1}(\LG(n,2n),2) \arrow[d,"\phi_1"] \\
\overline{M}_{0,0}(\LG(n,2n),2) \arrow[r,"\bar\varphi"'] &
\overline{M}_{0,0}(\LG(n,2n),2)
\end{tikzcd}
\,\, {\rm and}\,\,
\begin{tikzcd}
\overline{M}_{0,1}(\LG(n,2n),2) \arrow[r,"\varphi"] \arrow[d,"\ev_1"'] &
\overline{M}_{0,1}(\LG(n,2n),2) \arrow[d,"\ev_1"] \\
\LG(n,2n) \arrow[r,"\bar\varphi"'] &
\LG(n,2n).
\end{tikzcd}
\]
Here $\phi_1$ and $\ev_1$ are the forgetful and evaluation morphisms, respectively.

Now let $\varphi\in\ker(\chi)$. Then for every point $[f,x]\in \overline{M}_{0,1}(\LG(n,2n),2)$, we have
\[
\phi_1(\varphi([f,x]))=\phi_1([f,x])=[f],
\qquad
\ev_1(\varphi([f,x]))=\ev_1([f,x])=f(x).
\]
That is,  $\varphi([f,x])$ belongs to the intersection
\(
\pi^{-1}([f]) \cap \ev_1^{-1}(f(x))
\), which consists of the unique point $[f,x]$. Consequently $\varphi([f,x])=[f,x]$ for every $[f,x]$, and thus $\chi$ is an isomorphism.
\end{proof}

\section{Kausz--Type Compactifications related to Orthogonal Grassmannians}\label{og}

Based on \cite{FW,FMW}, one can extend the mechanism to orthogonal Grassmannians. Let $W$ be an $n$--dimensional complex vector space and set $V:=W\oplus W^\vee$ with its canonical quadratic form defined by $Q(w, \alpha): = \alpha(w)$.
We denote by $\OG(V,Q)$ the orthogonal Grassmannian parametrising maximal isotropic subspaces, and by $\OG^+(V,Q)$ the connected component of $\OG(V,Q)$ that contains the maximal isotropic subspace $W \oplus \{0\}\subset W\oplus W^\vee$.

The first canonical embedding of $\OG^+(V,Q)$ can be given as follows (see \cite{Knus}). Let $C(W\oplus W^\vee, Q)$ be the Clifford algebra with a linear map 
$i\colon W\oplus W^\vee\rightarrow C(W\oplus W^\vee, Q)$ such that $[i(x)]^2 = Q(x)$. The spinor representation $\rho\colon C(W\oplus W^\vee,Q)\rightarrow {\rm End}(\bigwedge W^\vee)$  defined by $\rho(w,\alpha)=l_{\alpha}+\iota_{w}$ splits into half-spinor representations $\bigwedge^{\rm even}W^\vee$, $\bigwedge^{\rm odd}W^\vee$, where 
$l_{\alpha}$ is the left multiplication
and $i_{w}$ is the derivation.
A spinor $\phi\in \bigwedge^{\rm even}W^\vee$ is pure if its annihilator $L_{\phi}:=\{x\in W\oplus W^\vee:x\cdot\phi=0\}$ is a maximal isotropic subspace. Each $L\in\OG^+(V,Q)$ corresponds to a unique pure spinor line 
$[\phi_L]$. we have the spinor embedding
\begin{equation}\label{eog}
e_{\OG^+}:\OG^+(V,Q)\xhookrightarrow{\,\,\,\,\,\,\,\,\,\,\,}\mathbb{P}\big(\bigwedge\nolimits^{\rm even}W^\vee\big).
\end{equation}

Similarly, we can define a $\mathscr{G} :=\GL(W^\vee)\times\Gm$--equivariant rational map
\begin{equation}\label{okl1}
\KO_n\colon\OG^+(V,Q)\dashrightarrow\mathbb P\big(\bigoplus\limits_{k=0}^{[n/2]}\bigwedge\limits^{2k}W^\vee\big)\times\prod\limits_{k=0}^{[n/2]}\mathbb P\big(\bigwedge\limits^{2k}W^\vee\big). 
\end{equation}
The Kausz--type compactification $\TO_n$ of the open $\mathscr{G}$--orbit in $\OG^+(n,2n)$ is defined as the closure of the graph of $\KO_n$.
By projection to the first factor, we define
\[
\RO_n \colon \TO_n \longrightarrow \OG^+(n,2n)\subset\mathbb P\big(\bigoplus\limits_{k=0}^{[n/2]}\bigwedge\limits^{2k}W^\vee\big).
\]

The \emph{space of complete skew--forms} is the closure of the graph
\begin{equation*}
\begin{split}
\mathbb P\big(\bigwedge\nolimits^2W^\vee\big)&\dashrightarrow
\mathbb P\big(\bigwedge\nolimits^2\bigwedge\nolimits^{1}W^\vee\big)\times\mathbb P\big(\bigwedge\nolimits^2\bigwedge\nolimits^{2}W^\vee\big)\times\cdots\times\mathbb P\big(\bigwedge\nolimits^2\bigwedge\nolimits^{[n/2]}W^\vee\big).   
\end{split}    
\end{equation*}
We denote this smooth projective variety by $\CS_n$. Note that the first boundary divisor $D_0^-\subset\TO_n$ is isomorphic to $\CS_n$, and so is $D_0^+$. Then by projection onto $D_0^{\pm}$, we can define a $\mathscr G$--equivariant flat morphism
\begin{equation}\label{opln}
\PL_n \colon \TO_n \longrightarrow \CS_n.   
\end{equation}

As in Proposition \ref{prop:TLn-blowup-equals-Kausz},
$\RL_n\colon\TO_n\to\OG^+(n,2n)$ can be written as an iterated blow--up as follows. 
Under a natural $\mathbb C^*$--action similar to (\ref{Gm}), the connected components of the fixed point scheme of $\OG^+(n,2n)$ are
\begin{equation*}
\begin{split}
&\underline{\mathcal V}_{(n-2k,2k)} :=\left\{\left. \left(
\begin{matrix}
0&X\\
Y&0\\
\end{matrix}\right)\in \OG^+(n,2n)\right\vert_{}\footnotesize\begin{matrix}
X\,\,{\rm is\,\,a\,\,}2k\times n\,\,{\rm matrix\,\,of\,\,rank}\,\,2k\,\\
Y\,\,{\rm is\,\,an\,\,}(n-2k)\times n\,\,{\rm matrix\,\,of\,\,rank}\,\,(n-2k)\\
\end{matrix}
\right\}\,,\,\,\,\,0\leq k\leq [n/2],\\
\end{split}
\end{equation*}
and the stable, unstable submanifolds are respectively
\begin{equation}\label{ovani}
\begin{split}
&\underline{\mathcal V}_{(n-2k,2k)}^+:= \left\{\left.\left(
\begin{matrix}
0&X\\
Y&W\\
\end{matrix}\right)\in \OG^+(n,2n)\right\vert_{}\footnotesize\begin{matrix}
X\,\,{\rm is\,\,a\,\,}2k\times n\,\,{\rm matrix \,\,of\,\,rank\,\,}2k\,\\
Y\,\,{\rm is\,\,an\,\,}(n-2k)\times n\,\,{\rm matrix \,\,of\,\,rank\,\,}(n-2k)\,\\
\end{matrix}\right\},\\
&\underline{\mathcal V}_{(n-2k,2k)}^-:= \left\{\left.\left(
\begin{matrix}
Z&X\\
Y&0\\
\end{matrix}\right)\in \OG^+(n,2n)\right\vert_{}{\footnotesize\begin{matrix}
X\,\,{\rm is\,\,a\,\,}2k\times n\,\,{\rm matrix \,\,of\,\,rank\,\,}2k\,\\
Y\,\,{\rm is\,\,an\,\,}(n-2k)\times n\,\,{\rm matrix \,\,of\,\,rank\,\,}(n-2k)\,\\
\end{matrix}}\right\}\,.   
\end{split}
\end{equation}
Define $Z^+_{n-2k}=\overline{\underline{\mathcal V}_{(n-2k,2k)}^+}$ and  $Z^-_{2k}=\overline{\underline{\mathcal V}_{(n-2k,2k)}^-}$.   
When $n$ is even, $\TO_n$ are obtained from $\OG^+(n,2n)$ by sequentially blowing up the strict transforms of $Z_0^-,Z_2^-,\dots,Z_{n-2}^-, Z_{0}^+,Z_{2}^+,\dots,Z_{n-2}^+$ (in that order); when $n$ is odd, $\TO_n$ are obtained by sequentially blowing up the strict transforms of $Z_0^-,Z_2^-,\dots,Z_{n-3}^-, Z_{1}^+,Z_{3}^+,\dots,Z_{n-2}^+$. 

The boundary divisors of $\mathcal {TO}_{n}$ are 
$\cup_{0\leq k\leq [\frac{n}{2}]-1}D_{k}^+\cup_{0\leq k\leq [\frac{n}{2}]-1}D_{k}^-$, where
\begin{equation*}
D_{k}^-:=\overline{(\RL_{n})^{-1}({\underline{\mathcal V}_{(n-2k,2k)}^-})}\,\,{\rm and} \,\,D_{k}^+:=\overline{(\RL_{n})^{-1}({\underline{\mathcal V}_{(n-2[\frac{n}{2}]+2k,2[\frac{n}{2}]-2k)}^+})}.    
\end{equation*}

We introduce the Mille Cr\^epes coordinates 
for $\TO_n$ as follows. Choose a basis $\{e_1,\cdots,e_n\}$ for $W$ and let $\{e_1^*,\cdots,e_n^*\}$ denote its dual basis in $W^\vee$. For any even integers $\ell$ with $0\leq \ell\leq n$, we have the following affine coordinate chart around the maximal isotropic subspace $L_{12\cdots \ell}$ generated by $e_{\ell+1},\cdots,e_n,e^*_{1},\cdots,e^*_{\ell}$. 
\begin{equation}\label{oull}
U^{\prime}_{12\cdots l}:=\left\lbrace\left(\begin{matrix}
  Z\\
-X^{\!\top}\\
\end{matrix}\right.
  \hspace{-.11in}\begin{matrix}
  &\hfill\tikzmark{a}\\
  &\hfill\tikzmark{b}  
  \end{matrix} \,\,\,\,\,
  \begin{matrix}
  0\\
I_{(n-l)\times(n-l)}\\
\end{matrix}\hspace{-.11in}
\begin{matrix}
  &\hfill\tikzmark{c}\\
  &\hfill\tikzmark{d}
  \end{matrix}\hspace{-.11in}\begin{matrix}
  &\hfill\tikzmark{g}\\
  &\hfill\tikzmark{h}
  \end{matrix}\,\,\,\,
\begin{matrix}
I_{l\times l}\\
0\\
\end{matrix}\hspace{-.11in}
\begin{matrix}
  &\hfill\tikzmark{e}\\
  &\hfill\tikzmark{f}\end{matrix}\hspace{0.11in}\left.\begin{matrix}
  X\\
  Y\\
\end{matrix}\right)\right\rbrace\cong\mathbb A^{\frac{n(n-1)}{2}},
  \tikz[remember picture,overlay]   \draw[dashed,dash pattern={on 4pt off 2pt}] ([xshift=0.5\tabcolsep,yshift=7pt]a.north) -- ([xshift=0.5\tabcolsep,yshift=-2pt]b.south);\tikz[remember picture,overlay]   \draw[dashed,dash pattern={on 4pt off 2pt}] ([xshift=0.5\tabcolsep,yshift=7pt]c.north) -- ([xshift=0.5\tabcolsep,yshift=-2pt]d.south);\tikz[remember picture,overlay]   \draw[dashed,dash pattern={on 4pt off 2pt}] ([xshift=0.5\tabcolsep,yshift=7pt]e.north) -- ([xshift=0.5\tabcolsep,yshift=-2pt]f.south);\tikz[remember picture,overlay]   \draw[dashed,dash pattern={on 4pt off 2pt}] ([xshift=0.5\tabcolsep,yshift=7pt]g.north) -- ([xshift=0.5\tabcolsep,yshift=-2pt]h.south);
\end{equation}
where $Y$, $Z$ are anti--symmetric. Such coordinate charts cover $\OG^+(n,2n)$ up to permutation.
We define  
\begin{equation*}
\underline{\mathbb I}_{n}^{k}:=\big\{(i_1,\cdots,i_{2k})\in\mathbb Z^{2k}: 1\leq i_1<i_2<\cdots<i_{2k}\leq n\}\,\,\,{\rm for}\,\,0\leq k\leq [n/2]. 
\end{equation*}
Denote by $\underline{\mathcal S}_k$ the ideal sheaf of $\mathcal O_{\mathrm{OG}^+(n,2n)}$ generated by $\{\left(e_{\mathrm{OG}^+}\right)^*z_{I}:{I\in \underline{\mathbb I}_{n}^k}\}$. 
Similarly, 
\begin{lemma}\label{sepao}
For any integers $k,l$ such that $0\leq k\leq [n/2]$, $0\leq l\leq n$,  the restriction of $\underline{\mathcal S}_k$ to the affine coordinate chart $U^{\prime}_{12\cdots l}$ is generated by   
\begin{equation*}
   \left\{\begin{array}{ll}
    \{\operatorname{Pf}(Z_{i_1\cdots i_{l-2k}})\}_{1\leq i_1<\cdots <i_{l-2k}\leq l},\,\,&2k<l \\
    \{\operatorname{Pf}(W_{j_1\cdots j_{2k-l}})\}_{l+1\leq j_1<\cdots <j_{2k-l}\leq n}, \,\,&2k>l\\
       1,\,\,&2k=l\\
   \end{array}\right.. 
\end{equation*}
Here $\operatorname{Pf}$ is the Pfaffian of a anti--symmetric matrix.
\end{lemma}

Then we have the open covering 
\[
\TO_n \ =\bigcup_{g \in \operatorname{Aut}(\TO_n)} \; \ \bigcup_{{\rm even}\,\,\ell=0}^{[n/2]} g(A_\ell),
\]
where 
\(
A_\ell \ \cong\ \A^{\frac{n(n-1)}{2}}
\)
with regular coordinates 
\[
\bigl\{
a_{\ell,i}^+\bigr\}_{[\frac{n}{2}]-\frac{\ell}{2}\le i\le [\frac{n}{2}]-1},
\quad
\bigl\{a_{\ell,i}^-\bigr\}_{\frac{\ell}{2}\leq i\leq [\frac{n}{2}]-1},
\quad
\bigl\{b_{\ell,1},\dots,b_{\ell,\frac{n(n-1)}{2}-[\frac{n}{2}]}\bigr\},
\]
satisfying the following properties:

\begin{itemize}
\item[-] On each $A_\ell$, for $[\frac{n}{2}]-\frac{\ell}{2}\leq i\leq [\frac{n}{2}]-1$ (resp. $\frac{\ell}{2}\leq i\leq [\frac{n}{2}]-1$), the intersection
$D_i^+\cap A_\ell$ (resp.\ $D_i^-\cap A_\ell$) is given by the vanishing
of a single coordinate $a_{\ell,i}^+$ (resp.\ $a_{\ell,i}^-$); otherwise, the intersection is empty.


\item[-] The coordinates $b_{\ell,1},\dots,b_{\ell,\frac{n(n-1)}{2}-[\frac{n}{2}]}$ are
horizontal coordinates along the fibers of (\ref{opln}).
\end{itemize}

Consequently, we obtain the following results. 

\begin{theorem}\label{othm:TLn-spherical}
The variety $\TO_n$ is smooth and projective. The group $\mathscr{G}$
acts on $\TO_n$ with a dense open orbit isomorphic to a homogeneous space $\mathscr{G}/\mathscr H$, where $\mathscr H\subset\mathscr G$ is the stabilizer of a fixed skew--form on $W^\vee$ of maximal rank. The complement of this open orbit is a divisor with $2[\frac{n}{2}]$ smooth irreducible components
\[
D_0^-,\dots,D_{[\frac{n}{2}]-1}^-,\ D_0^+,\dots,D_{[\frac{n}{2}]-1}^+.
\]
In particular, $\TO_n$ is a smooth projective spherical $\mathscr G$--variety.
\end{theorem}

\begin{theorem}\label{othm:Fano-weakFano}
The variety $\TO_n$ is weak Fano for every $n\ge 2$. Moreover, $\TO_n$ is Fano if
and only if $2\leq n\le 5$.
\end{theorem}

\begin{theorem}\label{othm:rigidity-absolute}
For every $n\ge 2$ we have
\[
H^i\bigl(\TO_n,T_{\TO_n}\bigr)=0\qquad\text{for all }i>0.
\]
In particular, $\TO_n$ is locally rigid, namely, it has no local deformations.
\end{theorem}

\begin{theorem}\label{all3}
The universal family of the Hilbert quotient $\mathbf{OG}^+(V \oplus V^*)\! /\!/\mathbb G_m$ is isomorphic to $\PL_n \colon \TO_n \rightarrow \CS_n$. In particular, all three Hilbert quotients in \cite{Th1} have smooth universal families.    
\end{theorem}

We refer to \cite{FMW} for a more general treatment. Note that our notation for boundary divisors is shifted relative to \cite{FMW}; specifically, our $D^{\pm}_{i}$ corresponds to $D^{\pm}_{i+1}$ therein.

\bibliographystyle{amsalpha}
\bibliography{Biblio}

\end{document}